\documentclass[reqno, 12pt]{amsart}

\usepackage{amsfonts, amsthm, amsmath, amssymb}

\usepackage{hyperref}
\usepackage{a4wide}

\RequirePackage{mathrsfs} \let\mathcal\mathscr

\DeclareRobustCommand{\SkipTocEntry}[5]{} 

\numberwithin{equation}{section}

\newtheorem{theorem}{Theorem}[section]
\newtheorem{lemma}[theorem]{Lemma}
\newtheorem{proposition}[theorem]{Proposition}
\newtheorem{corollary}[theorem]{Corollary}

\theoremstyle{definition}
\newtheorem*{ack}{Acknowledgements}
\newtheorem{rem}[theorem]{Remark}
\newtheorem{rems}[theorem]{Remarks}
\newtheorem*{rem*}{Remark}

\newtheorem{example}[theorem]{Example}
\newtheorem{definition}[theorem]{Definition}

\renewcommand{\d}{\,\mathrm{d}}
\newcommand{\vphi}{\varphi}
\renewcommand{\rho}{\varrho}
\newcommand{\1}{\mathbf{1}}

\newcommand{\ZZ}{\mathbb{Z}}
\newcommand{\NN}{\mathbb{N}}
\newcommand{\QQ}{\mathbb{Q}}
\newcommand{\RR}{\mathbb{R}}
\newcommand{\CC}{\mathbb{C}}

\DeclareMathOperator{\EE}{{\mathbb{E}}}
\DeclareMathOperator{\sgn}{{\mathrm{sgn}}}

\renewcommand{\leq}{\leqslant}
\renewcommand{\geq}{\geqslant}
\renewcommand{\bar}{\overline}

\newcommand{\x}{\mathbf{x}}
\renewcommand{\t}{\mathbf{t}}

\newcommand{\Mod}[1]{\;(\operatorname{mod}\,#1)}

\newcommand{\id}{\rm{id}}

\newcommand{\starsum}{\sideset{}{^{*}}{\sum}}

\DeclareMathOperator{\li}{li}
\DeclareMathOperator{\vol}{vol}
\DeclareMathOperator{\lcm}{lcm}
\DeclareMathOperator{\supp}{supp}
\newcommand{\eps}{\varepsilon}

\begin{document}

\title[Generalized Fourier coefficients of multiplicative functions]
{Generalized Fourier coefficients of\\ multiplicative functions}
\author{Lilian Matthiesen}
\address{KTH\\
Department of Mathematics\\
Lindstedtsvägen 25\\
10044 Stockholm\\
Sweden}
\email{lilian.matthiesen@math.kth.se}

\thanks{This work was partly supported through a postdoctoral fellowship of the 
Fondation Sciences Math{\'e}matiques de Paris, 
by the National Science Foundation (Grant No. DMS-1440140) while the author was in residence at MSRI in 
Spring 2017, and by the Swedish Research Council (Grant No. 2016-05198).}

\begin{abstract}
We introduce and analyse a general class of not necessarily bounded 
multiplicative functions, examples of which include the function
$n \mapsto \delta^{\omega (n)}$, where $\delta \in \RR \setminus\{0\}$ and where $\omega$ 
counts the number of distinct prime factors of $n$, as well as the function
$n \mapsto |\lambda_f(n)|$, where $\lambda_f(n)$ denotes the Fourier 
coefficients of a primitive holomorphic cusp form.

For this class of functions we show that after applying a `$W$-trick' their 
elements become orthogonal to polynomial nilsequences.
The resulting functions therefore have small uniformity 
norms of all orders by the Green--Tao--Ziegler inverse theorem, a consequence 
that will be used in a separate paper in order to asymptotically 
evaluate linear correlations of multiplicative functions from our class.
Our result generalises work of Green and Tao on the M\"obius function.

\end{abstract}

\maketitle

\tableofcontents

\section{Introduction}\label{s:introduction}
Let $f:\NN \to \CC$ be a multiplicative arithmetic function.
Daboussi showed (see Daboussi and Delange \cite{dab}) that if 
$|f|$ is bounded by $1$, then
\begin{equation}\label{eq:fourier-bound}
\frac{1}{x}\sum_{n\leq x} f(n) e^{2\pi i \alpha n} = o(x) 
\end{equation}
holds for every irrational $\alpha$.
A detailed proof of the following slightly strengthened version may be 
found in Daboussi and Delange \cite{DD82}:
Suppose that $f$ satisfies
\begin{equation}\label{eq:moment-cond}
 \sum_{n\leq x} |f(n)|^2 = O(x),
\end{equation}
then \eqref{eq:fourier-bound} holds for every irrational $\alpha$.
Montgomery and Vaughan \cite{mv77} give explicit error terms for the 
decay in \eqref{eq:fourier-bound} for multiplicative functions that 
satisfy, in addition to \eqref{eq:moment-cond}, a uniform bound at all primes, in 
the sense that $|f(p)| \leq H$ holds for some constant $H \geq 1$ and all primes 
$p$.

In this paper we will study the closely related question of bounding correlations 
of multiplicative functions with polynomial nilsequences in place of the 
exponential function $n \mapsto e^{2 \pi i \alpha n}$.
A chief concern in this work is to include unbounded multiplicative functions 
in the analysis. 
To this end we shall significantly weaken the moment condition 
\eqref{eq:moment-cond} by decomposing $f$ into a suitable Dirichlet convolution
$f=f_1 * \dots * f_t$ and analysing the correlations of the individual factors 
with exponentials, or rather nilsequences.
The benefit of such a decomposition is that we merely require control on the 
second moments of the individual factors of the Dirichlet convolution and not of 
$f$ itself. 
This essentially allows us to replace \eqref{eq:moment-cond} by 
the condition that there exists $\theta_f \in (0,1]$ such that
\begin{align} \label{eq:new-moment}
\sqrt{\frac{1}{x}\sum_{ n \leq x } |f_i(n)|^2}
\ll (\log x)^{1- \theta_f}
\frac{1}{x} \sum_{n \leq x} |f_i(n)|
\end{align}
for all $i \in \{1, \dots, t\}$.
To illustrate the difference between these two moment conditions, let us consider 
a simple example of a function that satisfies 
\eqref{eq:new-moment}, but neither \eqref{eq:moment-cond} nor
\begin{align}\label{eq:moment-2}
\sum_{n\leq x} |f(n)|^2 \ll \sum_{n\leq x} |f(n)|.
\end{align}

\begin{example} \label{ex:moments}
For any $t\in \NN$, let $d_t(n) = \1 * \dots * \1 (n)$ denote the general 
divisor function, which arises as a $t$-fold convolution of $\1$.
Choosing $f_i=\1$ for each $1\leq i \leq t$, it is clear that 
\eqref{eq:new-moment} holds with $\theta_f = 1$.
If $t>1$, then neither \eqref{eq:moment-cond} nor \eqref{eq:moment-2} hold, since 
$$
\frac{1}{x} \sum_{n \leq x} d_t(n) \asymp_t (\log x)^{t-1},
\quad \text{ but } \quad
\frac{1}{x} \sum_{n \leq x} d^2_t(n) \asymp_t (\log x)^{t^2-1}.
$$
Thus, the second moment is not controlled by the first.
\end{example}

In order to describe the three classes of multiplicative functions that we will 
be working with here, let us introduce some notation.
Throughout this paper, we write
$$
 S_f(x)
 = \frac{1}{x} \sum_{n\leq x } f(n)
\quad \text{and} \quad
 S_f(x;q,r)
 = \frac{q}{x} \sum_{\substack{n\leq x \\ x \equiv r \Mod{q}}} f(n)
$$ 
for $x \geq 1$ and integers $q, r \in \NN$.
We furthermore require the following functions $w$ and $W$:
\begin{definition}\label{d:w}
Let $w:\NN \to \RR$ be an increasing function such that 
$$\frac{\log \log x}{\log \log \log x} < w(x) \leq \log \log x$$
for all sufficiently large $x$, and set
$$
W(x) = \prod_{p\leq w(x)} p.
$$
\end{definition}

The basic class of function we will be interested in is the following:
\begin{definition} \label{def:M}
Given a positive integer $H \geq 1$, we let $\mathcal{M}_H$ 
denote the class of multiplicative arithmetic functions 
$f:\NN \to \CC$ such that: 
\begin{enumerate}
\item[(1)] $|f(p^k)| \leq H^k$ for all prime powers $p^k$.

\item[(2)] There is a positive constant $\alpha_f$ such that 
$$
\frac{1}{x} \sum_{p \leq x} |f(p)| \log p \geq \alpha_f
$$
for all sufficiently large $x$.
 \end{enumerate}
\end{definition}
For the purpose of our main result, Theorem \ref{t:non-corr}, it will be 
necessary to restrict attention to those functions $f$ that admit a so-called 
$W$-trick (see Section \ref{s:W}).
For this reason, we introduce the subset of elements of 
$\mathcal{M}_H$ that have stable mean values in certain arithmetic progressions:
\begin{definition} \label{d:F}
Let $\mathcal{F}_H \subset \mathcal{M}_H$ be the subset of multiplicative 
functions $f$ with the following property.
Let $x>1$ be a parameter.
Given any constant $C>0$, there exists a function $\varphi_C$ with 
$\varphi_C(x) \to 0$ as $x\to \infty$ such that, 
whenever $1\leq Q<(\log x)^C$ is a multiple of $W(x)$ and when $A \Mod{Q}$ is a reduced residue, then
\begin{equation} \label{eq:d-F}
S_{f}(x';Q,A)
= S_{f}(x;Q,A)
 + O\Bigg(\varphi_C(x)
 \frac{Q}{\phi(Q)} \frac{1}{\log x} 
 \prod_{\substack{p \leq x \\ p\nmid Q}} 
 \left(1 + \frac{|f(p)|}{p}\right)
 \Bigg)
\end{equation}
 for all $x' \in (x(\log x)^{-C},x)$.
\end{definition}

We will discuss this class of functions in detail in Section \ref{s:examples}, 
where we prove several sufficient conditions for $f \in \mathcal{M}_H$ to 
belong to $\mathcal{F}_H$, or to a related class that will be introduced below.
These sufficient conditions, recorded in Propositions \ref{p:sufficient-F}
and \ref{p:sufficient-2} and Lemmas \ref{l:appl-2} and \ref{l:appl-3}, prove to be 
much easier to verify in practice than the one given in the above definition, not at least 
because they take a form that allows for applications of the Selberg--Delange method as presented
in \cite{Tenen}.
As an application of Lemma \ref{l:appl-2} and \ref{l:appl-3} (see the remarks following their statements), 
we obtain the following simple criterion applicable to real-valued elements of $\mathcal{M}_H$:
\begin{proposition} \label{p:intro}
Suppose that $f \in \mathcal{M}_H$ is real-valued and that it is bounded away from zero at primes,
in the sense that there exists $\delta >0$ and a sign $\epsilon \in \{+,-\}$ such that
$$
\#\Big\{p\leq x: \epsilon f(p) \geq \delta \Big\} \geq \frac{(1+o(1))x}{\log x},
\qquad (\text{as } x \to \infty).
$$
Then $f \in \mathcal{F}_H$ if $f$ is non-negative or if, for every given $C>0$,
there exists a function $\psi_C:\RR_{\geq0} \to \RR_{\geq0}$ with $\psi_C(x) \to 0$ as $x\to \infty$ such that
$$S_{f \chi_0}(x) = O\bigg(\frac{\psi_C(x)}{\log x} \exp\Big(\sum_{p \leq x, p \nmid Q} \frac{|f(p)|}{p}\Big)\bigg),
\qquad (x>1)$$
for all trivial characters $\chi_0 \Mod{Q}$ with $Q \in (1,(\log x)^C)$ and $W(x)|Q$.
\end{proposition}
Observe, in particular, that this criterion may be applied to functions that take 
negative values at all primes, such as the M\"obius function.
In the latter case, the Prime-Number-Theorem-type estimate 
$S_{\mu}(x) \ll_B (\log x)^{-B}$, which holds for all $x\geq 2$ and $B>0$, implies that 
all conditions are satisfied; cf.\ Example \ref{ex:mobius-4.4}(i) for details.
As an easy consequence of the above proposition, it further follows that any function of the form 
$f(n)=\delta^{\omega(n)}$ for fixed $\delta >0 $ belongs to $\mathcal{F}_H$.
In Section \ref{ss:non-negative} we will show that the function
$n \mapsto |\lambda_f(n)|$ belongs to $\mathcal{F}_H$, where $\lambda_f(n)$ 
denotes the normalised Fourier coefficients of a primitive holomorphic cusp 
form.
This is an example which cannot be deduced from the above proposition.

In Section \ref{s:non-corr}, we will see that in the context of our main result
condition \eqref{eq:d-F} only needs to hold for slowly varying twists of $f$.
This allows us to slightly weaken the above definition 
and introduce the following intermediate class of functions
$\mathcal{F}_H \subset \mathcal{F}_{H,n^{it}} \subset \mathcal{M}_H$,
which will also be discussed in Section \ref{s:examples}.

\begin{definition} \label{def:F_H(x)}
Let $\mathcal{F}_{H,n^{it}} \subset \mathcal{M}_H$ denote the subset of 
functions $f$ with the following property. 
For every constant $C>0$ and every sufficiently large $x > 1$, there exists 
$t_x \in \RR$ with $|t_x| \leq 2 \log x$ such that the function 
$f_x: n \mapsto f(n)n^{-it_x}$ satisfies \eqref{eq:d-F} for all
$x' \in (x(\log x)^{-C},x)$, all $1\leq Q<(\log x)^C$, $W(x)|Q$, and all 
reduced residues $A \Mod{Q}$.
Observe that $\mathcal{F}_H \subset \mathcal{F}_{H,n^{it}}$ 
since we may take $t_x=0$ for all $x$.
\end{definition} 

Twists of the form $f(n)n^{-it}$ play an important role in the study of 
multiplicative functions as their behaviour is closely linked to that of the 
mean value of $f$ through Hal\'asz's theorem \cite{Halasz}, see also \cite[\S 
III.4.3]{Tenen}.
While Hal\'asz's theorem concerns bounded functions that are closely related to 
the constant function $\1$, an analogon to this result, applicable to our basic 
class $\mathcal{M}_H$, has recently been proved independently by Elliott
\cite[Theorems 2 and 4]{elliott} and Tenenbaum 
\cite[Th\'eor\`eme 1.2]{tenen}.
The following lemma, which we chiefly include for comparison of the error terms 
in \eqref{eq:d-F} and in later results,
is a straightforward consequence of their result.
The first part is due to Elliott and Kish \cite[Lemma 21]{elliott-kish}.
\begin{lemma}[Elliott--Kish, Elliott, Tenenbaum] \label{l:elliott}
Suppose $f \in \mathcal{M}_H$ and that 
$$\sum_{p \leq H} \sum_{k \geq 2} |f(p^k)|p^{-k} < \infty.$$
Then
$$
S_{|f|}(x) 
\gg \frac{1}{\log x} \exp\bigg(\sum_{p \leq x} \frac{|f(p)|}{p}\bigg).
$$
Furthermore, we have $|S_{f}(x)| =o(S_{|f|}(x))$ unless there exists $t \in \RR$ such that 
\begin{equation*}
\sum_{p \text{ prime}} \frac{|f(p)| - \Re (f(p)p^{it})}{p}< \infty, 
\end{equation*}
in which case $|S_{f}(x)| \asymp S_{|f|}(x)$.
\end{lemma}

Returning to the basic class $\mathcal{M}_H$, let us record the lemma that
shows that every element of $\mathcal{M}_H$ does indeed admit a Dirichlet 
decomposition with the properties described at the beginning of this 
introduction.
To be precise, the lemma below corresponds to 
$\theta_f = \frac{1}{2}$ in \eqref{eq:new-moment}.
We will prove this lemma in Section \ref{s:decomposition}.
In accordance with the earlier discussion, this lemma will only be needed in the 
case where $f$ is unbounded, i.e. when $H>1$.

\begin{lemma}(Dirichlet decomposition)\label{l:dirichlet}
Let $f\in \mathcal{M}_H$ and let $h$ be the multiplicative function defined as
\begin{equation}\label{eq:h(p)/H}
h(p^k) 
=
\begin{cases}
f(p)/H & \mbox{if $k=1$}\\
0      & \mbox{if $k>1$}
\end{cases}. 
\end{equation}
Let $h^{*H}$ denote the $H$-fold convolution of $h$ with itself. 
Then $$f=h^{*H}*h',$$ where $h'$ is a multiplicative function that satisfies 
$h'(p)=0$ at primes and $|h'(p^k)|\leq (2H)^k$ at prime powers.

Let $f=f_1 * \dots *f_H$ with $f_i=h$ for all but one of the factors and 
$f_i=h*h'$ for the remaining one.
If $x>1$ and if $Q\leq x^{1/2}$ is an integer multiple of $W(x)$, then 
the following bound holds for all $A \in (\ZZ/Q\ZZ)^*$:
\begin{align}\label{eq:expectation-condition*}
 \nonumber
&\sum_{\substack{D \leq x^{1-1/H}\\\gcd(D,Q)=1}}
\sum_{d_1 \dots d_{H-1} = D}
\frac{|f_1(d_1) \dots f_{H-1}(d_{H-1})|}{D}
\sqrt{\frac{D Q}{x}\sum_{\substack{n \leq x/D \\ n D \equiv A 
\Mod{Q}}} |f_H(n)|^2}\\
&\ll (\log x)^{1/2} 
\frac{Q}{\phi(Q)} \frac{1}{\log x} 
\prod_{\substack{p\leq x \\ p \nmid Q}}
\left(1 + \frac{|f(p)|}{p} \right).
\end{align}
\end{lemma}

\addtocontents{toc}{\SkipTocEntry}
\subsection*{Aim and motivation}
As mentioned before, the purpose of this paper is to study correlations of 
multiplicative functions, more specifically of functions from 
$\mathcal{M}_H$, with polynomial nilsequences.
In general, such correlations can only shown to be small if either the nilsequence 
is highly equidistributed or else if the multiplicative function is 
equidistributed in progressions with short common difference.
We will consider both cases, the former in Proposition \ref{p:equid-non-corr} and 
the latter in Theorem \ref{t:non-corr}.
In accordance with this restriction, the latter result only 
applies to the subsets $\mathcal{F}_H$ and $\mathcal{F}_{H, n^{it}}$ whose 
elements admit a $W$-trick as we will establish in Section \ref{s:W}.
Restricting attention to the class $\mathcal{F}_H$ for now, then `$W$-trick'
roughly means the following here. 
For every $f \in \mathcal{F}_H$ there is a product 
$\widetilde W = \widetilde W(x)$ of small prime powers such 
that $f$ has a constant average value in all suitable subprogressions of 
$\{n \equiv A \Mod{\widetilde W}\}$ for every fixed residue 
$A \in (\ZZ/\widetilde W \ZZ)^*$.
Instead of bounding Fourier coefficients of $f$ as in \eqref{eq:fourier-bound}, we 
aim to show that every $f \in \mathcal{F}_H$ satisfies\footnote{This statement 
needs to be slightly adapted if $f \in \mathcal{F}_{H, n^{it}}$.}
\begin{align}\label{eq:general-fourier}
\frac{\widetilde{W}}{x} 
&\sum_{n \leq x/\widetilde{W}} 
\Big(f(\widetilde{W}n+A) - S_f(x;\widetilde{W},A)\Big) 
F(g(n)\Gamma)\\ 
&= 
\nonumber
o_{G/\Gamma}\bigg(
\frac{1}{\log x} \frac{\widetilde{W}}{\phi(\widetilde{W})} 
\prod_{\substack{p \leq x,~ p \nmid \widetilde W}}
\left(1+ \frac{|f(p)|}{p}\right)
\bigg)
\end{align}
for all $1$-bounded polynomial nilsequences $F(g(n)\Gamma)$ of bounded degree 
and bounded Lipschitz constant that are defined with respect to a 
nilmanifold $G/\Gamma$ of bounded step and bounded dimension.
The precise statement will be given in Section \ref{s:non-corr}.
This result can be viewed as a generalisation of work of 
Green and Tao \cite{GT-nilmobius} who were the first to study correlations of 
the form \eqref{eq:general-fourier} and who prove \eqref{eq:general-fourier} for 
the M\"obius function.
In fact, we borrow the approach from \cite{GT-nilmobius} to reduce Theorem 
\ref{t:non-corr} to Proposition \ref{p:equid-non-corr} in Sections 
\ref{s:non-corr} and we work with techniques from \cite{GT-nilmobius} in Sections 
\ref{s:linear-subsecs} and \ref{s:products}.

Note carefully, that the bound proposed in \eqref{eq:general-fourier} is non-trivial even
in the case where the function $f$ satisfies $S_f(x)=o(1)$, i.e.\ even for a function like
$f(n) = \delta^{\omega(n)}$ with $\delta \in (0,1)$, 
which satisfies 
$$
S_f(x) \sim (\log x)^{\delta - 1} \asymp 
\frac{1}{\log x} \prod_{p \leq x} \left(1+ \frac{|f(p)|}{p}\right) = o(1).
$$
To see this, we observe that Lemma \ref{l:elliott} and Shiu's Lemma \cite[Theorem 1]{shiu}
imply that the error term in \eqref{eq:general-fourier} is, at least for a positive proportion of the 
reduced residues $A \Mod{\widetilde{W}}$, of the form 
$o(\text{`the trivial upper bound'})$, which is the bound obtained by inserting absolute values everywhere.

The interest in estimates of the form \eqref{eq:general-fourier} lies in 
the fact that the Green--Tao--Ziegler inverse theorem \cite{GTZ} allows 
one to deduce that $f(\widetilde Wn+A) - S_f(x;\widetilde W,A)$ has small 
$U^k$-norms of all orders, where `small' may depend on $k$.
Employing the nilpotent Hardy--Littlewood method of Green and Tao 
\cite{GT-linearprimes}, this in turn allows one to deduce asymptotic 
formulae for expressions of the form
\begin{align} \label{eq:star}
\sum_{\x \in K \cap \ZZ^s} f(\vphi_1(\x)+a_1) \dots f(\vphi_r(\x)+a_r),
\end{align}
where $a_1, \dots, a_r \in \ZZ$, where 
$\vphi_1, \dots, \vphi_r : \ZZ^s \to \ZZ$ are pairwise non-proportional 
linear forms, and where $K \subset \RR^s$ is convex, 
\emph{provided} that $f$ has a sufficiently pseudo-random majorant function.
We construct such pseudo-random majorants in the companion paper \cite{lmm2}, 
which also addresses the question of evaluating \eqref{eq:star} for
functions $f \in \mathcal{F}_{H, n^{it}}$ with the property that 
$|f(n)| \ll_{\eps} n^{\eps}$ for all $\eps > 0$.

\addtocontents{toc}{\SkipTocEntry}
\subsection*{Strategy and related work}
Our overall strategy is to decompose the given multiplicative function via 
Dirichlet decomposition in such a way that we can employ the Montgomery--Vaughan
approach to the individual factors.
This approach reduces matters to bounding correlations of sequences defined in 
terms of primes.
One type of correlation that appears will be handled with the help of Green and 
Tao's bound \cite[Prop.\ 10.2]{GT-linearprimes} on the correlation of the 
`$W$-tricked von Mangoldt function' with nilsequences.
Carrying out the Montgomery--Vaughan approach in the nilsequences setting makes 
it necessary to understand the equidistribution properties of certain families of 
product nilsequences which result from an application of the Cauchy--Schwarz 
inequality.
These product sequences are studied in Section \ref{s:products} refining 
techniques introduced in \cite{GT-nilmobius}.
More precisely, we show that most of these products are equidistributed 
provided the original sequence that these products are derived from was 
equidistributed.
The latter can be achieved by the Green--Tao factorisation theorem 
for nilsequences from \cite{GT-polyorbits}. 

The question studied in this paper is in spirit related to that of
Bourgain--Sarnak--Ziegler~\cite{BSZ}, who use an orthogonality criterion that can 
be proved employing ideas that go back to Daboussi--Delange \cite{dab} 
(cf.\ Harper \cite{Harper} and Tao \cite{Tao}).
Invoking the orthogonality criterion in the form it is presented in K\'atai 
\cite{katai}, recent and very substantial work of Frantzikinakis and Host 
\cite{FH} shows that 
every bounded multiplicative function can be decomposed into the sum of a 
Gowers-uniform function, a structured part and an error term.
This error term is small in the sense that the integral of the error 
term over the space of all $1$-bounded multiplicative functions is small.
While their result provides no information on the quality of the error term of 
individual functions, it allows one to study simultaneously all bounded 
multiplicative functions.

The point of view taken in the present work is a different one: we have 
applications to explicit multiplicative functions in mind. 
For many multiplicative functions $f$ that appear naturally in number theoretic 
contexts, the mean value $\frac{1}{x}\sum_{n \leq x} f(x)$ is described by a 
reasonably nice function in $x$, and one can hope to be able to verify the 
conditions from Definitions \ref{def:M} and \ref{d:F} (or \ref{def:F_H(x)}) for 
such functions.
In order to deduce asymptotic formulae for expressions as in \eqref{eq:star}, it 
is important that the bound on the correlation \eqref{eq:general-fourier} 
improves at least on the trivial bound given by the average value of $|f|$.
Thus, we need to be able to understand these bounds for individual functions $f$.
We establish a non-correlation result (Theorem \ref{t:non-corr}) with an
explicit bound that preserves information on $|f|$ just as in  
\eqref{eq:general-fourier}.
An important feature of this work is that it applies to a large class of 
unbounded functions.

\subsection*{Notation}
\addtocontents{toc}{\SkipTocEntry}
The following, perhaps unusual, piece of notation will be used throughout the
paper: Suppose $\delta \in (0,1)$, then we write $x=\delta^{-O(1)}$ instead of
$x=(1/\delta)^{O(1)}$ to indicate that there is a constant $0 \leq C \ll 1$
such that $x = (1/\delta)^{C}$. 

\subsection*{Convention}
\addtocontents{toc}{\SkipTocEntry}
If the statement of a result contains Vinogradov or $O$-notation in the 
assumptions, then the implied constants in the conclusion may depend on all 
implied constants from the assumptions.

\begin{ack}
 I would like to thank Hedi Daboussi, R\'egis de la Bret\`eche, Nikos 
Frantzikinakis, Andrew Granville, Ben Green, Adam Harper, 
Dimitris Koukoulopoulos and Terence Tao for very helpful discussions and 
Nikos Frantzikinakis for valuable comments 
on a previous version of this paper.
I am very grateful to the anonymous referee for extremely helpful and 
detailed comments, which led in particular to the development of Section 
\ref{s:examples}.
\end{ack}

\section{Brief outline of some ideas}\label{s:outline}
In this section we give a very rough outline of the ideas behind the application 
of the Montgomery--Vaughan approach in the nilsequences setting, making a number 
of simplifications for the benefit of the exposition.  
The main idea of Montgomery--Vaughan \cite{mv77} is to introduce a $\log$ factor 
into the Fourier coefficient that we wish to analyse. 
Let $f:\NN \to \RR$ be a multiplicative function that satisfies 
$|f(p)| \leq H$ for some constant $H \geq 1$ and all primes $p$ and suppose
\eqref{eq:moment-2} holds.
Then we have
\begin{align*}
\sum_{n \leq N} f(n) e(n \alpha) \log(N/n)
\leq
\Big( \sum_{n\leq N} (\log (N/n))^2
\Big)^{1/2}
\Big( \sum_{n\leq N} |f(n)|^2 \Big)^{1/2}
\ll N^{1/2} \Big( \sum_{n\leq N} |f(n)|^2 \Big)^{1/2},
\end{align*}
and thus
\begin{align*}
\log N \Big(\frac{1}{N}&\sum_{n \leq N} f(n) e(n \alpha) \Big)  
\ll \Big(\frac{1}{N} \sum_{n\leq N} |f(n)|^2 \Big)^{1/2} 
 + \Big|\frac{1}{N} \sum_{n \leq N} f(n) e(n \alpha) \log n  \Big|.
\end{align*}
The first term in the bound is handled by the assumptions on $f$, that is, by 
assuming that \eqref{eq:moment-2} holds.
To bound the second term, one invokes the identity 
$\log n = \sum_{d|n} \Lambda(d)$, which reduces the task to bounding the 
expression
$$
\sum_{nm \leq N} f(nm) \Lambda(m) e(nm \alpha).
$$
This in turn may be reduced to the task of bounding
$$
\sum_{np \leq N} f(n)f(p) \Lambda(p) e(pn \alpha), 
$$
where $p$ runs over primes.
Applying the Cauchy-Schwarz inequality and smoothing, it furthermore suffices to 
estimate expressions of the form
$$
\sum_{p,p'} f(p)f(p') \log(p) \log(p') \sum_{n} w(n) e((p-p')n \alpha), 
$$
where $p$ and $p'$ run over primes and where $w$ is a smooth weight function.
One employs a standard sieve estimate to bound $\#\{(p,p'):p-p'=h\}$ 
for fixed $h$. Standard exponential sum estimates and a delicate decomposition of 
the summation ranges for $n, p, p'$ yield an explicit bound on 
$\frac{1}{N}\sum_{n \leq N} f(n) e(n \alpha)$. 

We seek to employ the above approach to correlations of the 
form
$$\frac{1}{N} 
\sum_{n \leq N}
\Big(f(n) - \frac{1}{N}\sum_{m\leq N}f(m)\Big) 
F(g(n)\Gamma)$$
for multiplicative $f$.
One problem we face is that the above approach makes substantial use of the strong
equidistribution properties of the exponential functions $e((p-p') n \alpha)$ 
for distinct primes $p,p'$.
A general polynomial sequence $(g(n)\Gamma)_{n \leq N}$ on a nilmanifold 
$G/\Gamma$ may, on the other hand, not even be equidistributed.
This problem is resolved by an application of the factorisation theorem for 
polynomial sequences from Green--Tao \cite{GT-polyorbits}, which allows us 
to assume that $(g(n)\Gamma)_{n \leq N}$ is equidistributed in $G/\Gamma$ if $f$ 
is equidistributed in progressions to small moduli.
The latter will be arranged for by employing a $W$-trick.
As above, we then consider the following expression which we split into the case 
of large, resp.\ small primes with respect to a suitable cut-off parameter $X$: 
\begin{align*}
\frac{1}{N} \sum_{mp \leq N} f(m) f(p) \Lambda(p) F(g(mp) \Gamma) 
& = \frac{1}{N} \sum_{m \leq X} \sum_{p \leq N/m} 
f(m) f(p) \Lambda(p) F(g(mp)\Gamma) \\
&\quad + \frac{1}{N} \sum_{m > X} \sum_{p \leq N/m} 
f(m) f(p) \Lambda(p) F(g(mp)\Gamma).
\end{align*}
Applying Cauchy-Schwarz to both terms shows that it suffices to 
understand correlations of the form
\begin{align*}
\sum_{m,m'} f(m) f(m') 
\sum_{p} \Lambda(p) F(g(mp) \Gamma) \overline{F(g(m'p) \Gamma)}
\end{align*}
and
\begin{align*}
\sum_{p,p'} f(p) f(p') \Lambda(p) \Lambda(p')
\sum_{m} F(g(pm) \Gamma) \overline{F(g(p'm) \Gamma)}.
\end{align*}
Choosing $X$ suitably, only the first of these correlations matters.
We shall bound this correlation by employing Green and Tao's result that the 
$W$-tricked von Mangoldt function is orthogonal to nilsequences. 
The necessary equidistribution properties of the sequences
$n \mapsto F(g(mn) \Gamma) \overline{F(g(m'n) \Gamma)}$ will be established in 
Sections \ref{s:linear-subsecs} and \ref{s:products}.
The problem of extending the above method to functions from $\mathcal{M}_H$ will 
be addressed at the beginning of Section~\ref{s:proof-of-prop}. 
For this purpose the moment condition \eqref{eq:moment-2} will be replaced by 
Lemma \ref{l:dirichlet}.

\section{A suitable Dirichlet decomposition for $f \in \mathcal{M}_h$}
\label{s:decomposition}
In this section we prove Lemma \ref{l:dirichlet}, which shows that 
every function $f\in \mathcal{M}_H$ has a decomposition $f=f_1* \dots * f_H$ 
into multiplicative functions $f_i$ such that the $L^2$-norms of 
the $f_i$ are controlled on average by the mean value of $f$.
This lemma will replace the much more restrictive condition \eqref{eq:moment-2} 
in our application of the Montgomery--Vaughan approach outlined in the previous 
section.
Before we prove Lemma \ref{l:dirichlet}, let us record a straightforward 
consequence of Shiu \cite[Theorem~1]{shiu} that will be used.
\begin{lemma}[Shiu] \label{l:shiu}
 Let $H$ be a positive integer and suppose 
$f: \NN \to \RR$ is a non-negative multiplicative function 
satisfying $f(p^k) \leq H^k$ at all prime powers $p^k$.
Let $W=W(x)$ be as before, let $q>0$ be an integer and let 
$A' \in (\ZZ/Wq \ZZ)^*$.
Then,
\begin{align}\label{eq:shiu}
\sum_{\substack{x-y < n \leq x \\ n \equiv A' \Mod{Wq}}} f(n)
\ll \frac{y}{\phi(Wq)} \frac{1}{\log x}
\exp\bigg(\sum_{\substack{w(x) < p \leq x \\ p \nmid q}} \frac{f(p)}{p}\bigg)~,
\end{align}
uniformly in $A'$, $q$ and $y$, provided that $q\leq y^{1/2}$ 
and $x^{1/2} \leq y \leq x$.
\end{lemma}

\begin{proof}
 This lemma differs from \cite[Theorem~1]{shiu} in that it does not 
concern short intervals but at the same time it does not require $f$ to 
satisfy $f(n) \ll_{\eps} n^{\eps}$.
Shiu's result works with a summation range of the form
$x-y < n \leq x$, where $x^{\beta} < y \leq x$, $\beta \in (0,\frac{1}{2})$.
Thus, in our case the parameter $\beta$ can be regarded as fixed.
As observed in \cite{nair-tenenbaum}, the proof of \cite[Theorem~1]{shiu}
only requires the condition $f(n) \ll_{\eps} n^{\eps}$ to hold for one fixed 
value of $\eps$ once $\beta$ is fixed.

Note that any integer $n \equiv A' \Mod{Wq}$ is free from prime divisors 
$p < w(x)$. 
Thus, $f(n) \leq H^{\Omega(n)} \leq n^{\log H / \log w(x)}$.
Given any $\eps > 0$, we deduce that $n \equiv A' \Mod{Wq}$ implies
$f(n) \leq n^{\eps}$ provided $x$ is sufficiently large.
\end{proof}

\begin{proof}[Proof of Lemma \ref{l:dirichlet}]
Let $h$ and $h'$ be as in the statement of the lemma. 
We begin by showing that 
\begin{equation} \label{eq:h'-bound}
 |h'(p^k)| \leq (2H)^k,
\end{equation}
using induction.
Since $h'(p)=0$, the inequality holds for $k=1$.  
To analyse the general case, note that, since $h(p^k)=0$ whenever 
$k\geq 2$, we have
$$h^{*H}(p^k) = \binom{H}{k} h^k(p) = \binom{H}{k} \frac{f^k(p)}{H^k} $$
for $1\leq k \leq H$, and $h^{*H}(p^k) = 0$ if $k> H$.
Thus, $f=h'*h^{*H}$ implies that
\begin{align*}
h'(p^k) &= f(p^k) - \sum_{j=1}^{\min(k,H)} h'(p^{k-j}) h^{*H}(p^j).
\end{align*}
Suppose now that $k \geq 2$ and that the inequality holds for all $j < k$.
Then, invoking also (1) of Definition \ref{def:M}, we have
\begin{align*}
|h'(p^k)| < H^k + \sum_{j=1}^{\min(k,H)} (2H)^{k-j} \binom{H}{j} 
<(2H)^k \bigg(2^{-k} + \sum_{j=1}^{\min(k,H)} \frac{1}{ j! 2^j} \bigg)
<(2H)^k,
\end{align*}
as claimed.

To prove \eqref{eq:expectation-condition*}, suppose that $f_H = h$ or $h*h'$. 
By Shiu's bound, we have
\begin{align*}
\frac{DQ}{x}
\sum_{\substack{n \leq x/D \\ n \equiv A \Mod{Q}}} 
 f_H^2(n)
 \ll \frac{1}{\log (x/D)}\frac{Q}{\phi(Q)}
 \prod_{\substack{p\leq x\\ p \nmid Q}}
 \left(1 + \frac{|f(p)|}{Hp}\right),
\end{align*}
where we used the trivial inequality $f_H(p)^2 \leq |f_H(p)|=|f(p)|/H$ and 
extended the product over primes up to $x$.
Multiplying the right hand side with
$$
\frac{Q}{\phi(Q)}
 \prod_{\substack{p\leq x\\ p \nmid Q}}
 \left(1 + \frac{|f(p)|}{Hp}\right) \gg 1,
$$
and observing that $\log(x/D) \asymp_H \log x$, we obtain
\begin{align*}
\sqrt{\frac{DQ}{x}
\sum_{\substack{n \leq x/D \\ n \equiv A \Mod{Q}}} 
 f_i^2(n)}
 \ll_H
\frac{1}{(\log x)^{1/2}}\frac{Q}{\phi(Q)}
 \prod_{\substack{p\leq x\\ p \nmid Q}}
 \left(1 + \frac{|f(p)|}{Hp}\right).
\end{align*}
Thus, the left hand side of \eqref{eq:expectation-condition*} is 
bounded by
\begin{align*}
&\ll_H \frac{1}{(\log x)^{1/2}} \frac{Q}{\phi(Q)}
 \prod_{\substack{p\leq x\\ p \nmid Q}}
 \left(1 + \frac{|f(p)|}{Hp}\right)
\sum_{\substack{D \leq x^{1-1/H}\\\gcd(D,Q)=1}}
\sum_{d_1 \dots d_{H-1} = D}
\frac{|f_1(d_1) \dots f_{H-1}(d_{H-1})|}{D}
\\
&\ll_H \frac{1}{(\log x)^{1/2}} \frac{Q}{\phi(Q)}
 \prod_{\substack{p\leq x\\ p \nmid Q}}
 \left(1 + \frac{|f(p)|}{Hp}\right)
 \left(1 + \frac{(H-1)|f(p)|}{Hp}\right) 
 \left(1 + \sum_{k\geq 2} \frac{|f_1 * \dots *f_{H-1}(p^k)|}{p^k}\right)
\\
&\ll_H \frac{1}{(\log x)^{1/2}} \frac{Q}{\phi(Q)}
 \prod_{\substack{p\leq x\\ p \nmid Q}}
 \left(1 + \frac{|f(p)|}{p}\right)
 \left(1 + \frac{(H-1)}{p^2}\right) 
 \left(1 + \sum_{k\geq 2} \frac{|f_1 * \dots *f_{H-1}(p^k)|}{p^k}\right).
\end{align*}
The above is now seen to have the claimed bound
$$
\ll_H \frac{1}{(\log x)^{1/2}} \frac{Q}{\phi(Q)}
 \prod_{\substack{p\leq x\\ p \nmid Q}}
 \left(1 + \frac{|f(p)|}{p}\right)
$$
for all sufficiently large $x$, providing
$$
\sum_{w(x)<p\leq x} \sum_{k\geq 2} \frac{|f_1 * \dots *f_{H-1}(p^k)|}{p^k}
\ll_H 1.
$$
To show the latter, note that $f_1 * \dots *f_{H-1}$ equals either $h^{*(H-1)}$ 
or $h^{*(H-1)}*h'$.
Similarly as in the first part of this proof, we have
$$|h^{*(H-1)}(p^k)| \leq \binom{H-1}{k} \leq H^k/k!$$ for $k<H$ and 
$h^{*(H-1)}(p^k) = 0$ for $k\geq H$, and, consequently,
\begin{align*}
|(h^{*(H-1)}*h')(p^k)| 
&= \sum_{j=0}^{\min(k,H-1)} |h'(p^{k-j}) h^{*(H-1)}(p^j)|
\leq \sum_{j=0}^{\min(k,H-1)} (2H)^{k-j} H^{j}
\leq 2(2H)^k
\end{align*}
Thus, if $x$ is sufficiently large so that $w(x)>4H$, then
\begin{align*}
 \sum_{w(x)<p\leq x} \sum_{k\geq 2} \frac{|f_1 * \dots *f_{H-1}(p^k)|}{p^k}
 &\leq 2 \sum_{w(x)<p\leq x} \sum_{k\geq 2} \Big(\frac{2H}{p}\Big)^k 
 \leq 8H^2 \sum_{w(x)<p\leq x} \frac{1}{p^2} \Big(1-\frac{2H}{p}\Big)^{-1} \\
 & \leq 16H^2 \sum_{w(x)<p\leq x} \frac{1}{p^2} \ll_H 1,
\end{align*}
which completes the proof.
\end{proof}

\section{Multiplicative functions in progressions: The class of functions 
$\mathcal{F}_H$}
\label{s:examples}
Both of the conditions that define $\mathcal{M}_H$ are natural and simple
conditions on the behaviour of $|f|$ at prime powers.
Our aim in this section is to discuss the more complicated stability condition 
\eqref{eq:d-F} on mean values in progressions, that defines the class 
$\mathcal{F}_H$.
We will prove two sufficient conditions, recorded in Propositions 
\ref{p:sufficient-F} and 
\ref{p:sufficient-2}, for the bound \eqref{eq:d-F} to hold and apply these 
to provide several examples of natural functions that 
belong to $\mathcal{F}_H$.
In particular, we deduce a simple criterion, see Lemma \ref{l:appl-2}, 
for a non-negative function to belong to $\mathcal{F}_H$.

\subsection{A sufficient condition for $f \in \mathcal{F}_H$}
The main tool in our analysis of \eqref{eq:d-F} will be the following 
consequence of the `pretentious large sieve', which allows one to bound 
the tail of the character sum expansion of $S_h(x;q,a)$ for any bounded 
multiplicative function $h$ and thereby simplifies the task of analysing the 
expression $S_h(x;q,a)$.
\begin{proposition}[Granville and Soundararajan \cite{GS-book}; cf.\ 
Lemma 3.1 in \cite{BGS}, Theorem 2 in \cite{G:PLS} and Theorem 1.8 in \cite{GHS}] 
\label{p:GS}
Let $C>0$ be fixed and let $h$ be a bounded multiplicative function.
For any given $x$, consider the set of primitive characters of conductor at 
most 
$(\log x)^C$ and enumerate them as $\chi_1, \chi_2, \dots$ in such a way that 
$|S_{h \bar\chi_1}(x)| \geq |S_{h \bar\chi_2}(x)| \geq \dots$.
If $x$ is sufficiently large, then the following holds for all
$x^{1/2} \leq X \leq x$ and $q \leq (\log x)^C$.
Let $\mathcal{C}$ be any set of characters modulo $q$, $q \leq (\log x)^C$, 
which 
does not contain characters induced by $\chi_1, \dots , \chi_k$, where $k 
\geq 2$. 
Then
\begin{align*}
&\left|
\frac{1}{\phi(q)} 
\sum_{\chi \in \mathcal{C}} \chi(a) 
\sum_{n \leq X} h(n) \bar\chi(n)
\right|\\
&\ll_C \frac{e^{O_C(\sqrt{k})}X}{q} 
\left(\frac{\log \log x}{\log x}\right)^{1-\frac{1}{\sqrt{k}}}
\log \left(\frac{\log x}{\log \log x}\right)
\prod_{p\leq q, p\nmid q}
\left(1+\frac{|h(p)|-1}{p}\right).
\end{align*}
\end{proposition}

In order to deduce a sufficient condition for \eqref{eq:d-F} we begin by 
extending the above result to all unbounded elements of $\mathcal{M}_H$.

\begin{corollary} \label{c:GS}
 Let $f \in \mathcal{M}_H$ and set $h=f$ if $H=1$. If $H>1$, let $h$ be the 
multiplicative function defined in \eqref{eq:h(p)/H} so that 
$f= h^{*H}*h'$ for a multiplicative function $h'$ with support in the 
square-full numbers.
Let $C>0$ be a constant, let $\eps = \frac{1}{2}\min(1,\alpha_f/H)$, and set 
$k=\lceil \eps^{-2}\rceil \geq 2$ and $k' = \lceil \log_2(4H) \rceil$.
For each $j\in \{0, \dots, k'\}$, let 
$\mathcal{E}_j = \{\chi_1^{(j)}, \dots, \chi_k^{(j)} \}$ denote the set 
consisting of the first $k$ primitive characters of conductor at most 
$(\log x^{1/2^j})^C$ that are defined by Proposition \ref{p:GS} when applied to 
$h$ and with $x$ replaced by $x^{1/2^j}$.

If $x$ is sufficiently large, then the following holds for all
$x^{1/2} \leq y \leq x$ and all integer multiplies 
$0 < Q \leq (\log x^{1/(8H)})^C$ of $W(x)$.
Let $\mathcal{C}$ be any set of characters modulo $Q$
which does not contain characters induced by any 
$\chi \in \mathcal{E}:= \mathcal{E}_0 \cup \dots \cup \mathcal{E}_{k'}$, then
\begin{align*}
&\bigg|S_f(y;Q,a) - 
\frac{Q}{y}
\frac{1}{\phi(Q)} 
\sum_{\substack{\chi \Mod{Q} \\ \chi \not\in \mathcal{C}}} \chi(a) 
\sum_{n \leq y} f(n) \bar\chi(n)
\bigg| = \\
&\frac{Q}{\phi(Q)} \left| 
\sum_{\chi \in \mathcal{C}} \chi(a) 
\frac{1}{y}
\sum_{n \leq y} f(n) \bar\chi(n)
\right| 
\ll_{C,H,\alpha_f} 
\frac{1}{(\log x)^{1+\alpha_f/(3H)}}
\frac{Q}{\phi(Q)}
\exp \bigg( \sum_{p\leq x, \,p\nmid Q} \frac{|f(p)|}{p}\bigg).
\end{align*}
\end{corollary}

The proof of this corollary makes use of the following 
lemma about the contribution of the sparse function $h'$.

\begin{lemma} \label{l:h'-contribution}
 Let $H>1$, $f\in \mathcal{M}_H$ and let $f=h^{*H}*h'$ be the decomposition 
from \eqref{eq:h(p)/H}. Let $g:\NN \to \CC$ be a bounded completely 
multiplicative function that vanishes at all primes 
$p \leq w$ for a fixed $w \geq (2H)^{16}$, 
and let $\delta \in (0,1)$.
Then, if $x^{1/2} \leq y \leq x$, we have
\begin{align*}
\left|\frac{1}{y} 
\sum_{n\leq y} f(n) b(n) \right|
&\leq
\sum_{n_1 \leq y^{\delta}} \frac{|h'(n_1) b(n_1)|}{n_1} 
\bigg| \frac{n_1}{y} 
\sum_{n_2 \leq y/n_1} h^{*H}(n_2) b(n_2) \bigg|
+ O(x^{-\delta/8}(\log y)^{O(H)}).
\end{align*} 
\end{lemma}

\begin{proof}
Recall that $h'$ is supported on square-full numbers only and that
$|h'(p^k)| \leq (2H)^k$ by \eqref{eq:h'-bound}.
Since $b$ is completely multiplicative, we have
\begin{align*}
&\sum_{n\leq y} f(n) b(n) \\
&=
\sum_{n_1 n_2 \leq y} h'(n_1)h^{*H}(n_2) b(n_1)b(n_2)\\
&\leq
\sum_{n_1 \leq y^{\delta}} |h'(n_1) b(n_1)| \bigg| \sum_{n_2 \leq y/n_1} 
h^{*H}(n_2) b(n_2) \bigg|
+y\sum_{n_1 > y^{\delta}} \frac{|h'(n_1) b(n_1)|}{n_1}
 \sum_{n_2 \leq y/n_1} \frac{|h^{*H}(n_2)|}{n_2}\\
&\leq
\sum_{n_1 \leq y^{\delta}} |h'(n_1) b(n_1)| \bigg| \sum_{n_2 \leq y/n_1} 
h^{*H}(n_2) b(n_2) \bigg|
+ y(\log y)^{O(H)}
\sum_{\substack{n_1 \geq y^{\delta} \\ p|n_1 \Rightarrow p > w}} 
\frac{|h'(n_1)|}{n_1}.
\end{align*}
By decomposing every square-full number $n_1$ as $m^2d$ with $d|m$, 
we obtain the following bound for the sum in the final term:
\begin{align} \label{eq:d_0-sum}
 \sum_{\substack{n_1 \geq y^{\delta} \\ p|n_1 \Rightarrow p > w}} 
 \frac{|h'(n_1)|}{n_1} 
 &\leq \sum_{\substack{m \geq y^{\delta/3} \\ p|m \Rightarrow p > w}} 
 \frac{(2H)^{\Omega(m^2)}}{m^2} \sum_{d|m} 
 \frac{(2H)^{\Omega(d)}}{d} \\
 \nonumber
 &\leq \sum_{\substack{m \geq y^{\delta/3} \\ p|m \Rightarrow p > w}} 
 \frac{(2H)^{\frac{2\log m}{\log w}}}{m^2} 
 d(m)\\
 \nonumber
 &\ll 
 \sum_{\substack{m \geq y^{\delta/3} \\ p|m \Rightarrow p > w}} 
 m^{-2+ \frac{2\log (2H)}{\log w}+ \frac{1}{8}}
 \ll \sum_{\substack{m \leq y^{\delta/3} \\ p|m \Rightarrow p > w}} 
 m^{-2+  \frac{1}{4}}
 \ll y^{-\delta/4},
\end{align}
where we used the bound $d(n) \ll n^{1/8}$.
Combining the above two bounds completes the proof.
\end{proof}

\begin{proof}[Proof of Corollary \ref{c:GS}]
To start with, we consider the bounded multiplicative function $h$.
Note that Proposition \ref{p:GS} applies to values of $X$ with 
$x^{1/2} \leq X \leq x$.
Our application, will, however, require a range of the form 
$x^{1/(4H)} \leq X \leq x$.
For this reason, we will apply Proposition \ref{p:GS} once with $x$ replaced by 
$x^{1/2^j}$ for each $j \in \{ 0, \dots, \lceil \log_2(4H) \rceil\}$.
If $\mathcal{C}$ is as in the statement of the corollary, then Proposition 
\ref{p:GS} shows that for all $Q \leq (\log x^{1/(8H)})^{C}$ and for all 
$x^{1/(4H)} < X \leq x$, we have
\begin{align*}
\frac{1}{X}
\frac{Q}{\phi(Q)}
\bigg|
\sum_{\chi \in \mathcal{C}} 
\chi(A)
\sum_{n\leq X} h(n) \bar\chi(n) \bigg|
&\ll_{C,H,\alpha_f} 
\left(\frac{\log \log x}{\log x}\right)^{1-\frac{1}{\sqrt{k}}} 
\log\left(\frac{\log x}{\log \log x}\right)\\
&\ll_{C,H,\alpha_f}
(\log x)^{-1+\alpha_f/(2H)}
(\log \log x)^2,
\end{align*}
since
$1/\sqrt{k} = 1/\sqrt{[\eps^{-2}]} \leq \eps = \frac{1}{2}\min(1,\alpha_f/H) 
\leq \alpha_f/(2H)$.

By property (2) of Definition \ref{def:M}, we have
$$
\frac{Q}{\phi(Q)}
\exp \bigg( \sum_{\substack{p\leq x\\p\nmid Q}} \frac{|h(p)|}{p}\bigg)
\geq \exp \bigg( \sum_{\substack{p\leq x\\p\nmid Q}} \frac{|f(p)|}{Hp}\bigg)
\geq \bigg(\frac{\log x}{C\log \log x}\bigg)^{\alpha_f/H},
$$
and, thus,
\begin{align}\label{eq:preliminary}
 \nonumber
 \frac{1}{X}
\frac{Q}{\phi(Q)}
\bigg|
\sum_{\chi \in \mathcal{C}} 
\chi(A)
\sum_{n\leq X} h(n) \bar\chi(n) \bigg|
&\ll_{C,H,\alpha_f} 
 \frac{(\log \log x)^{2+\alpha_f/H}}{ (\log x)^{1+\alpha_f/(2H)}}
 \frac{Q}{\phi(Q)}
\exp \bigg( \sum_{\substack{p\leq x\\p\nmid Q}} \frac{|h(p)|}{p}\bigg)\\
&\ll_{C,H,\alpha_f} 
 \frac{1}{ (\log x)^{1+\alpha_f/(3H)}} \frac{Q}{\phi(Q)}
\exp \bigg( \sum_{\substack{p\leq x\\p\nmid Q}} \frac{|h(p)|}{p}\bigg).
\end{align}

To handle the case where $H>1$, consider the decomposition $f=h^{*H} * h'$ with 
$h$ as in \eqref{eq:h(p)/H}.
If $x^{1/2} \leq y \leq x$, then Lemma \ref{l:h'-contribution} implies that 
for any $\delta \in (0,1)$,
\begin{align} \label{eq:remove-h'}
&\frac{1}{y} 
\sum_{\chi \in \mathcal{C}}
\chi(A) \sum_{n\leq y} f(n) \bar\chi(n) \\
\nonumber
&\leq
\frac{1}{y}
\sum_{\chi \in \mathcal{C}}
\sum_{d_0 \leq y^{\delta}} |h'(d_0) \bar\chi(d_0)| \bigg| \sum_{d \leq y/d_0} 
h^{*H}(d) \bar\chi(d) \bigg|
+ O(x^{-\delta/8}(\log x)^{O(H)}).
\end{align}
The error term in this bound is acceptable. 
A generalisation of the hyperbola method applied to the sum over $d$ (see 
Section \ref{ss:hyperbola} for a deduction) shows that the main term 
satisfies:
\begin{align}\label{eq:less-preliminary}
&\frac{1}{y}
\sum_{\chi \in \mathcal{C}}
\sum_{d_0 \leq y^{\delta}} |h'(d_0) \bar\chi(d_0)| \bigg| \sum_{d \leq y/d_0} 
h^{*H}(d) \bar\chi(d) \bigg|\\
\nonumber
&\leq \sum_{\substack{d_0 \leq y^{\delta}\\ p|d_0 \Rightarrow p>w(x) }}  
\frac{|h'(d_0)|}{d_0}
\sum_{D \leq (y/d_0)^{1-1/H}}
\sum_{d_1 \dots d_{H-1} = D}
\frac{|h(d_1) \dots h(d_{H-1})||\bar\chi(D)|}{D} \times
\\
\nonumber
&\qquad \qquad \qquad \sum_{i=1}^H
\frac{Dd_0}{y}
\Bigg|~
\sum_{\chi \in \mathcal{C}}
\chi(A)
\sum_{\substack{n:\\ (y/d_0)^{1-1/H}
\max(d_1, \dots, d_{i-1})\\ \leq Dn \leq y/d_0}} 
h(n) \bar\chi(n) \Bigg|.
\end{align}
Observe that the upper bound on $n$ in the inner sum satisfies 
$y/(Dd_0) \in [y^{1/H-\delta},x]$. 
By choosing $\delta = 1/(4H)$, this interval is contained in 
$[x^{1/(2H)-\delta/2},x] = [x^{3/(8H)},x]$.
An application of the triangle inequality shows that the inner sum is bounded 
by 
$$
\frac{Dd_0}{y}
\Bigg|~
\sum_{\chi \in \mathcal{C}}
\chi(A)
\sum_{ n \leq y/(d_0D)} h(n) \bar\chi(n) \Bigg|
+ \frac{Dd_0}{y}
\Bigg|~
\sum_{\chi \in \mathcal{C}}
\chi(A)
\sum_{n \leq y' } 
h(n) \bar\chi(n) \Bigg|
$$
where $y' = \min(y/(d_0D), (y/d_0)^{1-1/H} D^{-1} \max(d_1, \dots, d_{i-1}))$. 
We are now in the position to apply \eqref{eq:preliminary} to bound the first 
of these terms by
$$
\ll_{C,H,\alpha_f} 
 \frac{1}{ (\log x)^{1+\alpha_f/(3H)}}
\exp \bigg( \sum_{\substack{p\leq x, \, p\nmid Q}} \frac{|h(p)|}{p}\bigg).
$$
If $y' > x^{1/(4H)}$, then the same bound applies to the second term.
If, on the other hand, $y' \leq x^{1/(4H)} \leq y^{1/(2H)}$, then the second 
term may trivially 
be bounded by 
$$\phi(Q) y^{1/(4H)-1} d_0D 
\leq \phi(Q) y^{1/(4H)-1 + 1/(4H) + 1 - 1/H}
\leq (\log x)^{C} x^{-1/(4H)}.$$
Inserting these bounds into \eqref{eq:less-preliminary} and completing the 
outer sums, we deduce that
\eqref{eq:less-preliminary} is bounded by
\begin{align*}
\ll_{C,H,\alpha_f} 
 \frac{1}{ (\log x)^{1+\alpha_f/(3H)}}
\exp \bigg( \sum_{\substack{p\leq x, \, p\nmid Q}} \frac{|h(p)|}{p}\bigg)
\sum_{\substack{d_0:  p|d_0 \Rightarrow p>w(x) }}  
\frac{|h'(d_0)|}{d_0}
\bigg(\sum_{d \leq x} \frac{|h(d)\chi(d)|}{d} \bigg)^{H-1}.
\end{align*}
The sum over $d$ in this bound satisfies
$$\bigg(\sum_{d \leq x} \frac{|h(d)\chi(d)|}{d} \bigg)^{H-1}
\leq \prod_{p \leq x} \left(1 + \frac{|h(p)\chi(p)|}{p}\right)^{H-1}
\leq \exp \bigg( (H-1) \sum_{p\leq x, p \nmid Q} \frac{|h(p)|}{p} 
\bigg),$$
and the sum over $d_0$ converges by \eqref{eq:d_0-sum}, applied with $y=1$, 
provided $x$ is sufficiently large for
$w(x) \geq (2H)^{16}$ to hold.
Collecting all information together, it follows from \eqref{eq:remove-h'} 
that
\begin{align*}
\frac{1}{x}
\frac{Q}{\phi(Q)} 
\sum_{\chi \in \mathcal{C}} 
\chi(A) \sum_{n\leq x} f(n) \bar\chi(n) 
&\ll_{C,H,\alpha_f} 
\frac{1}{(\log x)^{1+\alpha_f/(3H)}}
\frac{Q}{\phi(Q)}
\exp \bigg( \sum_{\substack{p\leq x, \,p\nmid Q}} \frac{|f(p)|}{p}\bigg),
\end{align*}
which competes the proof.
\end{proof}

With Corollary \ref{c:GS} in place, we obtain the following sufficient 
condition for $f \in \mathcal{M}_H$ to belong to $\mathcal{F}_H$:
\begin{proposition}[Sufficient condition] \label{p:sufficient-F}
Suppose that $f \in \mathcal{M}_H$. Then $f \in \mathcal{F}_H$ if the 
following holds.
For every $C>0$, there exists a function $\psi_C: \RR_{\geq 0} \to \RR_{\geq 0}$, 
with the property that $\psi_C(x) \to 0$ as $x \to \infty$, such that
\begin{equation} \label{eq:p-sufficient-F}
  S_{f\chi}(x') 
  = S_{f \chi}(x) 
  + O\Bigg( \frac{\psi_C(x)}{\log x}
  \exp\bigg(\sum_{\substack{p \leq x, \,p \nmid Q}} 
            \frac{|f(p)|}{p}\bigg) \Bigg), 
\qquad ( x \geq 2),
\end{equation}
uniformly for all $x' \in (x(\log x)^{-C},x]$ and all
characters $\chi \Mod{Q}$ with $1 < Q \leq (\log x)^C$ and $W(x)|Q$.
\end{proposition}

\begin{proof}
Recall from Definition \ref{d:F} that we have to show that there exists $\varphi_C = o(1)$
such that
\begin{equation} \label{eq:recall-1.5}
|S_{f}(x';Q,A) 
- S_{f}(x;Q,A)|
= O\Bigg( \frac{\varphi_C (x)}{\log x}
 \frac{Q}{\phi(Q)} 
 \prod_{\substack{p \leq x, \,p\nmid Q}} 
 \left(1 + \frac{|f(p)|}{p}\right)
 \Bigg)
\end{equation}
uniformly for all $x' \in (x(\log x)^{-C},x]$, all $1 \leq Q \leq (\log x)^C$ with $W(x)|Q$ and 
all reduced $A \Mod{Q}$.
This will be a straightforward consequence of the fact that by Corollary~\ref{c:GS} 
there are only finitely many characters in the character sum expansions of 
$S_f(x';Q,A)$ and $S_f(x;Q,A)$ that matter.
Using the notation from the corollary, let $\mathcal{E}(Q)$ denote the set of 
characters modulo $Q$ that are induced by the elements of $\mathcal{E}_0 
\cup \dots \cup \mathcal{E}_{k'}$. Then
\begin{align*}
S_f(x';Q,A) &= 
\frac{Q}{\phi(Q)} 
\sum_{\substack{\chi \Mod{Q} \\ \chi \in \mathcal{E}(Q)}} 
\chi(A) 
S_{f \bar \chi} (x')
+ O\Bigg( \frac{\psi(x)}{\log x}
\frac{Q}{\phi(Q)}
\exp \bigg( \sum_{\substack{p\leq x, \, p\nmid Q}} \frac{|f(p)|}{p}\bigg) 
\Bigg),
\end{align*}
where $\psi(x) = O_{C,H,\alpha_f}((\log x)^{-\alpha_f/(3H)})$,
uniformly in $x'$, $Q$ and $A$ as above. 
Thus, \eqref{eq:recall-1.5} follows from our assumptions with
$\psi_C(x) = \psi(x) + \# \mathcal{E} \cdot \varphi_C(x)$.
\end{proof}

\begin{example}[Applications using Selberg--Delange type arguments]
The conditions required by Proposition \ref{p:sufficient-F} 
are of a type that can usually be checked by means of 
the Selberg--Delange method (see e.g.\ Tenenbaum \cite[Section II.5]{Tenen})
provided the function $f$ is closely related to a $\zeta$- or $L$- function. 
The range of the modulus $Q$ of the characters $\chi$ that appear is small 
enough to ensure that exceptional characters can be handled.
Examples of functions suitable for this approach include:
\begin{enumerate}
 \item[(i)] the function 
$$\frac{r(n)}{4} = \frac{1}{4}\#\{(x,y) \in \ZZ^2: x^2+y^2 =n\},$$
 \item[(ii)] the indicator function of the set of sums of two squares,
 \item[(iii)] the characteristic function of set of numbers composed of primes 
that split completely in a given Galois extension $K/\QQ$ of finite degree.
\end{enumerate}
\end{example}

In the following subsection, we will further analyse the Lipschitz condition \eqref{eq:p-sufficient-F} and 
prove another sufficient condition, in this case for an element $f \in \mathcal{M}_H$ to 
belong to $\mathcal{F}_{H,n^{it}}$.

\subsection{Lipschitz estimates for elements of $\mathcal{M}_H$ and another sufficient condition}

For applications of Proposition \ref{p:sufficient-F} or Corollary \ref{c:GS}, the following four
lemmas, which we all prove in Section \ref{ss:proofs-of-lemmas}, are very useful.
The first lemma is a slight generalisation of the Lipschitz estimate for 
bounded multiplicative functions and a related decay estimate that Granville and Soundararajan 
established in \cite[Theorems 3 and 4]{GS-decay}.

\begin{lemma}[Lipschitz estimates]\label{l:lipschitz}
Let $f_0 \in \mathcal{M}_1$ and let $x \geq 3$.
Suppose that $f : \NN \to \CC$ is multiplicative, bounded in absolute value by $1$ and satisfies
$|f(p^k)| = |f_0(p^k)|$ for all primes $p \geq \exp((\log \log x)^2)$ and $k \geq 1$. 
 Define
 $$
 F(s) = 
 \prod_{p \leq x} 
 \left( 1 + \frac{f(p)}{p^s} + \frac{f(p^2)}{p^{2s}} + \dots \right).
 $$  
 If the maximum of 
 $$
 \max_{|y| \leq 2 \log x} |F(1+iy)|
 $$
 is attained at $y=t_{x,f}$, then,
  uniformly in $x$ and $f$ as above, 
  we have
\begin{equation} \label{eq:lipschitz}
 \Big| \frac{1}{x} \sum_{n \leq x} f(n)n^{-it_{x,f}} 
  - \frac{1}{x'} \sum_{n \leq x'} f(n)n^{-it_{x,f}} \Big| 
 \ll_{f_0} \frac{1}{(\log x)^{1+C_0}} \exp \Bigg(\sum_{p \leq x} \frac{|f(p)|}{p}\Bigg)
\end{equation}
for all $x' \in [x \exp(-(\log \log x)^{-4}),x]$, 
where $C_0\in (0,\alpha_{f_0}/2)$ is a positive constant that only depends on $\alpha_{f_0}$.
Furthermore, we have, for any $t_{x,f}$ as above,
\begin{equation} \label{eq:GS-Thm3}
 \Big| \frac{1}{x} \sum_{n \leq x} f(n) \Big|
 \ll_{f_0} \frac{1}{|t_{x,f}| + 1} + \frac{\log \log x}{\log x} +
 \frac{1}{(\log x)^{1+C_0}} \exp \Bigg(\sum_{p \leq x} \frac{|f(p)|}{p}\Bigg).
\end{equation}
\end{lemma}

The conditions on $f$ above will allow us to apply the lemma to twists 
$h\chi$ where $h \in \mathcal{M}_1$ and $\chi$ is a character modulo $Q$ 
with $Q \leq (\log x)^{C}$ for any given constant $C>0$.
In order to extend this to twists $f \chi$ for $f \in \mathcal{M}_H$ and
$H>1$, we note that the function $h$ associated to $f$ via  
\eqref{eq:h(p)/H} belongs to $\mathcal{M}_1$. 
The following lemma will enable us to employ Lemma \ref{l:lipschitz} in general.

\begin{lemma} \label{l:lipschitz-h-to-f}
Let $f \in \mathcal{M}_H$ and let $h$ be as in \eqref{eq:h(p)/H}.
Let $C>0$ be a fixed constant and let $\psi:\RR_{\geq 0} \to \RR_{\geq 0}$ be a function that satisfies 
$\psi(x) \to 0$ as $x \to \infty$.
Let $x > 1$ and suppose that $\chi \Mod{Q}$, with 
$Q \leq (\log x)^C$ and $W(x)|Q$, is a character such that
\begin{equation*}
  | S_{h\chi}(y) 
  - S_{h \chi}(y') |
  \leq \frac{\psi(x)}{\log y}
            \exp\bigg(\sum_{\substack{p \leq y,\, p \nmid Q}} 
            \frac{|h(p)|}{p}\bigg) 
\end{equation*}
for all $y \in (x^{1/(2H)},x]$ and $y' \in (y(\log x)^{-C},y]$.
Then, for all $x' \in (x(\log x)^{-C},x]$, we have
\begin{equation*}
  | S_{f\chi}(x) -  S_{f \chi}(x') |
  \leq \frac{\psi'(x)}{\log x}
            \exp\bigg(\sum_{\substack{p \leq x,\, p \nmid Q}} 
            \frac{|f(p)|}{p}\bigg),
\end{equation*}
where $\psi'$ is independent of $\chi$ and $Q$ and satisfies $\psi'(x) \to 0$ as $x \to \infty$;
more precisely,
$$\psi'(x)=O_{H,C}\Big(\psi(x) + (\log x)^{- \min(1, \alpha_f/(2H))} + x^{-1/8} (\log x)^{O(H)}\Big).$$
\end{lemma}

The next lemma shows that if $\chi$ is a character that is negligible in the 
application of Proposition \ref{p:sufficient-F} or Corollary \ref{c:GS}
to the function $h$, then $\chi$ is also negligible in an application of the proposition or corollary to $f$.

\begin{lemma} \label{l:small-chi-f-to-h}
Let $f \in \mathcal{M}_H$, let $h$ be as in \eqref{eq:h(p)/H} and let $C>0$.
Let $\psi:\RR_{\geq 0} \to \RR_{\geq 0}$ be a function that satisfies 
$\psi(x) \to 0$ as $x \to \infty$.
Let $x>1$ and suppose that $\chi \Mod{Q}$, with $Q \leq (\log x)^C$ and $W(x)|Q$, is 
any character such that
$$
|S_{h\chi}(x')| 
\leq \frac{\psi(x)}{\log y}
  \exp\bigg(\sum_{\substack{p \leq y, \,p \nmid Q}} 
            \frac{|h(p)|}{p}\bigg) 
$$
for all $x' \in [x^{1/(4H)}, x]$. Then
$$
|S_{f\chi}(y)| 
\leq \frac{\psi'(x)}{\log x}
  \exp\bigg(\sum_{\substack{p \leq x, \,p \nmid Q}} 
            \frac{|f(p)|}{p}\bigg), \qquad(y \in [x^{1/2},x]),
$$
where $\psi'$ is independent of $\chi$ and $Q$ and satisfies
$\psi'(x) = O(\psi(x) + x^{-1/(32 H)}(\log x)^{O(H)})$.
\end{lemma}

Finally, we observe that \eqref{eq:lipschitz} holds \emph{uniformly} 
in $f$ for some $t = t_x$ that only depends on $x$ and $f_0$.
This proves particularly valuable when dealing with families of induced characters.

\begin{lemma} \label{l:lipschitz-t-uniform}
 Let $x$, $f_0$ and $f$ be as in Lemma \ref{l:lipschitz}, and let $x' \in [x\exp(-(\log \log x)^4/2),x]$. 
 Then there exists $|t_x|\leq 2 \log x$, only dependent on $x$ and $f_0$, but not on $f$ or $x'$, 
 such that, uniformly in $x$, $x'$ and $f$ as before,
\begin{equation} \label{eq:lipschitz'}
 \Big| \frac{1}{x} \sum_{n \leq x} f(n)n^{-it_x} 
  - \frac{1}{x'} \sum_{n \leq x'} f(n)n^{-it_x} \Big| 
 \ll \frac{1}{(\log x)^{1+C_0}} \exp \Bigg(\sum_{p \leq x} \frac{|f(p)|}{p}\Bigg),
\end{equation}
where $C_0>0$ is a positive constant that only depends on $\alpha_{f_0}$.
\end{lemma}

As a consequence of Lemmas \ref{l:lipschitz}--\ref{l:lipschitz-t-uniform} and 
Corollary \ref{c:GS}, we obtain the following sufficient condition for testing whether a function 
belongs to $\mathcal{F}_{H,n^{it}}$.

\begin{proposition}[Another sufficient condition] \label{p:sufficient-2}
Let $f \in \mathcal{M}_H$ and $h$ as in \eqref{eq:h(p)/H}.
For every $C>0$,  let $\psi_C:\RR_{\geq 0} \to \RR_{\geq 0}$ be a function that satisfies
$\psi_C(x) \to 0$ as $x \to \infty$.
Suppose that for every sufficiently large $x$ there exists $\tau_x \in \RR$ with 
$|\tau_x| \leq 2 \log x$ such that the following holds:
If $1 \leq Q \leq (\log x)^C$ with $W(x)|Q$ and if $\chi \Mod{Q}$ is a character then either the bound
\begin{align} \label{eq:p-sufficient-2}
|S_{g_x\chi}(x')| \leq \frac{\psi_C(x')}{\log x'}
  \exp\bigg(\sum_{\substack{p \leq x', p \nmid Q}} 
            \frac{|g(p)|}{p}\bigg),
  \qquad(x^{1/(8H)} \leq x' \leq x),
\end{align}
holds for either $g=h$ or $g=f$ and for $g_x:n\mapsto g(n)n^{-i\tau_x}$,
or else we have $t_{x,f}=\tau_x$ or $t_x=\tau_x$ in the statement of 
Lemma \ref{l:lipschitz} or Lemma \ref{l:lipschitz-t-uniform}
when applied with $f_0=h$ and with $f$ replaced by $h \chi$.

Then the function $n \mapsto f(n)n^{-i\tau_x}$ satisfies \eqref{eq:d-F} and 
$f \in \mathcal{F}_{H,n^{it}}$.
\end{proposition}

\begin{example} \label{ex:prop-suff-2}
The above proposition applies to:

(i) the M\"obius function $f = \mu$. 
In this case we may take $\tau_x=0$ for all $x$
since $S_{\mu \chi}(x) \ll_B q^{1/2} (\log x)^{-B}$ for all $B>0$ and all 
$\chi \Mod{q}$.\footnote{This is the same information about $\mu$ as was used in 
Green and Tao's work \cite{GT-nilmobius} (see \cite[Proposition A.1]{GT-nilmobius}) to handle the `major arcs'.}
Indeed, if $\chi$ is a trivial character this estimate follows from Prime-Number-Theorem-type bounds on $S_{\mu}(x)$;
see Example \ref{ex:mobius-4.4}(i) for details.
If $\chi \Mod{q}$ is non-trivial, then its conductor, $q'$ say, is at least $2$ and one may deduce the estimate from
\cite[Corollary 5.29]{IK}, which proves the claimed bound for non-trivial primitive characters.
In fact, if $q \leq x$, then it follows from \cite[(5.79)]{IK} that
$$\sum_{p \leq x} \chi(p) \ll_B {q'}^{1/2} x (\log x)^{-B} + \omega(q) \ll_B q^{1/2} x (\log x)^{-B},$$
since $\omega(q) \ll \log x$.
If $q>x$, then $\sum_{p \leq x} \chi(p) \ll_B {q}^{1/2} x (\log x)^{-B}$ holds trivially.
Thus, \cite[(5.79)]{IK} generalises to all non-trivial $\chi$ and, by following the original proof from \cite{IK}, 
so does \cite[(5.80)]{IK}.

(ii) every multiplicative function $f$ that takes values on the unit circle, i.e.\ for which $|f(n)|=1$ for all $n$. 
This, in turn, follows from \cite[Theorem 2]{BGS}, which provides the bound
\begin{equation} \label{eq:unit-circle-decay}
 |S_{f \chi}(x)| \ll \Big((\log \log x)^2/\log x\Big)^{1/20},
\end{equation} 
valid for all characters $\chi$ of conductor $Q \leq \exp((\log \log x)^2)$, 
except perhaps for those induced by one exceptional character, $\xi$ say.
By Lemma~\ref{l:lipschitz-t-uniform}, there exists $|t_x| \leq 2 \log x$ such that \eqref{eq:lipschitz'} holds for all
$\chi \Mod{Q}$ induced by $\xi$.
Suppose now that $|t_x| > (\log x)^{1/100}$. 
Then Lemma~\ref{l:lipschitz-t-uniform}, combined with the bound \eqref{eq:GS-Thm3}, implies that the above bound on 
$|S_{f \chi}(x)|$ also holds for characters induced by $\xi$.
In this case, we may take $\tau_x = 0$.
If, however, $|t_x| \leq (\log x)^{1/100}$, then we may use partial summation to deduce from \eqref{eq:unit-circle-decay}
that $$|S_{f_x \chi}(x)| \ll \Big((\log \log x)^2/\log x\Big)^{3/100}$$ for all $\chi$ not induced by $\xi$.
In this case, we may set $\tau_x=t_x$.
\end{example}

\begin{proof}[Proof of Proposition \ref{p:sufficient-2}]
 This result follows from Corollary \ref{c:GS} in a similar way as Proposition \ref{p:sufficient-F} does.
 To show that $n \mapsto f_x(n) := f(n)n^{-i\tau_x}$ satisfies \eqref{eq:d-F}, let $1 \leq Q \leq (\log x)^C$ be such that $W(x)|Q$
 and let $\mathcal{E}(Q)$ denote the set of characters modulo $Q$ that are induced by the elements of
 $\mathcal{E}_0 \cup \dots \cup \mathcal{E}_{k'}$ from Corollary \ref{c:GS}, when applied to the function 
 $f_x$.
 Then
\begin{align} \label{eq:prop4.10-proof}
S_{f_x}(y;Q,A) - S_{f_x}(x;Q,A) &= 
\frac{Q}{\phi(Q)} 
\sum_{\substack{\chi \Mod{Q} \\ \chi \in \mathcal{E}(Q)}} 
\chi(A) 
\left( S_{f_x \bar \chi} (y) - S_{f_x \bar \chi} (x) \right) \\
\nonumber
& \quad+ O_{C,H,\alpha_f}\Bigg( (\log x)^{-\alpha_f/(3H)}
\frac{Q}{\phi(Q)}
\frac{1}{\log x}
\exp \bigg( \sum_{\substack{p\leq x, \, p\nmid Q}} \frac{|f(p)|}{p}\bigg) 
\Bigg),
\end{align}
whenever $x^{1/2} \leq y \leq x$.

We begin with the contribution from those characters $\bar \chi$ to which the first alternative from the statement applies.
Observe that \eqref{eq:p-sufficient-2} implies that
$$
|S_{g_x\chi}(x')| \leq \frac{\psi'_C(x)}{\log x}
  \exp\bigg(\sum_{\substack{p \leq x, p \nmid Q}} 
            \frac{|g(p)|}{p}\bigg),
  \qquad(x^{1/(8H)} \leq x' \leq x),
$$
where $\psi'_C(x)= O_H(1) \max_{x^{1/(8H)} \leq x' \leq x} \psi_C(x')$.
To see this, note that for all $x'$ as above,
\begin{align*}
 \nonumber
 &\exp\bigg(\sum_{p \leq x, \,p \nmid Q} \frac{|g(p)|}{p}\bigg)
  \exp\bigg(-\sum_{p \leq x', \,p \nmid Q} \frac{|g(p)|}{p}\bigg)
 \leq \exp\bigg(\sum_{x^{1/(8H)} < p \leq x} \frac{H}{p}\bigg)\\
 &\leq \exp\bigg(H (\log \log x + \log (8H) - \log \log x + o(1))\bigg)
  \leq (8H)^{H(1 + o(1))} \ll_{H} 1.
\end{align*}
Thus, by Lemma \ref{l:small-chi-f-to-h} it follows that all characters $\bar \chi$ with 
$\chi \in \mathcal{E}(Q)$ to which the first alternative from the statement applies satisfy
 \begin{align}\label{eq:prop4.10-proof-2}
 |S_{f_x \bar \chi} (y)| \leq 
 \frac{\psi_C''(x)}{\log x}
 \exp \bigg( \sum_{\substack{p\leq x, \, p\nmid Q}} \frac{|f(p)|}{p}\bigg), \qquad(x^{1/2} \leq y \leq x), 
 \end{align}
 for a suitable function $\psi_C''=o(1)$.
 
For all remaining $\chi \in \mathcal{E}(Q)$, Lemma \ref{l:lipschitz} or Lemma \ref{l:lipschitz-t-uniform} 
provides the Lipschitz estimate 
$$
S_{f_x \bar \chi}(y) = S_{f_x \bar \chi}(x) + O\Bigg( \frac{\psi'''(x)}{\log x}
 \exp \bigg( \sum_{\substack{p\leq x, \, p\nmid Q}} \frac{|f(p)|}{p}\bigg)  \Bigg),
 \qquad (y \in [x\exp(-(\log \log x)^4/2),x])
$$
with $\psi'''(x) = (\log x)^{-C_0}$.
Thus, the result follows from \eqref{eq:prop4.10-proof}.
\end{proof}

\subsection{Proofs of Lemmas \ref{l:lipschitz}--\ref{l:lipschitz-t-uniform}} \label{ss:proofs-of-lemmas}
We prove Lemmas \ref{l:lipschitz}, 
\ref{l:lipschitz-t-uniform},
\ref{l:lipschitz-h-to-f} and \ref{l:small-chi-f-to-h}, 
in this order.

\begin{proof}[Proof of Lemma \ref{l:lipschitz}]
The proof of this lemma is almost identical to the proofs of the original results,
\cite[Theorem 3 and 4]{GS-decay}, except for one ingredient, namely 
\cite[Lemma 2.3]{GS-decay}, which needs to be replaced by Lemma \ref{l:GS-2.3} 
below.
The estimate \eqref{eq:GS-Thm3} follows immediately from \cite[\S5]{GS-decay} and Lemma \ref{l:GS-2.3}.
Concerning the Lipschitz estimate \eqref{eq:lipschitz}, we replace the application of 
\cite[Theorem 3]{GS-decay} at the beginning of \cite[\S6]{GS-decay} by the estimate \eqref{eq:GS-Thm3}.
The bound in \cite[eq.\ (6.2)]{GS-decay} continues to apply.
The first term in this bound is acceptable since in our case 
$w \leq \exp((\log \log x)^4)$, and since $C_0<\alpha_{f_0}$.
To bound the integrand in the second term, we use the bound 
\cite[eq.\ (6.5)]{GS-decay}
if $\alpha$ is large, which in our situation means  
$\alpha > 
\exp( \sum_{p\leq x} \frac{|f(p)|}{p} )^{-1} (\log x)^{C_0}$ 
with $C_0=C_0(\alpha_{f_0},1)$ as in the lemma below.
If 
$\alpha \leq \exp(\sum_{p\leq x} \frac{|f(p)|}{p} )^{-1} (\log x)^{C_0}$,
we proceed as in the small-$\alpha$-case from the original proof
but, again, apply our Lemma~\ref{l:GS-2.3} instead of \cite[Lemma 2.3]{GS-decay}.
\end{proof}

\begin{lemma}[`new Lemma 2.3'] \label{l:GS-2.3}
Let $x\geq 3$, $f_0 \in \mathcal{M}_H$ and let $f : \NN \to \CC$ be a multiplicative function 
such that $|f(p^k)| \leq |f_0(p^k)|$ at all prime powers $p^k$, and such that
$|f(p^k)| = |f_0(p^k)|$ whenever $p \geq \exp((\log \log x)^2)$ and $k \in \NN$. 
 Let 
$$
 F(s)
  = \prod_{p\leq x} 
    \left( 1 + \frac{f(p)}{p^s} + \frac{f(p^2)}{p^{2s}} + \dots \right).
$$
Then there exists a positive constant $C_0 = C_0(\alpha_{f_0},H) \in (0, \alpha_f/2)$ such that 
for all real numbers $y$ and $1/\log x \leq |\beta| \leq \log x$,
we have
$$
|F(1 + iy) F(1 + i(y+\beta))| \ll
\exp \Big( 2 \sum_{p\leq x} \frac{|f(p)|}{p} \Big) 
(\log x)^{-2C_0}.
$$
\end{lemma}

\begin{rem*}
 Observe that we actually only use this lemma in the case where $H=1$.
\end{rem*}

\begin{proof}
Suppose $|f(p)| = g(p) +h(p)$ for two non-negative functions $g$ and $h$.
Then
\begin{align}\label{eq:GS-lemma-0}
|F(1 + iy) F(1 + i(y+\beta))| 
&\ll \exp \Big(\Re \sum_{p\leq x} \frac{f(p)p^{-iy} + f(p)p^{-i(y+\beta)}}{p} 
\Big) \\
\nonumber
&\ll \exp \Big( \sum_{p\leq x} \frac{(g(p) + h(p)) |1 + p^{-i\beta}|}{p} \Big) 
\\
\nonumber
&\ll \exp \Big( 2 \sum_{p\leq x} 
          \frac{g(p) |\cos(\frac{|\beta|}{2} \log p)|}{p} \Big) 
     \exp \Big( 2 \sum_{p\leq x} \frac{h(p)}{p} \Big).
\end{align}
The aim is to exploit the fact that $|\cos|$ is not the constant function $1$ 
in order to bound this expression.
We begin by decomposing the set of primes less than $x$ into subsets on which 
$|\cos(\frac{|\beta|}{2} \log p)|$ is almost constant.
For this purpose, let $\delta = 1/(\log x)^3$ and consider the decomposition of 
$[0, 2 \pi)$ into intervals of the form 
$((n-1)\log(1+\delta) |\beta|/2,n\log(1+\delta) |\beta|/2]$.
Thus, in order to cover the interval $[0, 2 \pi)$, the parameter 
$n$ runs over the range $1 \leq n \leq N$, for some $N \asymp (\delta 
|\beta|)^{-1}$, and, in particular, we have $(\log x)^{2} \ll N \ll (\log x)^4$.
By changing $\delta$ slightly, we can insure that 
$N\log(1+\delta) |\beta|/2 = 2 \pi$
so that the decomposition of $[0, 2 \pi)$ has exactly $N$ full intervals and 
no smaller or larger ones.
Next, we set $Y = \exp((\log \log x)^2)$ and decompose the set of primes in the 
interval $[Y,x]$ into $N$ sets of the form
$$
P_n(x) = 
 \bigcup_{m \equiv n \Mod{N}} 
 \{p\in (Y(1+\delta)^{m-1}, Y(1+\delta)^{m}] \cap (Y,x] \}, 
 \quad (1 \leq n \leq N).
$$
If
$M = \log (x/Y)/\log(1+\delta)$, the Brun-Titchmarsh inequality implies that 
for each $n \leq N$:
\begin{align}\label{eq:P_n-bound}
 \nonumber
 \sum_{p\in P_n(x)} \frac{1}{p}
 &\leq \sum_{\substack{0 \leq m \leq M \\ m \equiv n \Mod{N}}}
 \frac{ \pi\left(Y(1+\delta)^{m}\right) - \pi\left(Y(1+\delta)^{m-1}\right)}
 {Y(1+\delta)^{m-1}} \\ \nonumber
 &\ll \sum_{m \equiv n \Mod{N}} 
  \frac{ \delta }{\log(Y \delta (1+\delta)^{m-1})}\\ \nonumber
 &\ll \delta \sum_{0\leq k \leq M/N} \frac{1}{\log(x(1+\delta)^{-kN})}\\ 
 \nonumber
 &\ll \delta \sum_{0\leq k \leq M/N} \frac{1}{\log x - k N \log(1+\delta)}\\
 \nonumber
 &\ll \frac{\delta}{N\log(1+\delta)} 
 \log \Big(\frac{\log x}{N \log(1+\delta)}\Big)\\ 
 &\ll \frac{1}{N} \log \log x.
\end{align}
Suppose now that $g$ satisfies
\begin{equation} \label{eq:g-property}
\sum_{p \leq x} \frac{g(p)}{p} \sim \alpha \log \log x
\end{equation}
for some $\alpha > 0$ and
let $S \subset \{1,\dots,N\}$ denote the set of all indices $n$ such that
\begin{equation} \label{eq:def-S}
 \sum_{p \in P_n(x)} \frac{g(p)}{p} 
\geq \frac{\alpha}{2} \sum_{p \in P_n(x)}
\frac{1}{p}.
\end{equation}
Then, taking our choice of $Y$ into account and since $g(p) \leq |f(p)| \leq 
H$, we have 
\begin{equation}\label{eq:GS-lemma-b1}
\sum_{n\in S}\sum_{p \in P_n(x)} \frac{1}{p} 
\geq \frac{1}{H} \sum_{n\in S}\sum_{p \in P_n(x)} \frac{g(p)}{p} 
\geq \frac{1}{H} \bigg(\sum_{Y \leq p \leq x} \frac{g(p)}{p}
- \frac{\alpha}{2} \sum_{p \leq x} \frac{1}{p} \bigg)
\sim \frac{\alpha}{2H} \log \log x. 
\end{equation}
Comparing this bound with \eqref{eq:P_n-bound} shows that $S$ contains a 
positive proportion of the integers up to $N$.
Our next aim is to find a subset $T \subset S$ that satisfies
\begin{equation}\label{eq:GS-lemma-b2}
\sum_{n\in T}\sum_{p \in P_n(x)} \frac{1}{p} \geq
\frac{1}{2} \sum_{n\in S}\sum_{p \in P_n(x)} \frac{1}{p},
\end{equation}
and for which $|\cos(\frac{|\beta|}{2} \log p)|$ is bounded away from $1$ as
$p$ ranges over $\bigcup_{n\in T} P_n$.

By \eqref{eq:P_n-bound} and \eqref{eq:GS-lemma-b1}, we can 
choose a positive proportion of all $n \leq N$ not to belong to $T$. 
In particular, we can exclude all $n$ from $T$ for which
$((n-1)\log(1+\delta) |\beta|/2,n\log(1+\delta) |\beta|/2]$ intersects
$[0, c) \cup (\pi-c,\pi+c) \cup (2\pi - c, 2\pi)$ for some small constant $c>0$ 
that only depends on $\alpha$, $H$ and on the implied constant in 
\eqref{eq:P_n-bound}.
By doing so, we ensure that
$$\left| \cos\Big(\frac{|\beta|}{2} \log p \Big) \right| < \cos c < 1$$
for all $p \in P_n(x)$ with $n\in T$.
Writing $c':=\cos c$ and considering the cosine sum in the final 
expression of \eqref{eq:GS-lemma-0}, the above yields
\begin{align*}
 \sum_{p \in P_n(x)} \frac{g(p) |\cos(\frac{|\beta|}{2} \log p)|}{p}
\leq c' \sum_{p \in P_n(x)} \frac{g(p)}{p}
\leq \sum_{p \in P_n(x)} \frac{g(p)}{p} - (1 - c')
\sum_{p \in P_n(x)} \frac{g(p)}{p}
\end{align*}
for $n\in T$.
If $n\not\in T$, we have the trivial bound
\begin{align*}
 \sum_{p \in P_n(x)} \frac{g(p) |\cos(\frac{|\beta|}{2} \log p)|}{p}
&\leq \sum_{p \in P_n(x)} \frac{g(p)}{p}.
\end{align*}
By combining these two bounds with \eqref{eq:def-S}, 
\eqref{eq:GS-lemma-b1} and \eqref{eq:GS-lemma-b2}, it follows that
that
\begin{align*}
\sum_{p \leq x} \frac{g(p) |\cos(\frac{|\beta|}{2} \log p)|}{p} 
&\leq  \sum_{p\leq x} \frac{g(p)}{p} - 
(1 - c') \sum_{n\in T}\sum_{p \in P_n(x)} \frac{g(p)}{p}\\
&\leq  \sum_{p\leq x} \frac{g(p)}{p} - 
\frac{\alpha (1 - c')}{2} \sum_{n\in T}\sum_{p \in P_n(x)} 
\frac{1}{p}\\
&\leq \sum_{p\leq x} \frac{g(p)}{p} - (C_0 + o(1)) \log \log x
\end{align*}
for some constant $C_0 > 0$ that only depends on $\alpha$ and $H$.
By \eqref{eq:GS-lemma-0}, we thus deduce that
\begin{align*}
|F(1 + iy) F(1 + i(y+\beta))| \ll 
\exp \Big( 2 \sum_{p\leq x} \frac{|f(p)|}{p} \Big) 
(\log x)^{-2C_0}.
\end{align*}

It remains to show that there exists a decomposition of $|f(p)|$ into  
non-negative functions $g$ and $h$ such that \eqref{eq:g-property} holds.
This will follow from Elliott \cite[Lemma 5]{elliott}.
To apply this result, we observe that the two conditions from 
Definition \ref{def:M} and partial summation show that every 
$f_0 \in \mathcal{M}_H$ has the property that 
\begin{equation}\label{eq:elliott-liminf}
\liminf_{x \to \infty} \frac{1}{\eps \log x}
\sum_{x^{1-\eps } < p \leq x} \frac{|f_0(p)| \log p}{p} \geq \alpha_{f_0} .
\end{equation}
for every  $\eps \in (0,1)$.
Thus, the assumptions of \cite[Lemma 5]{elliott} are met and the lemma implies 
that there exists a non-negative completely multiplicative function 
$g_0 \leq |f_0|$, which satisfies
$$
\lim_{x \to \infty} (\log x)^{-1} \sum_{p \leq x} \frac{g_0(p) \log p}{p}
= \frac{\alpha_{f_0}}{2}.
$$
The function $g_0$ arises from $|f_0|$ as the result of a simple greedy-type argument 
that decides one by one for each prime $p$ if $g_0(p)=0$ or $g_0(p)= |f_0(p)|$.
Partial summation yields
$$\sum_{p\leq z} \frac{g_0(p)}{p} \sim \frac{\alpha_{f_0}}{2} \log \log z.$$
If we let $g(p) = g_0(p)$ for all $p \geq Y$ and $g(p)=0$ otherwise, then
$$\sum_{p\leq x} \frac{g(p)}{p} = 
\sum_{p\leq x} \frac{g_0(p)}{p} 
+ O(H \log \log Y)
\sim \frac{\alpha_{f_0}}{2} \log \log x + O(H \log \log \log x),$$
as required. 
Thus, we may set 
$\alpha = \alpha_{f_0}/2$ in the first part of the proof and, hence,
$C_0$ only depends on $\alpha_{f_0}$ and $H$.
\end{proof}

\begin{proof}[Proof of Lemma \ref{l:lipschitz-t-uniform}]
 Let $f^*$ denote the multiplicative function that satisfies
 $f^*(p^k)=0$ whenever $k\geq 1$ and $p \leq \exp((\log \log x)^2)$, and
 $f^*(p^k)=f_0(p^k)$ whenever $k\geq 1$ and $p > \exp((\log \log x)^2)$.
 Then, by applying Lemma \ref{l:lipschitz} twice, we have
 $$
 \Big| \frac{1}{y} \sum_{n \leq y} f^*(n)n^{-it} 
  - \frac{1}{y'} \sum_{n \leq y'} f^*(n)n^{-it} \Big| 
 \ll \frac{1}{(\log x)^{1+C_0}} \exp \Bigg(\sum_{p \leq x} \frac{|f(p)|}{p}\Bigg),
 $$
 for all $y,y' \in [x \exp(-(\log \log x)^{-4}),x]$ and some $t=t_{x,f^*}$ with $|t|\leq 2\log x$.
 
 Observe that any $f$ may be decomposed as $f = f' * f^*$, where $f'(p^k)=0$ for all $p > \exp((\log \log x)^2)$.
 Thus, if $w:=\exp((\log \log x)^{-4}/2)$ and
 $x' \in [x /w ,x]$, then
 \begin{align*}
 &\Big| \frac{1}{x} \sum_{n \leq x} f(n)n^{-it} 
  - \frac{1}{x'} \sum_{n \leq x'} f(n)n^{-it} \Big| \\
 &\ll \sum_{d \leq w} \frac{|f'(d)| }{d}
 \Big| \frac{d}{x} \sum_{n \leq x/d} f^*(n)n^{-it} 
  - \frac{d}{x'} \sum_{n \leq x'/d} f^*(n)n^{-it} \Big| \\
 &\qquad+  \frac{1}{x}\sum_{d > w} \sum_{m<x/d} |f'(d)f^*(m)|
  +  \frac{1}{x'}\sum_{d > w} \sum_{m<x'/d} |f'(d)f^*(m)|\\
 &\ll \sum_{d \leq w} \frac{|f'(d)| }{d}
  \frac{1}{(\log x)^{1+C_0}} \exp \Bigg(\sum_{p \leq x} \frac{|f^*(p)|}{p}\Bigg)\\
 &\qquad+  \sum_{m \leq x} \frac{1}{m} \frac{m}{x}\sum_{w <d \leq x/m} |f'(d)|
  +  \sum_{m \leq x'} \frac{1}{m} \frac{m}{x'}\sum_{m \leq x'}\sum_{w <d \leq x'/m} |f'(d)|\\
\end{align*} 
To bound the last two terms, recall that $|f'(n)|\leq 1$ and that (cf.\ \cite[Theorem III.5.1]{Tenen}) 
 \begin{equation} \label{eq:smooth-nmbrs-bd}
  \Psi(z,y) := \#\{n\leq z: p|n \Rightarrow p \leq y\} \ll ze^{-u/2}\qquad (z \geq y \geq 2),
 \end{equation}
 where $u= \log z / \log y$. In particular
 $$\Psi(z,\exp((\log \log x)^2)) \ll ze^{-(\log \log x)^2/4} = z(\log x)^{-(\log \log x)/4}$$
 whenever $z \geq \exp((\log \log x)^4/2)$. 
 Thus, the above is bounded by
 \begin{align*}
 &\ll 
  \frac{1}{(\log x)^{1+C_0}} 
  \exp \Bigg(\sum_{p \leq x} \frac{|f(p)|}{p} + \sum_{p \leq d} \frac{1}{p^2(1-p^{-1})}\Bigg)
  +  \sum_{m \leq x} \frac{1}{m} (\log x)^{-(\log \log x)/2} \\ 
  &\ll 
  \frac{1}{(\log x)^{1+C_0}} \exp \Bigg(\sum_{p \leq x} \frac{|f(p)|}{p}\Bigg),
\end{align*}
which completes the proof.
\end{proof}

\begin{proof}[Proof of Lemma \ref{l:small-chi-f-to-h}]
By Lemma \ref{l:h'-contribution} and \eqref{eq:less-preliminary} it follows that
 \begin{align*}
|S_{f\chi}(y)| 
&\leq  x^{-\delta/8} (\log x)^{O(H)} 
\\ &\quad 
+ 
\sum_{d_0 \leq y^{\delta}} 
\sum_{D \leq (y/d_0)^{1-1/H}}
\sum_{d_1 \dots d_{H-1} = D}
\frac{|h'(d_0)h(d_1) \dots h(d_{H-1})||\chi(d_0D)|}{d_0D} 
\\
\nonumber
&\qquad  \times \sum_{i=1}^H
\frac{Dd_0}{y}
\Bigg(~
\Bigg|
\sum_{\substack{n \leq y/(Dd_0)}} 
h(n) \chi(n) \Bigg|
+
\Bigg|~
\sum_{\substack{n < (y/d_0)^{1-1/H}
\max(d_1, \dots, d_{i-1})/D}} 
h(n) \chi(n) \Bigg| \Bigg).
\end{align*}
As in the proof of Corollary \ref{c:GS}, we set $\delta = 1/(4H)$.
Then the inner sums may be bounded using either the assumption or, if 
$(y/d_0)^{1-1/H} \max(d_1, \dots, d_{i-1})/D<x^{1/(4H)},$ 
by the trivial estimate 
$$
\frac{Dd_0}{y} x^{1/(4H)} \leq y^{1 - 1/H + 1/(4H) - 1 } x^{1/(4H)} = y^{-3/(4H)} x^{1/(4H)} \leq x^{-1/(8H)}.
$$
Bounding the sums over $D$ and $d_0$ as in the proof of Corollary \ref{c:GS},
the lemma follows.
\end{proof}

\begin{proof}[Proof of Lemma \ref{l:lipschitz-h-to-f}]
We first use Lemma \ref{l:h'-contribution} to remove the contribution 
of the function $h'$ defined in Lemma \ref{l:dirichlet}.
Given $\delta \in (0,1)$, let $\delta'$ be such that $x^{\delta} = 
{x'}^{\delta'}$.
Then
\begin{align} \label{eq:h-to-f-1}
& |S_{f\chi}(x) - S_{f\chi}(x')| \\
\nonumber
& \leq x^{-\delta/4} (\log x)^{O(H)}  + \sum_{d_0 \leq x^{\delta}} 
\frac{|h'(d_0) \chi(d_0)|}{d_0}
 \bigg|
\frac{d_0}{x} \sum_{n\leq x/d_0} h^{*H}(n)\chi(n)
- \frac{d_0}{x'} \sum_{n\leq x'/d_0} h^{*H}(n)\chi(n)
 \bigg|
\end{align}
To analyse the difference above, we seek to decompose $h^{*H}$ using $H-1$ 
applications of the hyperbola trick\footnote{This proof requires a different 
decomposition from the one used in \eqref{eq:less-preliminary} and 
\S\ref{ss:decomposition} in order to be able to collect together terms in 
\eqref{eq:hyperbola2} below.},
\begin{align*}
 \sum_{nm \leq Y} 
 = \sum_{n \leq X} 
   \sum_{m \leq Y/n}
  +\sum_{m \leq Y/X} 
   \sum_{X \leq n \leq Y/m}. 
\end{align*}
Fix $d_0$ and let $X = (x'/d_0)^{1/H}$.
If $y \in \{x'/d_0,x/d_0\}$, then applying the hyperbola trick with 
the chosen cut-off $X$ and with $Y=y$, $n=d_1$ and $m=d_2 \dots d_H$, we obtain
\begin{align*}
 \sum_{d_1 \dots d_H \leq y} 
 = \sum_{d_1 \leq X} 
   \sum_{d_2 \dots d_H \leq y/d_1}
  +\sum_{d_2 \dots d_H \leq y/X} 
   \sum_{X \leq d_1 \leq y/(d_2 \dots d_H)}. 
\end{align*}
We keep the second term and decompose the first term again, using the same 
cut-off $X$, and $Y=y/d_1$, $n=d_2$ and $m= d_3 \dots d_H$.
This leads to:
\begin{align*}
 \sum_{d_1 \dots d_H \leq y} 
 &= \sum_{d_1 \leq X}
   \sum_{d_2 \leq X} 
   \sum_{d_3 \dots d_H \leq y/(d_1 d_2)}
  +\sum_{d_1 \leq X}
   \sum_{d_3 \dots d_H \leq y/(d_1X)} 
   \sum_{X \leq d_2 \leq y/(d_1 d_3 \dots d_H)}\\
 &+\sum_{d_2 \dots d_H \leq y/X} 
   \sum_{X \leq d_1 \leq y/(d_2 \dots d_H)} .
\end{align*}
Continuing in this manner, i.e., keeping every time the second new term and 
further decomposing the first, we arrive at
\begin{align} \label{eq:h-to-f-decomposition}
 \sum_{d_1 \dots d_H \leq y} 
 &= \sum_{d_1, d_2 \dots, d_{H-1} \leq X}
   \sum_{d_H \leq y/(d_1 d_2 \dots d_{H-1})} \\
\nonumber   
 &\quad + \sum_{i=1}^{H-1} \sum_{d_1, \dots, d_{i-1} \leq X}
   \sum_{\substack{d_{i+1}, \dots, d_H: \\ d_1 \dots \hat{d_i} \dots d_{i-1}\leq y/X}} 
   \sum_{X \leq d_i \leq y/(d_1 \dots \hat d_i \dots d_H)}.
\end{align}
In order to apply this decomposition to \eqref{eq:h-to-f-1}, let us consider 
the difference of the normalised sums \eqref{eq:h-to-f-decomposition} for 
$y=x/d_0$ and $y=x'/d_0$. 
Recall that that $x \geq x'$.
By splitting the third sum of the second term of the decomposition into two 
sums when $y=x/d_0$, we obtain the following:
\begin{align}\label{eq:hyperbola2}
&\frac{d_0}{x}\sum_{d_1 \dots d_H \leq x/d_0} -  
 \frac{d_0}{x'}\sum_{d_1 \dots d_H \leq x'/d_0} \\
\nonumber 
&= \sum_{d_1, d_2 \dots, d_{H-1} \leq X}
   \left( \frac{d_0}{x} \sum_{d_H \leq x/(d_0 d_1 \dots d_{H-1})} - 
          \frac{d_0}{x'} \sum_{d_H \leq x'/(d_0 d_1 \dots d_H-1)}
   \right)\\
\nonumber
&\quad  + \sum_{i=1}^{H-1} \sum_{d_1, \dots, d_{i-1} \leq X}
    \sum_{\substack{d_{i+1}, \dots, d_H: 
    \\ d_0 \dots \hat{d_i} \dots d_{i-1}\leq x'/X}} \\
\nonumber    
&\qquad   \left( 
    \frac{d_0}{x} \sum_{d_i \leq x/(d_1 \dots \hat d_i \dots d_H)}
    - \frac{d_0}{x'}\sum_{d_i \leq x'/(d_1 \dots \hat d_i \dots d_H)}
    + \frac{d_0}{x'} \sum_{d_i \leq X}
    - \frac{d_0}{x} \sum_{d_i \leq X}
    \right)\\
\nonumber
&\quad  + \sum_{i=1}^{H-1} \sum_{d_1, \dots, d_{i-1} \leq X}
    \sum_{\substack{d_{i+1}, \dots, d_H: \\
    x'/X \leq d_0 \dots \hat d_i \dots d_H \leq x/X}}
    \frac{d_0}{x}
    \sum_{X \leq d_i \leq x/(d_0 \dots \hat d_i \dots d_H)}
\end{align}
When we apply this decomposition with the summation argument
$g(d_1) \dots g(d_H)$, where $g(n)=h(n)\chi(n)$, then the first and the second 
term above contain expressions of the form $S_g(z)-S_g(z')$ for suitable $z$ 
and $z'$. 
These will be estimates using the assumptions of the lemma.
Before turning towards these, let us consider the remaining terms.

The second term contains two short sums up to $X$ that will be 
estimated using Shiu's bound \eqref{eq:shiu} in the following form.
For every fixed $j\in \NN$, every $q \in \NN$ for which the interval
$((W(x) q)^2,x]$ is non-empty and every $y \in ((W(x) q)^2,x]$, we have
\begin{align} \label{eq:shiu-character}
\sum_{n \leq y} |g^{*j}(n)|
= \sum_{A \in (\ZZ/qW(x)\ZZ)^*}
\sum_{\substack{n \leq y \\ n \equiv A \Mod{qW(x)}}} |h^{*j}(n)| 
\ll \frac{x}{\log x} 
 \exp\Big(j \sum_{\substack{p\leq y \\ p\nmid qW(x)}} \frac{|h(p)|}{p} \Big), 
\end{align}
since $W(x)q \leq y^{1/2}$, and thus $W(x)q = W(y)q'$ for some $q' \leq y^{1/2}$.

By applying this bound twice with $W(x)q=Q$, we obtain:
\begin{align*}
& \Big|
 \frac{d_0}{x'} 
 \sum_{d_1, \dots, d_{i-1} \leq X} g(d_1) \dots g(d_{i-1})
 \sum_{\substack{d_{i+1} \dots d_H \leq \\ x'/(d_0 \dots d_{i-1}X)}} 
 g(d_{i+1}) \dots g(d_H)
 \sum_{d_i \leq X} g(d_i) \Big| \\
&\leq \frac{d_0 X}{x'} \frac{1}{\log X} 
 \exp\Big( \sum_{\substack{p\leq X \\ p\nmid Q}} \frac{|h(p)|}{p} \Big)
 \sum_{d_1, \dots, d_{i-1} \leq X} |g(d_1) \dots g(d_{i-1})|
 \sum_{\substack{d \leq x'/(d_0 \dots d_{i-1}X)}} 
 | g^{*(H-i)}(d)| \\
 &\leq \frac{1}{(\log X)^2} 
 \exp\Big( (H-i+1) \sum_{\substack{p\leq x' \\ p\nmid Q}} 
 \frac{|h(p)|}{p} \Big)
 \sum_{d_1, \dots, d_{i-1} \leq X}
 \frac{|g(d_1) \dots g(d_{i-1})|}{d_1 \dots d_{i-1}} \\
 &\ll_{H,C} \frac{1}{(\log x)^2} 
 \exp\Big( \sum_{\substack{p\leq x' \\ p\nmid Q}} 
 \frac{|f(p)|}{p} \Big),
\end{align*}
which saves $(\log x)^{-1}$.

In the final term of \eqref{eq:hyperbola2}, we will take advantage of the fact 
that the third sum is short.
Starting off with another application of \eqref{eq:shiu-character}, we get
\begin{align*}
&\sum_{d_1, \dots, d_{i-1} \leq X} g(d_1) \dots g(d_{i-1})
 \sum_{\substack{d_{i+1}, \dots, d_H: \\
    x'/X \leq d_0 \dots \hat d_i \dots d_H \leq x/X}}
    g(d_{i+1}) \dots g(d_H)
 \frac{d_0}{x}
    \sum_{X \leq d_i \leq x/(d_0 \dots \hat d_i \dots d_H)} g(d_i) \\
& \leq 
 \frac{1}{\log x} 
 \exp\Big( \sum_{\substack{p\leq x' \\ p\nmid Q}} \frac{|h(p)|}{p} \Big)
 \sum_{d_1, \dots, d_{i-1} \leq X} 
  \frac{|g(d_1) \dots g(d_{i-1})|}{d_1 \dots d_{i-1}}
 \sum_{\substack{d_{i+1}, \dots, d_H: \\
    x'/X \leq d_0 \dots \hat d_i \dots d_H \leq x/X}}
  \frac{|g(d_{i+1}) \dots g(d_H)|}{d_{i+1} \dots d_H} \\
& \leq 
 \frac{1}{\log x} 
 \exp\Big( \sum_{\substack{p\leq x' \\ p\nmid Q}} \frac{|h(p)|}{p} \Big)
 \sum_{\substack{d_1, \dots, d_{i-1} \leq X \\ d_{i+1}, \dots d_{H-1} \leq x}}
 \frac{|g(d_1) \dots \widehat{g(d_i)} \dots g(d_{H-1})|}
      {d_1 \dots \hat d_i \dots d_{H-1}}
 \sum_{\substack{d_H: \\
    x'/X \leq d_0 \dots \hat d_i \dots d_H \leq x/X}}
    \frac{1}{d_H} \\
& \leq 
 \frac{1}{\log x} 
 \exp\Big( \sum_{\substack{p\leq x' \\ p\nmid Q}} \frac{|h(p)|}{p} \Big)
 \sum_{\substack{d_1, \dots, d_{i-1} \leq X \\ d_{i+1}, \dots d_{H-1} \leq x}}
    \frac{|g(d_1) \dots \widehat{g(d_i)} \dots g(d_{H-1})|}
      {d_1 \dots \hat d_i \dots d_{H-1}}
    \Big(\log (x/x') + O(1) \Big) \\
& \leq 
 \frac{\log \log X + \log C + O(1)}{\log x} 
 \exp\Big((H-1)\sum_{\substack{p\leq x'\\p\nmid Q}} \frac{|h(p)|}{p}\Big),
\end{align*}
which saves a factor $(\log x)^{- \alpha_f/H + \eps}$.

To summarise our progress so far, note that the decomposition \eqref{eq:hyperbola2} 
and the previous two bounds yield
\begin{align*}
&
 \bigg|
\frac{d_0}{x} \sum_{n\leq x/d_0} h^{*H}(n)\chi(n)
- \frac{d_0}{x'} \sum_{n\leq x'/d_0} h^{*H}(n)\chi(n)
 \bigg| \\
&=
\sum_{\substack{d_1, \dots, d_{H-1}  \leq \\ (x'/d_0)^{1/H}}}
\frac{|g(d_1) \dots g(d_{H-1})|}{d_1 \dots d_{H-1}}
~\bigg| 
S_{g} \Big(\frac{x}{d_0 \dots d_{H-1}}\Big) 
- S_{g}\Big(\frac{x'}{d_0 \dots d_{H-1}}\Big) 
\bigg|\\
&\quad+ \sum_{i=1}^{H-1}
\sum_{\substack{d_1, \dots, d_{i-1}  \leq \\ (x'/d_0)^{1/H}}}
\frac{|g(d_1) \dots g(d_{i-1})|}{d_1 \dots d_{i-1}}
\sum_{\substack{ d_{i+1}, \dots, d_H: \\
 d_1 \dots \hat d_i \dots d_H \leq \\(x'/d_0)^{1-1/H}}} 
\frac{|g(d_{i+1}) \dots g(d_H) |}{d_{i+1} \dots d_H} \times
  \\
&\qquad \qquad \qquad \times
\bigg| 
S_g \Big(\frac{x}{d_0 \dots \hat d_i \dots d_{H}}\Big) 
- S_g\Big(\frac{x'}{d_0 \dots \hat d_i \dots d_{H}}\Big) 
\bigg|\\
&\quad+ O_{H,C} \Bigg(
 (\log x)^{- \min(1,\alpha_f/(2H))}
 \frac{1}{\log x} 
 \exp\Big(\sum_{\substack{p\leq x'\\p\nmid Q}} \frac{|f(p)|}{p}\Big)
\Bigg).
\end{align*}
Choosing $\delta = 1/2$ to ensure that $d_0 \leq x^{1/2}$, it follows that
the terms $x/(d_0 \dots d_{H-1})$ and $x/(d_0 \dots \hat d_i \dots d_{H})$ 
in the above expression are at least as large as $x^{1/(2H)}$.
Since $g=h\chi$, we may thus apply the assumptions of the lemma to deduce 
that the above is bounded by:
\begin{align*}
&\ll \frac{\psi(x)}{\log ((x'/d_0)^{1/H})} 
 \exp\Big(\sum_{\substack{p\leq x'\\p\nmid Q}} \frac{|h(p)|}{p}\Big)
 \Bigg(\prod_{d \leq x} \frac{g(d)}{d}\Bigg)^{H-1} \\
&\qquad +
 (\log x)^{- \min(1,\alpha_f/(2H))}
 \frac{1}{\log x} 
 \exp\Big(\sum_{\substack{p\leq x'\\p\nmid Q}} \frac{|f(p)|}{p}\Big)\\
&\ll \Big( \psi(x) + (\log x)^{- \min(1,\alpha_f/(2H))} \Big) 
 \frac{1}{\log x} 
 \exp\Big(\sum_{\substack{p\leq x'\\p\nmid Q}} \frac{|f(p)|}{p}\Big)
\end{align*}
The lemma then follows from \eqref{eq:h-to-f-1} since by \eqref{eq:d_0-sum}, 
applied with $y=1$ and $w=w(x)$, the completed outer sum over $d_0$ converges, 
i.e. $ \sum_{d_0 =1}^{\infty} |h'(d_0) \chi(d_0)|/d_0 < \infty$, provided $x$ 
is sufficiently large for $w(x)\geq(2H)^{16}$ to hold.
\end{proof}

\subsection{Applications to functions bounded away from zero at primes}
\label{ss:non-negative}
In this subsection, we will discuss a concrete example of an element of $\mathcal{F}_H$ and 
prove a criterion for real-valued $f$ to belong to $\mathcal{F}_H$ that is just based on the values of $f$ at primes.
Let us begin by stating a special case of Proposition \ref{p:sufficient-2} for non-negative $f \in \mathcal{M}_H$.
\begin{lemma}[Sufficient condition for non-negative functions]
\label{l:lipschitz-non-negative}
Let $f \in \mathcal{M}_H$ be a non-negative function.
Then there exists a constant $c>0$, only depending on $f$, such that the following holds:
If $x>3$, if $1 < Q \leq \exp((\log \log x)^2)$ is a multiple of $W(x)$, and if 
$\chi_0 \Mod{Q}$ denotes the trivial character, then 
 $$S_{f\chi_0}(x) = S_{f\chi_0}(x')
 + O\Bigg( (\log x)^{-c}
\frac{1}{\log x}
 \prod_{\substack{p \leq x,\, p\nmid Q}} 
 \left(1 + \frac{|f(p)|}{p}\right)
\Bigg)
$$
uniformly for all $x>3$, $x' \in [x \exp( - (\log \log x)^{2},x]$ and all $Q$ as above.
If, furthermore, for either $g=h$ or $g=f$ and for any $C>0$, we have a uniform bound
of the form
\begin{align}\label{eq:non-neg-condition}
S_{g \chi}(x) = 
O\Bigg(
\frac{\psi_C(x)}{\log x}
 \prod_{\substack{p \leq x, \, p\nmid Q}} 
 \left(1 + \frac{|g(p)|}{p}\right)
\Bigg),
\end{align}
valid for all $x>3$, all non-trivial $\chi \Mod{Q}$ and all $1 \leq Q \leq (\log x)^C$ with $W(x)|Q$, 
and where $\psi_C=o(1)$ may depend on $C$ but is otherwise independent of $\chi$ and $Q$,
then $f \in \mathcal{F}_H$.
\end{lemma}

\begin{rem} \label{r:lipschitz-nn}
Note that in the context of this corollary, the main term in the character 
sum expansion of $S_f(x;Q,A)$ always comes from the trivial character.
\end{rem}

\begin{proof}
The first part follows from Lemmas \ref{l:lipschitz} and 
\ref{l:lipschitz-h-to-f} 
provided we can show that for all sufficiently large $x$ we have $t_{x,h\chi_0}=0$ in the statement of 
Lemma \ref{l:lipschitz} when applied with $f$ replaced by $h\chi_0$.
This, however, is immediate since $h$ is non-negative.
The second part is a consequence of Proposition \ref{p:sufficient-2}.
\end{proof}

The following three lemmas all arise as (non-trivial) applications of Lemma~\ref{l:lipschitz-non-negative}.

\begin{lemma}[Coefficients of cusp forms] \label{l:appl-1}
 Let $f$ be a primitive holomorphic cusp form\footnote{See \cite[\S14.1 and \S14.7]{IK} for definitions.} 
 of weight $k\in 2\NN$ and level $N \in \NN$ and let
$$
f(z)=
\sum_{n=1}^{\infty}
\lambda_f(n) n^{(k-1)/2}e(nz), 
$$
be its Fourier expansion, where the $\lambda_f(n)$ are the 
normalised Fourier coefficients.
Then the function $n \mapsto |\lambda_f(n)|$ belongs to $\mathcal{F}_2$.
\end{lemma}

\begin{lemma}[Non-negative $f$] \label{l:appl-2}
For every $H \geq 1$ and $\alpha > 0$, there exists $c = c(H,\alpha) > 0$ such 
that the following holds.
If $f \in \mathcal{M}_H$ is non-negative with $\alpha_f \geq \alpha$, and if 
there exists $\delta > 0$ such that
$$
\#\Big\{ p \leq x: f(p) > \delta \Big\} 
\geq \frac{ (1-c) x}{\log x}
$$
for all sufficiently large $x$, then $f \in \mathcal{F}_H$.
\end{lemma}

\begin{rem*}
As a special case, Lemma \ref{l:appl-2} yields the following simple criterion, 
which also proves one part of Proposition \ref{p:intro}: 

A non-negative function $f \in \mathcal{M}_H$ belongs to $\mathcal{F}_H$ if it 
is bounded away from zero on the primes, i.e.\ if there exists $\delta >0$ such 
that $f(p)>\delta$ for all $p$.
The same holds true if the latter condition is replaced by
$\#\{ p \leq x: f(p) > \delta \} \geq (1 + o(1)) x/\log x$ as $x\to \infty$.
\end{rem*}

The following variant of Lemma \ref{l:appl-2} will follow with minor changes in the proof.

\begin{lemma}[Real-valued $f$] \label{l:appl-3}
 For every $H \geq 1$ and $\alpha > 0$, there exists $c = c(H,\alpha) > 0$ such 
that the following holds.
If $f \in \mathcal{M}_H$ is a real-valued function with $\alpha_f \geq \alpha$, and if there exists $\delta > 0$ 
and a sign $\epsilon \in \{+,-\}$ such that
$$
\#\Big\{ p \leq x: \epsilon f(p) > \delta \Big\} 
\geq \frac{ (1-c) x}{\log x}
$$
for all sufficiently large $x$, then $f \in \mathcal{F}_{H,n^{it}}$.
If, furthermore, for every $C>0$ there exists a function $\psi_C =o(1)$ such that
$$S_{f \chi_0}(x) = 
O\Bigg( 
\frac{\psi_C(x)}{\log x}
 \prod_{p \leq x, p \nmid Q} 
 \left(1 + \frac{|f(p)|}{p}\right)
\Bigg),
$$
whenever $\chi_0$ is the trivial character modulo $Q$ 
for any $Q \in (1, (\log x)^C)$ with $W(x)|Q$, 
then $f \in \mathcal{F}_{H}$.
\end{lemma}

\begin{rem*}
As a particular consequence, we deduce that $f \in \mathcal{F}_{H,n^{it}}$ for any function 
$f \in \mathcal{M}_H$ for which there exists $\delta >0$ such that  
$f(p)< - \delta < 0$ at all primes $p$.
\end{rem*}

\begin{example}\label{ex:mobius-4.4}
Examples of functions the above results apply to include:
\begin{enumerate}
 \item[(i)] the M\"obius function $f(n) = \mu(n)$. Here, the full statement of 
 Lemma \ref{l:appl-3} applies. We may deduce this from the estimate 
 $S_{\mu}(x) \ll_B (\log x)^{-B}$ for $B>0$ and $x \geq 2$. 
 In fact, writing $d|Q^{\infty}$ to indicate that $p|d$ implies $p|Q$, it follows via repeated M\"obius inversion that
 $$  \sum_{\substack{n \leq x\\ (n,Q)=1}} \mu(n) 
  = \sum_{d| Q^{\infty} }\sum_{n \leq x/d} \mu(n)$$
 Recalling \eqref{eq:smooth-nmbrs-bd}, the above is seen to be bounded by
  \begin{align*}
  &\ll \sum_{\substack{d| Q^{\infty}\\ d \leq x^{1/2} }}\sum_{n \leq x/d} \mu(n)
  + \sum_{x^{1/2} < 2^k \leq x}  \Psi(2^k, (\log x)^C) x 2^{-k}\\
  & \ll_B \sum_{\substack{d| Q^{\infty}\\ d \leq x^{1/2} }} \frac{x}{d} (\log x)^{-B} 
  + \sum_{ x^{1/2} < 2^k \leq x} x  \exp \Big(- \frac{\log x }{4 C \log \log x}\Big)  \\
  & \ll_B x (\log x)^{-B} \prod_{p|Q} (1 - p^{-1})^{-1} 
  \ll_B x (\log x)^{-B + 1},
 \end{align*}
 which yields the required decay estimate.
 \item[(ii)] the function $f(n)= \delta^{\omega(n)}$ for any non-zero real number $\delta$ and  
where $\omega(n)$ counts the number of distinct prime factors of $n$.
If $\delta >0$, then Lemma \ref{l:appl-2} applies.
For $\delta <0$ we will now show that the full statement of Lemma \ref{l:appl-3} applies.
Since $W(x)|Q$, we may simplify our task by removing finitely many primes from consideration to start with:
let $A\geq|\delta|$ be a constant to be chosen later, let $Q_0 = \prod_{p>A} p^{v_p(Q)}$ and let
$h(n) = \delta^{\omega(n)}\1_{\gcd(n,\prod_{p\leq A}p)=1}$ denote the restriciton of 
$f$ to integers free from primes factors $p\leq A$. 
For this function, the Selberg-Delange method in the form \cite[Theorem 7.18]{MV-book} implies
$S_{h}(x) \ll (\log x)^{\delta - 1}$ and 
$S_{|h|}(x) \asymp (\log x)^{|\delta| - 1}$ for all $x \geq 2$, while
Lemma \ref{l:elliott} and Shiu's lemma in its original form \cite[Theorem 1]{shiu} yield
$S_{|h|}(x)\asymp \frac{1}{\log x} \prod_{p \leq x} (1 + \frac{|\delta|}{p})$.
Proceeding in a similar way as in (i), repeated M\"obius inversion shows that
\begin{align} \label{eq:delta-bound-1}
\nonumber
  \sum_{\substack{n \leq x\\ (n,Q)=1}} \delta^{\omega(n)}
  &= \sum_{k\geq 1} \sum_{\substack{ d_1| Q_0^{\infty} \\ d_1 \geq 1 }} \sum_{\substack{d_2| d_1^{\infty} \\ d_2 >1}} 
     \dots \sum_{\substack{d_k| d_{k-1}^{\infty} \\ d_k>1 }}
  |\delta|^{\omega(d_1) + \dots + \omega(d_k)} \sum_{\substack{n \leq x/d \\ p|n \Rightarrow p>A}} \delta^{\omega(n)} \\
  &\leq  \sum_{d| Q_0^{\infty} } |\delta|^{\varpi(d)} \wp_m(d) \Big| \sum_{n \leq x/d} h(n) \Big|,
\end{align}  
where $\varpi(d) = \omega(d)$ if $|\delta|<1$ and $\varpi(d) = \Omega(d)$ if $|\delta|\geq1$, and where
$\wp_m$ counts factorisations of the following form:
$$\wp_m(d) = \# \left\{(d_1 \dots, d_k) \in \NN_{>1}^k, k\geq 1 : 
\begin{array}{c}
d=d_1 \dots d_k, \cr
p | d_j \text{ for some } 1 \leq j \leq k \cr \Rightarrow p| d_i \text{ for all } i < j 
\end{array}
\right\}.
$$
If $\wp(n)$ denotes the partition number of $n$ as defined in \cite{HardyRamanujan}, then 
$$\wp_m(d) \leq \prod_{p|d} \wp(\nu_p(d)).$$
To bound \eqref{eq:delta-bound-1}, we will use the fact that there exists a constant $B \geq 1$ such that 
$\wp(n) \leq B^{\sqrt{n}}$, as proved in \cite[\S2]{HardyRamanujan}. 
Further, we require a bound corresponding to the one recalled in
\eqref{eq:smooth-nmbrs-bd} but for sums over $|\delta|^{\varpi(n)}\wp_m(n)$ restricted to smooth numbers
that are coprime to all $p\leq A$.
For such sums we have
\begin{align*}
\Psi^*(x,y) &:= \sum_{\substack{n\leq x\\ p|n \Rightarrow  p \in (A,y]}} |\delta|^{\varpi(n)} \wp_m(n) \\
 &\leq C_{\eps} x^{1/2 + \eps} 
 + \sum_{\substack{n\leq x \\  p|n \Rightarrow  p \in (A,y] }} (n x^{-1/2})^{\alpha} 
   \prod_{p|n} |\delta|^{\varpi(p^{v_p(n)})} B^{\sqrt{v_p(n)}} 
\end{align*}
for any $\alpha >0$.
Let $\alpha = (\log y)^{-1}$ and suppose that $y \geq 2$.
By \cite[Cor.\ III.3.5.1]{Tenen}, the final sum in the expression above is bounded by
\begin{align*}
 & \ll x^{1-\alpha/2} \prod_{A<p \leq y} (1-p^{-1}) \sum_{k\geq 0} \frac{p^{\alpha k} |\delta|^{\varpi(p^k)} B^{k}}{p^k} \\
 & \ll x^{1-\alpha/2} \prod_{A<p \leq y} (1-p^{-1})  (1-e|\delta|Bp^{-1})^{-1}
 \ll x^{1-\alpha/2} (\log y)^{O(1)},
\end{align*}
provided $A \geq e B \max(1, |\delta|) $.
Thus, in total, we obtain:
\begin{align*}
 \Psi^*(x,y) &\ll x^{1-\alpha/2} (\log y)^{O(1)} 
 = x \exp\Big(-\frac{\log x}{2 \log y} + O(1) \log \log y \Big)\\
 &\ll x \exp\Big(-\frac{\log x}{4 \log y} \Big), \qquad (2 \leq y \leq x, \quad A \geq e B \max(1, |\delta|) ).
\end{align*}

Returning to \eqref{eq:delta-bound-1}, we choose $A= e B \max(1, |\delta|) $ in the definition of
$h$ and $Q_0$, and recall that $Q_0 \leq Q \leq (\log x)^C$.
With the above bound on $\Psi^*(x,y)$ in place, the expression \eqref{eq:delta-bound-1} can now be bounded by:
\begin{align*}  
  & \ll \sum_{\substack{d| Q_0^{\infty}\\ d \leq x^{1/2} }} |\delta|^{\varpi(d)} \wp_m(d) 
    \sum_{n \leq x/d} \delta^{\omega(n)}
  + \sum_{x^{1/2} < 2^k \leq x}  \Psi^*(2^k, (\log x)^C) x 2^{-k} (\log (x 2^{-k}))^{|\delta| -1} \\
  & \ll \sum_{\substack{d| Q_0^{\infty}\\ d \leq x^{1/2} }} x\frac{|\delta|^{\varpi(d)} \wp_m(d) }{d} (\log x)^{\delta - 1} 
  + \sum_{ x^{1/2} < 2^k \leq x} x  \exp \Big(- \frac{\log x }{4 C \log \log x}\Big) (\log x)^{|\delta| } \\
  & \ll x (\log x)^{\delta-1} 
    \prod_{p|Q_0} \sum_{k \geq 0} 
    \frac{|\delta|^{\varpi(p^k)} B^{\sqrt{k}}}{p^k}
    +O_A(x(\log x)^{-A}),
\end{align*}
which is further bounded by
\begin{align*}
  & \ll x (\log x)^{\delta -1}
    \prod_{\substack{p|Q \\ p> \max(|\delta|,B)~ }} \sum_{k \geq 0 } 
    \frac{ |\delta|^{\varpi(p^k)} B^{k}}{p^k}
    +O_A(x(\log x)^{-A})  \\
  & \ll x (\log x)^{\delta -1} (\log Q)^{B|\delta|} 
   \ll \frac{(C \log \log x)^{O(|\delta|)}}{(\log x)^{2|\delta|}} 
    \frac{x}{\log x} \prod_{\substack{p \leq x,\, p \nmid Q}} \Big( 1 + \frac{|\delta|}{p} \Big).
 \end{align*}
 Thus, the required decay estimate holds.
 \item[(iii)] the general divisor functions 
 $d_k(n) = \mathbf{1}^{(*k)}(n)$ for $k \geq 2$, i.e.\ the $k$-fold convolution of $\1$ with 
 itself. In this case Lemma \ref{l:appl-2} applies.
\end{enumerate}
\end{example}

The remainder of this subsection contains the proofs of Lemmas \ref{l:appl-1}--\ref{l:appl-3}.
We begin with the proof of Lemma \ref{l:appl-1}, which is the least technical case.
Lemma \ref{l:appl-2} and \ref{l:appl-3} will follow with small modifications from 
the same proof.

\begin{proof}[Proof of Lemma \ref{l:appl-1}]
 The function $\lambda_f$ that describes the normalised Fourier coefficients 
of $f$ is a multiplicative function and satisfies Deligne's bound
$$
|\lambda_f(n)| \leq d(n),
$$
where $d$ is the divisor function.
This shows that part (1) of Definition \ref{def:M} holds with 
$H=2$.
Condition (2) of the definition follows from Rankin \cite[Theorem 2]{rankin}, 
since
$$
\sum_{p \leq x}|\lambda_f(p)| \log p \geq 
\frac{1}{2}
\sum_{p \leq x}\lambda_f(p)^2 \log p
\sim \frac{x}{2},
$$
which allows us to take $\alpha_{\lambda_f} = 1/2 - \eps$ for any $\eps > 0$. 
Hence, $g=|\lambda_f|$ belongs to $\mathcal{M}_2$.

To show that $g \in \mathcal{F}_2$, let $h$ be the bounded multiplicative 
functions defined, as in Lemma \ref{l:dirichlet}, by
\begin{equation*}
h(p^k) 
=
\begin{cases}
|\lambda_f(p)|/2 & \mbox{if $k=1$}\\
0      & \mbox{if $k>1$},
\end{cases} 
\end{equation*}
and note that, by Lemma \ref{l:lipschitz-non-negative}, it suffices to show that
\begin{equation} \label{eq:lambda-task}
|S_{h \chi}(x)| = 
o\Bigg( \frac{1}{\log x}
 \prod_{\substack{p \leq x \\ p\nmid qW(x)}} 
 \left(1 + \frac{|h(p)|}{p}\right) \Bigg) 
\end{equation}
for all non-trivial $\chi \Mod{Q}$ with $Q \leq (\log x)^C$ and $W(x)|Q$.
We begin this task by invoking Hal\'asz's theorem.
Since $g$ is bounded, the Hal\'asz--Granville--Soundararajan bound 
\cite[Corollary 1]{GS-decay} implies that
\begin{align} \label{eq:Halasz-G-S}
|S_{h \chi}(x)|
= \frac{1}{x}\left|\sum_{n\leq x} \chi(n) h(n)\right|
 \ll (M + 1) e^{- M} + \frac{1}{Y} + \frac{\log \log x}{\log x},
\end{align}
where 
$$
M = M(x,Y) 
= \min_{|y| \leq 2Y} \sum_{p \leq x} \frac{1 - \Re( h(p)\chi(p) p^{iy})}{p}.
$$
Note that
\begin{align}\label{eq:splitting-the-distance}
\nonumber
M(x,Y) 
&= \min_{|y| \leq 2Y} \sum_{p \leq x} \frac{1 - h(p) + h(p) - \Re( h(p)\chi(p) 
p^{iy})}{p} \\
&= \sum_{p \leq x} \frac{1 - h(p)}{p}
+ \min_{|y| \leq 2Y} \sum_{p \leq x} \frac{h(p)(1 - \Re(\chi(p) p^{iy}))}{p}.
\end{align}
and let 
$$M_{h\chi}(x,Y) 
= \min_{|y| \leq 2Y} \sum_{p \leq x} \frac{h(p)(1 - \Re(\chi(p) p^{iy}))}{p}$$ 
denote the second term from this expression.

Observe that the product in the bound \eqref{eq:lambda-task} satisfies
\begin{align} \label{eq:lambda-triv}
\prod_{\substack{p\leq x\\ p \nmid qW(x)}} 
\bigg(1+\frac{|h(p)|}{p}\bigg)
&\gg \exp \bigg( \sum_{\substack{(\log x)^{C+2} <p\leq x}} 
\frac{|h(p)|}{p} \bigg) \\
\nonumber
&\gg_{\eps} (\log x)^{-\eps} \exp\left( \sum_{p\leq x} \frac{|h(p)|}{p} \right)
\gg_{\eps} (\log x)^{ \alpha_h - \eps}
\end{align} 
with $\alpha_h=\alpha_g/H = \alpha_g/2$.
Thus, if we let $Y = (\log x)^{1-\alpha_h/2}$, 
then the last two terms in \eqref{eq:Halasz-G-S} are negligible compared with 
the bound \eqref{eq:lambda-task}. Combining \eqref{eq:Halasz-G-S}, 
\eqref{eq:splitting-the-distance} and \eqref{eq:lambda-triv} it 
follows that 
\begin{align} \label{eq:lambda-first-comp}
\nonumber
 |S_{h\chi}(x)| 
 &\ll (1+M) e^{-M_{h\chi}(x,Y)} 
 \exp \left( \sum_{p\leq x} \frac{|h(p)|-1}{p} \right)
 + (\log x)^{-1+\alpha_h/2}\\
\nonumber
 & \ll \frac{\log \log x }{\log x}
  e^{-M_{h\chi}(x,Y)} 
  \exp \left( \sum_{p\leq x} \frac{|h(p)|}{p} \right)
  + (\log x)^{-1+\alpha_h/2} \\
 &\ll_{\eps} 
  \frac{ (\log x)^{\eps} e^{-M_{h\chi}(x,Y)}}{\log x}
 \prod_{\substack{p \leq x \\ p\nmid qW(x)}} 
 \left(1 + \frac{|h(p)|}{p}\right)
 + (\log x)^{-1+\alpha_h/2}.
\end{align}
This reduces our task to that of finding a sufficiently good lower bound 
on $M_{h\chi}(x,Y)$.
To achieve this, we aim to show that there are positive constants 
$\delta_0, \delta_1, \delta_2 > 0$ such that for all non-trivial 
$\chi \Mod{Q}$ with $Q \leq (\log x)^C$ and $W(x)|Q$, for all 
$0\leq t \leq 2 Y$ and for all
$y \in (\exp((\log x)^{1-\alpha_h/4}), x]$, the set
\begin{align}\label{eq:dense-primes-subset}
\mathcal{P}_{\delta_1,\delta_2}(y) =
 \Big\{ p \leq y: h(p) > \delta_1 \Big\} 
 \cap \Big\{ p \leq y: 1 - \Re (\chi(p) p^{it}) > \delta_2 \Big\}
\end{align}
has positive relative density at least $\delta_0$ in the set of primes up to 
$y$, i.e.\
\begin{align}\label{eq:dense-primes-subset-bound}
\#\mathcal{P}_{\delta_1,\delta_2}(y) \geq  \frac{\delta_0 y}{\log y}.
\end{align}
The restriction to non-negative $t$ is justified here since we consider 
together with every non-trivial $\chi \Mod{qW(x)}$ also its conjugate character 
$\bar \chi$.

Assuming \eqref{eq:dense-primes-subset-bound} for the moment, we then have
\begin{align*}
 \sum_{p \in \mathcal{P}_{\delta_1,\delta_2}(x)} \frac{1}{p}
 &\geq \frac{\#\mathcal{P}_{\delta_1,\delta_2}(x)}{x} 
 + \int_2^x \frac{\#\mathcal{P}_{\delta_1,\delta_2}(t)}{t^2} dt \\
 &\geq \frac{\delta_0}{\log x} 
 + \delta_0 
   \int_{\exp((\log x)^{1-\alpha_h/4})}^x 
   \frac{dt}{t \log t} dt 
  \geq \frac{\delta_0 \alpha_h}{4} \log \log x,
\end{align*}
and, hence,
\begin{align*}
 e^{M_{h\chi}(x,Y)} 
 \gg \exp  \Bigg(\delta_1\delta_2 
    \sum_{p \in \mathcal{P}_{\delta_1,\delta_2}(x)} \frac{1}{p} \Bigg) 
 \gg (\log x)^{\delta_0 \delta_1 \delta_2 \alpha_h/4}.
\end{align*}
Combined with \eqref{eq:lambda-first-comp}, this shows, in particular, that
$$
 |S_{g\chi}(x)| 
 \ll_{\eps} (\log x)^{-\delta_0 \delta_1 \delta_2 \alpha_h/4 + \eps}
 \frac{1}{\log x} 
 \prod_{\substack{p \leq x \\ p\nmid Q}} 
 \left(1 + \frac{|g(p)|}{p}\right) + (\log x)^{-1+\alpha_h/2},
$$
and, hence, that \eqref{eq:lambda-task} holds.
Thus, it remains to establish \eqref{eq:dense-primes-subset-bound}.

The set of primes $\mathcal{P}_{\delta_1,\delta_2}(y)$ is determined by two 
conditions involving the behaviour of $h$, $\chi$ and $n^{it}$ at these primes.
To find a lower bound on the cardinality of 
$\mathcal{P}_{\delta_1,\delta_2}(y)$, our first step is to remove the condition 
that $h(p)>\delta_1$ from consideration.
To do so, recall that the Sato--Tate law \cite{BGHT} implies that 
$$
\#\{p \leq y: 0 \leq |\lambda_p| \leq \alpha \} 
\sim \frac{\mu(\alpha) y}{\log y}
$$
for every $\alpha \in [0,2]$, where
\begin{align*}
\mu(\alpha) = 
\Big( 2
\arcsin(\alpha/2) +
\sin(2\arcsin(\alpha/2)) \Big)/\pi.
\end{align*}
This shows, in particular, that for every $c_1 \in (0,1)$ there exists  a 
$\delta(c_1)>0$ such that 
\begin{equation} \label{eq:primes-g-lower-bd}
\#\Big\{ p \leq y: g(p) > \delta(c_1) \Big\} \geq  \frac{c_1y}{\log y} 
\end{equation}
for all sufficiently large $y$.
Thus, to prove \eqref{eq:dense-primes-subset-bound} for $\delta_2=1/12$, say, 
it suffices to show that for every $0 \leq t \leq 2Y$ and 
every $y \in (\exp((\log x)^{1-\alpha_h/4}), x]$, the set
\begin{align}\label{eq:prime-subset-pretentious}
\mathcal{P}_{\chi,t}(y):= \Big\{ p \leq y: \Re (\chi(p) p^{it}) < 11/12  \Big\}
\end{align}
has positive relative density in the set of primes up to $y$. 
Indeed, if 
\begin{align}\label{eq:prime-subset-pretentious-bound}
\# \mathcal{P}_{\chi,t}(y) \geq \frac{c_2 y}{\log y}
\end{align}
for some $c_2>0$, then, setting $c_1= 1 - c_2/2$ in 
\eqref{eq:primes-g-lower-bd} and letting $\delta_1 = \delta(c_1)$, we find that 
$\#\mathcal{P}_{\delta_1,\delta_2}(y) \geq c_2 y/(2 \log y)$, i.e. 
that \eqref{eq:dense-primes-subset-bound} holds with $\delta_0= c_2/2 > 0$, as 
required\footnote{
In view of the reduction to \eqref{eq:prime-subset-pretentious-bound}, 
it becomes clear that we will only require \eqref{eq:primes-g-lower-bd} to hold for one specific value of $c_1$ in the end.
This will later allow us to deduce Lemma \ref{l:appl-2} from this proof and, with some further modifications, also 
Lemma \ref{l:appl-3}. 
For this reason we will track the information gathered on $c_2$ throughout the rest of the proof.}.

Having simplified our problem to that of establishing \eqref{eq:prime-subset-pretentious-bound}
for a set of primes only defined 
by the behaviour of $\chi(p)$ and $p^{it}$, our next step is to also 
remove $\chi$ from consideration and to essentially turn the problem into a 
question about the distribution of $(\frac{t}{2\pi} \log p)_{p \leq y}$ 
modulo one.
Let us begin by decomposing the set of primes into 
classes on which $\chi(p)$ is constant and consider the primes in each 
progression $A \Mod{Q}$ for $\gcd(A,Q)=1$ separately.
Let $\{z\} = z - \lfloor z \rfloor$ denote the fractional part 
of a real number $z$, let $T=t/(2\pi)$ and consider for each $A$ as above 
the set 
\begin{align}\label{eq:N_A}
\mathcal{N}_A(y) 
&= \Big\{ p < y: \{ T \log p \} \in I_{T\log y}
\text{ and } p \equiv A \Mod{Q} \Big\}
\end{align}
where $I_{T \log y} = [T\log y-1/9,T\log y] \Mod{1}$ is an interval of fixed 
length $1/9$, the position of which only depends on the parameters $y$ and $t$, 
but not on the residue class $A$.
Our aim is to show that there exists a constant $c_3>0$ such that for every
reduced residue class $A \Mod{Q}$ and every $y \in (\exp((\log x)^{1-\alpha_h/4}), x]$, we have
\begin{equation} \label{eq:N_A-claim}
 \#\mathcal{N}_A(y) \geq \frac{c_3y}{\phi(Q)\log y}.
\end{equation}
Since this bound clearly holds for all invertible residue classes if $t=T=0$ 
and if $c_3=1-\eps$, $\eps > 0$, we may restrict attention to the case 
$t \in (0,2Y]$ below.

Assuming \eqref{eq:N_A-claim} for the moment, let us first show how to deduce the claimed 
bound \eqref{eq:prime-subset-pretentious-bound}.
In view of \eqref{eq:N_A-claim}, it suffices to show that for
a positive proportion of the reduced residues $A \Mod{Q}$ we have
$\mathcal{N}_A(y) \subset \mathcal{P}_{\chi,t}(y)$ for all
$y \in (\exp((\log x)^{1-\alpha_h/4}), x]$.

If $\chi$ is a non-trivial real character, then each of the pre-images 
$\chi^{-1}(1)$ and $\chi^{-1}(-1)$ contains $\phi(Q)/2$ residue classes 
$A \Mod{Q}$.
If the distance of $T \log y$ to the closest integer satisfies 
$\|T \log y\| > \frac{1}{6}$, then we have
$\Re e(z) < \cos(2\pi/6) + 1/9 = 1/2 + 1/9 < 3/4$ for every $z \in I_{T\log y}$,
and, hence,
$$
\Re(\chi(p) p^{it})
=\Re(e^{it\log p}) 
< 3/4 < 11/12
$$
for all $p \in \mathcal{N}_A(y)$ and all $\phi(Q)/2$ classes $A \in \chi^{-1}(1)$.
If, on the other hand, $\|T \log y\| \leq 1/6$, then we have
$\Re e(z) \geq \cos(2\pi/6) - 1/9 = 1/2 - 1/9 > 0$ for every 
$z \in I_{T\log y}$,
and, thus,
$$
\Re(\chi(p) p^{it})
= - \Re(e^{it\log p}) < 0 < 11/12 
$$
for all $p \in \mathcal{N}_A(y)$ and all $\phi(Q)/2$ classes $A \in \chi^{-1}(-1)$.
Thus, \eqref{eq:prime-subset-pretentious-bound} holds with
$c_2 = c_3/2$ if $\chi$ is non-trivial and real.

Turning toward the case where $\chi$ is not real, recall that the non-zero 
values of any Dirichlet character $\chi \Mod{Q}$ are the $k$-th roots of unity 
if $\chi$ has order $k$ in the group of characters modulo $Q$ and recall also 
that $\chi(A)$ assumes each $k$-th root of unity equally often as $A$ runs over 
the reduced residue classes modulo $Q$.
Thus, if $\chi \Mod{Q}$ is not a real character, then each of 
the four sets
\begin{align*}
\mathcal{R}_{>} = \{A \Mod{Q}: \Re (\chi(A)) > 0 \}, \qquad 
\mathcal{R}_{<} = \{A \Mod{Q}: \Re (\chi(A)) < 0 \},\\
\mathcal{I}_{>} = \{A \Mod{Q}: \Im (\chi(A)) > 0 \}, \qquad
\mathcal{I}_{<} = \{A \Mod{Q}: \Im (\chi(A)) < 0 \}~
\end{align*}
is non-empty and contains a positive proportion of the reduced residues $A 
\Mod{Q}$. 
To see this, note that $k \geq 3$, since $\chi$ is not real.
By the symmetry of the set of $k$-th roots of unity, we have
$\# \mathcal{I}_{<} = \# \mathcal{I}_{>}$ and, if $i$ is a $k$-th root of 
unity, then $\# \mathcal{R}_{<} = \# \mathcal{R}_{>}$ as well.
If $i$ is not a $k$-th root of unity, then
$|\# \mathcal{R}_{<} - \# \mathcal{R}_{>}| \leq \phi(Q)/k$. 
Since $\mathcal{I}_{<} \cup \mathcal{I}_{>}$ and 
$\mathcal{R}_{<} \cup \mathcal{R}_{>}$ both excludes at most two of the $k$ 
$k$-th roots and since the latter set excludes none if $i$ is not a 
$k$-th root, we have 
$$\# \mathcal{S} \geq \frac{k-2}{2} 
\frac{\phi(Q)}{k} 
= \Big(\frac{1}{2} - \frac{1}{k} \Big) \phi(Q) \geq \frac{\phi(Q)}{6}$$
for each set 
$\mathcal{S}\in \{ \mathcal{R}_>, \mathcal{R}_<,\mathcal{I}_>, \mathcal{I}_<\}$.
This proves the claim.

For each of the above sets $\mathcal{S}$, the product set
$$
\{ \chi(A) e^{2 \pi i \tau}: 
   A \in \mathcal{S}, 
   \tau \in [T\log y-\frac{1}{9},T\log y] \}
$$
is contained in an arc of length $2 \pi (1/2+1/9)$ on the unit circle 
and a rotation by $\pi/2$ maps each of these four arcs onto another one of them.
This configuration has the property that no arc of length at most $\pi/4$ meets 
more than three of the product sets.
Thus, for each choice of the endpoint $T \log y$ there is one set $\mathcal{S}$ 
for which the above product set avoids $\{e^{iz}: \| z \| < \pi/8\}$, and
for that particular set $\mathcal{S}$, we then have
$$
  \Re(\chi(p) p^{it})
= \Re(\chi(A)e^{it\log p})  
< \cos(\pi/8)
= \frac{1}{2} \sqrt{2+\sqrt{2}}
< \frac{11}{12}
$$
for all $A \in \mathcal{S}$ and all $p \in \mathcal{N}_A(y)$. 
Thus, \eqref{eq:prime-subset-pretentious-bound} holds with $c_2=c_3/6$.
This completes the proof of the claim that 
\eqref{eq:N_A-claim} implies \eqref{eq:prime-subset-pretentious-bound}.

It finally remains to analyse the set $\mathcal{N}_A(y)$ that was defined in
\eqref{eq:N_A} and we will do this by borrowing an approach from Wintner's 
work \cite{wintner} on the distribution of $(\log p_n)_{n \leq x}$ modulo one.

Let us fix a reduced residue class $A \Mod{Q}$ and let 
$(p_n^{(A)})_{n \in \NN}$ denote the sequence of primes congruent to $A 
\Mod{Q}$, ordered in increasing order.
Adapting the notation from \cite{wintner} to our setting, let $N_A(\tau)$ 
denote the largest index $m$ for which $\log p_m^{(A)} < \tau$, if such an $m$ 
exists, and let $N_A(\tau) =0$ otherwise.
By the prime number theorem in arithmetic progressions, we then have
\begin{equation}\label{eq:PNTAP}
N_A(\tau)=\frac{e^\tau}{\phi(Q)\tau}\left(1 + O(\tau^{-1}) \right),
\quad (\tau > 0).
\end{equation}
Observe that $N_A(\tau/T)$ counts the number of $m>0$ such that 
$T \log p_m \leq \tau$.
Thus, if we set $\xi := \{T\log y\}$, so that $T\log y = [T\log y] + \xi$, 
then, in analogy to \cite[eq.\ (3)]{wintner}, we may express 
the quantity $\# \mathcal{N}_A(y)$ as
\begin{align} \label{eq:N_A(y)}
 \nonumber
 \# \mathcal{N}_A(y)
 &= \sum_{n=1}^{[T\log y]} 
 \left( N_A\left(\frac{n+\xi}{T}\right) 
 -  N_A\left(\frac{n+\xi - 1/9}{T}\right) 
 \right)\\
 &= \sum_{n=T}^{[T\log y]} N_A\left(\frac{n+\xi}{T}\right) 
   - \sum_{n=T}^{[T\log y]} N_A\left(\frac{n+\xi - 1/9}{T}\right),  
\end{align}
If $T \in (0,C']$ for any fixed constant $C'\geq 1$, then
\begin{align} \label{eq:N_A-C'}
\nonumber
 \# \mathcal{N}_A(y)
&> N_A\left(\frac{[T\log y]+\xi}{T}\right) 
 - N_A\left(\frac{[T\log y]+\xi - 1/9}{T}\right) \\
\nonumber 
&= N_A\left(\log y\right) 
 - N_A\left(\log y - \frac{1}{9T}\right) \\
\nonumber 
&=\pi(y;Q,A) - \pi(y e^{-1/(9T)};Q,A) \\
\nonumber
&>\pi(y;Q,A) - \pi(y e^{-1/(9C')};Q,A)\\
&\gg_{C'} \pi(y;Q,A),
\end{align}
and $c_3 \gg_{C'} 1$ in \eqref{eq:N_A-claim}.
This leaves us to establish \eqref{eq:N_A-claim} for 
$T \in (C', \frac{1}{\pi}(\log x)^{1- \alpha_h/2}]$.

To bound \eqref{eq:N_A(y)} below, note that the prime number theorem  
\eqref{eq:PNTAP} implies that
\begin{align} \label{eq:N_A-sum}
 \sum_{n=T}^{[\tau]} N_A\left(\frac{n+\xi}{T}\right)
 = \frac{T}{\phi(Q)} \sum_{n=T}^{[\tau]} \frac{e^{\frac{n+\xi}{T}}}{n+\xi}
   + O\Bigg(\frac{T^2}{\phi(Q)} \sum_{n=T}^{[\tau]} 
\frac{e^{\frac{n+\xi}{T}}}{(n+\xi)^2}\Bigg).
\end{align}
A corresponding expansion for the second sum in \eqref{eq:N_A(y)} is 
obtained on replacing $\xi$ by $\xi - 1/9$.
The sum in the main term above may be asymptotically evaluated, using induction:
\begin{equation}\label{eq:wintner-sum}
(e^{\frac{1}{T}}-1)\sum_{n=T}^{N-1} \frac{e^{\frac{n}{T}}}{n+\xi}
= \frac{e^{\frac{N}{T}}}{N}
   + O\Bigg(\frac{1}{T} + \sum_{n=T}^{N} 
            \frac{e^{\frac{n}{T}}}{(n+1)^2}\Bigg), \quad (N \geq T+1). 
\end{equation}
Indeed, if $N=T+1$, then
$$
\left| \frac{e}{T+\xi} - \frac{e^{1+\frac{1}{T}}}{T+1}\right|
= \left|e^{1+\frac{1}{T}} \frac{1-\xi}{(T+1)(T+\xi)} - \frac{e}{T+\xi}\right|
= O\left(\frac{e^{1+\frac{1}{T}} }{(T+1)^2} + \frac{1}{T} \right),
$$
and if we assume that \eqref{eq:wintner-sum} holds for $N=M$, then it 
follows for $N=M+1$, since
\begin{align*}
\frac{e^{\frac{M}{T}}}{M} + \frac{(e^{\frac{1}{T}}-1) e^{\frac{M}{T}}}{M+\xi}
&=  \frac{e^{\frac{M}{T}}}{M} + \frac{(e^{\frac{1}{T}}-1) e^{\frac{M}{T}}}{M} 
+ \frac{\xi e^{\frac{M}{T}}(e^{\frac{1}{T}}-1)}{M(M+\xi)}\\
&= \frac{e^{\frac{M+1}{T}}}{M} + O\left(\frac{e^{\frac{M+1}{T}}}{M^2} \right)
= \frac{e^{\frac{M+1}{T}}}{M+1} + O\left(\frac{e^{\frac{M+1}{T}}}{M^2} \right).
\end{align*}

Thus, evaluating main term in \eqref{eq:N_A-sum} by means of 
\eqref{eq:wintner-sum}, we obtain
\begin{align} \label{eq:N_A-sum'}
 \sum_{n=T}^{[\tau]} N_A\left(\frac{n+\xi}{T}\right)
 = \frac{T}{\phi(Q)} 
\frac{e^{\frac{\xi}{T}} e^{\frac{[\tau]+1}{T}}}{(e^{\frac{1}{T}}-1)([\tau]+1)}
      + O\Bigg(\frac{T^2}{\phi(Q)} 
                \sum_{n=T}^{[\tau]} \frac{e^{\frac{n+1}{T}}}{(n+1)^2} 
                + \frac{1}{\phi(Q)}\Bigg).
\end{align}
Since $\int \frac{\d x}{(\log x)^2} = \li(x) - \frac{x}{\log x} \ll 
\frac{x}{(\log x)^2}$, the sum in the error term satisfies
\begin{align*}
 \sum_{n=T}^{[\tau]} \frac{e^{\frac{n+1}{T}}}{(n+1)^2} 
 \leq \int_T^{\tau+1} \frac{e^{t/T}}{t^2} \d t
 = \frac{1}{T} \int_e^{e^{(\tau+1)/T}} \frac{\d u}{(\log u)^2}
 = O\bigg( \frac{Te^{\frac{\tau+1}{T}}}{(\tau+1)^2} + 1 \bigg).
\end{align*}
Inserting this information, \eqref{eq:N_A-sum'} and the analogues expression 
with $\xi$ replaced by $\xi - 1/9$ into \eqref{eq:N_A(y)}, 
we obtain
\begin{align*}
\# \mathcal{N}_A(y)
&= \frac{T}{\phi(Q)} 
   \frac{e^{\frac{[T \log y]+1}{T}}e^{\frac{\xi}{T}} }{[T \log y]+1}
\frac{1 - e^{-\frac{1}{9T}}}{e^{\frac{1}{T}}-1}
+O\bigg( \frac{T}{\phi(Q)} \frac{e^{\frac{T \log y+1}{T}}}{(T \log y+1)^2/T^2} 
+ \frac{T^2}{\phi(Q)} \bigg).
\end{align*}
Recalling that $1 < C' < T \leq (\log x)^{1 - \alpha_h/2}/\pi $
and that $(1- \alpha_h/4) \log x < \log y \leq \log x$, this yields
\begin{align} \label{eq:N_A-sum-3}
\nonumber
\# \mathcal{N}_A(y)
&= \frac{T}{\phi(Q)} 
   \frac{e y }{[T \log y]+1}
\frac{1 - e^{-\frac{1}{9T}}}{e^{\frac{1}{T}}-1}
+O\bigg( \frac{T}{\phi(Q)} \frac{y}{(\log y)^2} \bigg) \\
\nonumber
&= \frac{T}{\phi(Q)} 
   \frac{e y }{[T \log y]+1}
\frac{1 - e^{-\frac{1}{9T}}}{e^{\frac{1}{T}}-1}
+O_{\alpha_h}\bigg( y(\log y)^{-1 - \alpha_h/2}/\phi(Q) \bigg) \\
&\gg_{\alpha_h}  \frac{1 - e^{-\frac{1}{9T}}}{e^{\frac{1}{T}}-1}
\frac{1}{\phi(Q)} \frac{y}{\log y}.
\end{align}
Thus, it remains to bound below the leading fraction in this bound.
To this end, note that
$$
e^{-\tau} 
= 1 - \frac{\tau}{2} - \frac{\tau - \tau^2}{2} 
 - \sum_{k=1}^{\infty} \frac{\tau^{2k+1}}{(2k+1)!}(1-\frac{\tau}{2k+2})
\leq 1 - \frac{\tau}{2}
$$ 
for every $\tau \in [0,1]$, and that
$$
e^\tau \leq 1 + \tau + \frac{\tau^2}{2} \sum_{k=0}^{\infty} 2^{-k} 
= 1 + \tau + \tau^2 \leq 1 + 2\tau
$$
for all $\tau \in [0,\frac{1}{2}]$.
Thus, if $T \geq 2$, then the leading factor in the lower bound 
\eqref{eq:N_A-sum-3} satisfies
$$
\frac{1 - e^{-\frac{1}{9T}}}{e^{\frac{1}{T}}-1}
>\frac{1  - 1  + \frac{1}{18T}}{1 +\frac{2}{T}-1}
=\frac{1}{36},
$$
and it follows that $c_3 \gg_{\alpha_h} 1$ in this case.
Choosing $C'=2$ in \eqref{eq:N_A-C'}, this completes the proof of 
\eqref{eq:N_A-claim} and of the lemma.
\end{proof}

\begin{proof}[Proof of Lemma \ref{l:appl-2}]
 This lemma follows from the proof above, observing that the information 
\eqref{eq:primes-g-lower-bd} gained from the Sato-Tate law is 
now included as an assumption in the statement of the lemma.
More precisely, \eqref{eq:primes-g-lower-bd} is only required for
$c_1 = 1 - c_2/2$, with $c_2 = c_3/6 \gg_{\alpha_h} 1$.
Thus, $c_1 = 1 - c$ for some $c>0$ only depending on 
$\alpha_h= \alpha/H$.
\end{proof}

\begin{proof}[Proof of Lemma \ref{l:appl-3}]
To deduce this lemma, we need to apply Proposition \ref{p:sufficient-2}
instead of the special case recorded in Lemma \ref{l:lipschitz-non-negative}.
We restrict attention to the first part of this lemma, the second being a simplification. 
Let $h$ be the function associated to $f$ via \eqref{eq:h(p)/H}. 
Let $x>1$ and let $t_x$ be a real number as in Lemma \ref{l:lipschitz-t-uniform}, 
applied with $f_0=h$.
By arguing as in the proof of Lemma \ref{l:lipschitz-t-uniform}, it follows from 
\eqref{eq:GS-Thm3} that
$$
|S_{h\chi_0}(x)| \ll \frac{1}{|t_x| + 1} + \frac{\log \log x}{\log x} + 
\frac{1}{(\log x)^{1+C_0}} \exp\Big(\sum_{p\leq x} \frac{|h(p)\chi_0(p)|}{p}\Big)
$$
whenever $\chi_0 \Mod{Q}$, $Q \leq \exp((\log\log x)^2)$, is a trivial character.
If $|t_x| > (\log x)^{1-\alpha_h/2}$, then 
$|S_{h\chi_0}(x)|$ is small, and we set $\tau_x = 0$.
If $|t_x| \leq (\log x)^{1-\alpha_h/2}$, we instead set $\tau_x = t_x$.

The rest of the proof proceeds almost exactly as that of Lemma \ref{l:appl-1}, 
but with the following changes.
Instead $|S_{h\chi}(x)|$, we now seek to bound 
$|S_{h(n)\chi(n)n^{-it_x}}(x)|$, or even 
\begin{equation}\label{eq:l4.17-proof}
 \max_{|t| \leq (\log x)^{1-\alpha_h/2}} |S_{h(n)\chi(n)n^{-it}}(x)|.
\end{equation}
Since the parameter $Y$ is chosen as $Y = (\log x)^{1-\alpha_h/2}$ in the proof of Lemma \ref{l:appl-1},
we may readily turn the bound \eqref{eq:Halasz-G-S} into one on \eqref{eq:l4.17-proof}
by redefining $M$ as $M=M(x,Y')$ with $Y'=2Y$, a change which does not affect the rest of the argument.
Continuing from here, we replace the decomposition \eqref{eq:splitting-the-distance} by
\begin{align*}
\nonumber
M(x,Y') 
&= \min_{|y| \leq 2Y'} \sum_{p \leq x} \frac{1 - |h(p)| + |h(p)| - \Re( h(p)\chi(p) 
p^{iy})}{p} \\
&= \sum_{p \leq x} \frac{1 - |h(p)|}{p}
+ \min_{|y| \leq 2Y'} \sum_{p \leq x} \frac{|h(p)|(1 - \sgn(h(p)) \Re(\chi(p) p^{iy}))}{p},
\end{align*}
and let
$$M_{h\chi}(x,Y') 
= \min_{|y| \leq 2Y'} \sum_{p \leq x} \frac{|h(p)|(1 - \sgn(h(p))\Re(\chi(p) p^{iy}))}{p}$$ 
denote the second term from this new expression.
As in the proof of Lemma \ref{l:appl-2}, we need to replace \eqref{eq:primes-g-lower-bd} by our new 
assumptions, which will also allow us to fix $\sgn(h(p)) = \epsilon$. 
The set of primes in \eqref{eq:prime-subset-pretentious} now takes the form
$$\mathcal{P}_{\chi,t}(y) = \{ p \leq y: \epsilon \Re(\chi(p) p^{iy})) < 11/12\}.$$
The deduction of \eqref{eq:prime-subset-pretentious-bound} from \eqref{eq:N_A-claim} remains, 
apart from obvious changes taking into account the additional sign $\epsilon$, unchanged.
\end{proof}

\section{$W$-trick} \label{s:W}
In the same way as the bounds mentioned in the introduction only apply to Fourier 
coefficients $\frac{1}{x}\sum_{n\leq x} f(n) e(\alpha n)$
at an irrational phase $\alpha$, it is the
case that an arbitrary multiplicative function $f$ may
\emph{correlate} with a given nilsequence, unless this sequence itself is 
sufficiently equidistributed. 
Thus, statements of the form 
$$\frac{1}{N}\sum_{n\leq N} h(n)F(g(n)\Gamma)=o_{G/\Gamma}(1)$$ 
with $h=f$ or $h = f - S_f(N;1,1)$
cannot be expected to hold in general.
On the other hand, it turns out to be sufficient to ensure that $h$ is
equidistributed in progressions to small moduli in order to resolve this
problem.
For arithmetic applications such as establishing a result of the form 
\eqref{eq:star}, this can be achieved with the help of the $W$-trick from 
\cite{GT-longprimeaps}.
The basic idea is to decompose $f$ into a sum of functions that are
equidistributed in progressions to small moduli.
This is achieved by decomposing the range $\{1,\dots, N\}$ into subprogressions
modulo a product $W(N)$ of small primes, which has the effect of fixing or
eliminating the contribution from small primes on each of the subprogressions.

For multiplicative functions some minor modifications are necessary.
Our aim is to decompose the interval $\{1, \dots, N\}$ into subprogressions 
$r \Mod{q}$ in such a way that
\begin{equation}\label{eq:constant-av}
 S_f(N;q,r)=(1+o(1))S_f(N;qq',r + qr')
\end{equation}
for small $q'$ and $0 \leq r' < q$.
Thus, $f$ should essentially have a constant average value when decomposing one of
the given subprogressions into further subprogressions of small moduli $q'$.
The example of the characteristic function of sums of two squares shows that we
cannot in general choose $q$ to be a product of small primes (consider the case
where $r \equiv 1 \Mod{2}$, $q'=2$ and $r+qr' \equiv 3 \Mod{4}$), but rather need
to allow $q$ to be a product of small prime powers.
Note further, that if $f$ is a function for which Shiu's bound on $S_f(N;q,r)$ 
is correct in the sense that
$$
S_f(N;q,r) \sim 
\frac{q}{\phi(q)}
\frac{1}{\log N} 
\prod_{\substack{p\leq N \\ p \nmid q}}
\left(1 + \frac{f(p)}{p}\right),
$$
then, in order for \eqref{eq:constant-av} to hold, we must have $p|q$ whenever 
$p|q'$ and $p$ is small. 

Our aim in this section is to show that for every $f \in \mathcal{F}_H$ 
we may, instead of $q = W(N)$ as in \cite{GT-longprimeaps}, take 
$q=\widetilde{W}(N):=q^*(N)W(N)$ for some integer-valued function 
$q^*: \NN \to \NN$ that satisfies the bound $q^*(x) \leq (\log x)^{O(1)}$.
For comparison, recall that $W(x) = \prod_{p\leq w(x)} p$, with $w$ as in 
Definition \ref{d:w}.
Thus, 
$$\log W(x) = \sum_{p \leq w(x)} \log p \sim w(x) 
\quad \text{ and } \quad
W(x) \leq (\log x)^{1+o(1)}.
$$

For such a function $\widetilde{W}$, we may decompose the range $[1,N]$ into 
subprogressions of the form
\begin{align*}
\Big\{ 1 \leq m \leq N : m \equiv w_1A \Mod{w_1\widetilde{W}(N)} \Big\}, 
\end{align*}
where $A \in (\ZZ/\widetilde{W}(N)\ZZ)^*$ and where $w_1 \geq 1$ is composed 
entirely of primes dividing $\widetilde{W}(N)$.
Abbreviating $\widetilde{W}=\widetilde{W}(N)$, we have
$\gcd(w_1, \widetilde{W}n+A)= 1$ and hence 
$f(w_1(\widetilde{W}n+A)) = f(w_1) f(\widetilde{W}n+A)$.
Thus, it suffices to study the family of functions 
$$
\left\{n \mapsto f(\widetilde{W}n+A) :
\begin{array}{l}
0<A<\widetilde{W}(N) \\
\gcd(A, \widetilde W)=1
\end{array}
\right\}.$$ 
Our first concern is to discard the set of large values of $w_1$ from 
consideration, as by doing so we can insure that the range on which each function 
$n \mapsto f(\widetilde{W}n+A)$ 
needs to be considered is always large.
Since large values of $w_1$ form a sparse set, their contribution in any 
arithmetic application can usually be bounded by just using the Cauchy--Schwarz 
inequality and a bound on the second moment of $f$ as in \cite[Lemma 7.9]{bm}.
More precisely, one can show that if, for $C_1>1$,
$$
\mathcal{S}_{C_1}(N) = 
\left\{w_1  \in \NN : 
\begin{array}{l}
w_1 > (\log N)^{C_1} \\
p|w_1 \Rightarrow p|\widetilde W(N) 
\end{array}
\right\},
$$
then
$$
\frac{1}{N}\sum_{n \leq N} 
\sum_{w_1 \in \mathcal{S}_{C_1}(N)}  
\1_{w_1|n} |f(n)|
\ll (\log N)^{-C_1/3}
$$
provided $q^*(N) < (\log N)^{C_1/3}$ and $C_1$ is sufficiently large with 
respect to $H$; see \cite[\S2]{lmm2} for details.
By choosing $C_1 > 3 \alpha_f$, we can for instance ensure that this bound is 
$o(\frac{1}{N}\sum_{n \leq N} |f(n)|)$.
As shown in \cite[\S2]{lmm2}, the contribution of $\mathcal{S}_{C_1}(N)$ 
to correlations of the form \eqref{eq:star} is negligible.

Thus, for the purpose of arithmetic applications, it suffices to consider 
$n \mapsto f(\widetilde{W}n+A)$ for $n \in \{1, \dots, T\}$ with
$$T = \frac{N-A w_1}{w_1\widetilde{W}(N)} 
\gg  \frac{N}{(\log N)^{C_1} \widetilde{W}(N)}.$$

The next proposition shows that every function $f \in \mathcal{F}_{H}$ admits a 
$W$-trick.
More precisely, any finite collection $f_1, \dots, f_r$ of elements from
$\mathcal{F}_{H}$ simultaneously admits a $W$-trick and we moreover have 
control over the size of $\widetilde W(=q)$ and over the level of $q'$ up to 
which (a weakened form of) the relation \eqref{eq:constant-av} holds.
Below, $\widetilde W$ plays the role of $q$ and $q$ plays the role of $q'$.

\begin{proposition}[The elements of $\mathcal{F}_{H,n^{it}}$ admit a 
$W$-trick]\label{p:major-arc}
Let $E, H \geq 1$ be constants and let $f_1, \dots, f_r \in \mathcal{F}_{H, n^{it}}$.
Then there exists a constant $\kappa$, depending on $E$, $H$, $r$ and 
$\alpha = \min_{1 \leq j \leq r} \alpha_{f_j}$, 
and functions $\varphi':\NN \to \RR$ and $q^*:\NN \to \NN$ such that the 
following holds:
\begin{enumerate}
 \item $\varphi'(x) \to 0$ as $x\to\infty$,
 \item $q^*(x) \leq (\log x)^{\kappa}$ for all sufficiently large $x \in \NN$,
 \item if $x\in \NN$ is sufficiently large, if we set 
 $\widetilde{W}(x):=q^*(x)W(x)$, and if we define
 $$f_x:n \mapsto f(n)n^{-it_x} \quad \text{ for any } f \in \{f_1, \dots, f_r\}$$ 
and with $t_x$ as in Definition \ref{def:F_H(x)} with $C=2E + \kappa + 4$,
 then the estimate 
\begin{align}\label{eq:major-arc}
\nonumber
\frac{q \widetilde{W}(x)}{|I|}
& \sum_{\substack{  m \in I \\ m \equiv A ~(q \widetilde{W}(x))}} f_x(m)
- S_{f_x}(x;\widetilde{W}(x),A)\\
&= O_{E,H,\kappa}\bigg(\varphi'(x) \frac{1}{\log x} 
     \frac{\widetilde W(x)}{\phi(\widetilde W(x))}
   \prod_{\substack{p<x\\p\nmid \widetilde W(x)}} \left(1 + \frac{|f(p)|}{p} 
\right) 
\bigg)
\end{align}
holds uniformly for all intervals $I \subseteq \{1,\dots, x\}$ with 
$|I|> x(\log x)^{-E}$, for all integers 
$0<q \leq (\log x)^{E}$ 
and for all $A \in (\ZZ/q \widetilde{W}(x)\ZZ)^*$.
\end{enumerate}
\end{proposition}
\begin{rems} \label{rem:phi'}
(1) If $f \in \mathcal{F}_H$, then $f_x=f$. \indent
(2) We will show that \eqref{eq:major-arc} holds with
$\varphi'(x)= \varphi_{C}(x) + (\log w(x))^{-1} + (\log x)^{-\alpha_f/(3H)} + 
(\log x)^{-E}$, where
$\varphi_C$ is as in Definition \ref{d:F} with $C=2E + \kappa + 4$.
\end{rems}

The rest of this section is devoted to a proof of Proposition 
\ref{p:major-arc}. 
Our strategy is to first relate the left hand side of \eqref{eq:major-arc} to a 
restricted character sum, which we will then attempt to bound by means of 
the `pretentious large sieve'-consequence recorded in Corollary \ref{c:GS}.

We begin with a technical lemma that will at various points in the argument 
allow us to control the contribution of the prime divisors $p|\widetilde W(N)$ 
that are larger than $w(N)$.
\begin{lemma} \label{l:a'}
Let $1 \leq a \leq(\log N)^{E}$ be an integer that is free from prime 
factors $p<w(N)$ and suppose that $0 \leq g(p) \leq H$ for all $p$.
Then 
$$
\prod_{p|a} \left(1+\frac{g(p)}{p}\right)
= 1 + O_{E,H}\Big(\frac{1}{\log w(N)}\Big).
$$
\end{lemma}
\begin{proof}
The assumptions on $a$ imply the bound 
$\Omega(a) \leq \frac{E\log\log N}{\log w(N)}$ on the total number of prime 
factors of $a$.
Let
$m = [w(N)/\log w(N) + \Omega(a')]$ and recall that the $n$-th prime 
$p_n$ satisfies $p_n \sim n \log n$.
Then,
$$
p_m \sim m \log m \leq \frac{w(N) + E \log \log N}{\log w(N)} \log m
\sim w(N) + E \log \log N.
$$
Using the bounds on $w(N)$ from Definition \ref{d:w} and Mertens' estimate, we 
obtain
\begin{align*}
 \prod_{p|a} \left(1+\frac{g(p)}{p}\right)
 &\leq \prod_{p|a} \left(1+ p^{-1}\right)^H
  \leq \prod_{w(N) < p < p_m}(1 + p^{-1})^H
\\
 &\leq \left(
 \frac{\log (w(N)+ E \log \log N) + O(1)}{\log w(N) + O(1)}
 \right)^H \\
 &= \left( 1 + O_E\left(\frac{1}{\log w(N)}\right) \right)^H
 =  1 + O_{E,H}\Big(\frac{1}{\log w(N)}\Big),
\end{align*}
as claimed.
\end{proof}

\begin{corollary} \label{c:a'}
If $x$ and $q$ are as in Proposition \ref{p:major-arc}, then  
$$
\frac{q\widetilde W(x)}{\phi(q\widetilde W(x))}
\leq \left( 1 + O_{E,H}\left(\frac{1}{\log w(N)}\right) \right)
  \frac{\widetilde W(x)}{\phi(\widetilde W(x))}.
$$
\end{corollary}
\begin{proof} Let $a = \prod_{p|q, p \nmid \widetilde W(x)}p$. Then
\begin{align*}
\prod_{p|a} (1 - p^{-1})^{-1} &=
\prod_{p|a} (1 + p^{-1})(1 - p^{-2})^{-1} \\
&\leq \exp\bigg( \sum_{p|a} \frac{2}{p^2}\bigg) \prod_{p|a} (1 + p^{-1}) 
=\Big(1+ O\Big(\frac{1}{w(x)}\Big)\Big) \prod_{p|a} (1 + p^{-1}).
\end{align*}

\end{proof}

The next lemma replaces the general interval $I$ from \eqref{eq:major-arc} by 
one of the form $\{1, \dots, y\}$.

\begin{lemma}\label{l:major-arc}
If $E,H,x,f$ and $f_x$ are as in Proposition \ref{p:major-arc},
if $\kappa \geq 0$ is a given constant and if
$\widetilde W(x) \leq (\log x)^{\kappa + 2}$ is a multiple of $W(x)$,
then \eqref{eq:major-arc} follows if there exists a function $\varphi''=o(1)$ 
such that
\begin{align}\label{eq:major-arc'}
S_{f_x}(x;q\widetilde{W}(x),A)
= S_{f_x}(x;\widetilde{W}(x),A)
+ O_{E,H,\kappa}\Bigg(
 \frac{\varphi''(x)}{\log x}
 \frac{\widetilde Wq}{\phi(\widetilde Wq)}
   \prod_{\substack{p<x\\p\nmid \widetilde Wq}}
\left(1 + \frac{|f(p)|}{p} \right) \Bigg)
\end{align}
for all $q \in (0, (\log x)^{-E}]$ and 
$A \in (\ZZ/{q \widetilde{W}(x)}\ZZ)^*$.
More precisely, we may take 
$\varphi'(x)=\varphi_C(x)+\varphi''(x) + (\log w(x))^{-1}$ 
in \eqref{eq:major-arc}, where $\varphi_C$ is as in Definition \ref{d:F}
with $C=2E + \kappa + 4$.
\end{lemma}
\begin{proof}
In view of \eqref{eq:major-arc'}, it suffices to relate the first term in \eqref{eq:major-arc}
to $S_{f_x}(x;q\widetilde{W}(x),A)$.
Let $y_1, y_2 \in \ZZ_{\geq 0}$ and suppose that $I=(y_1,y_1+y_2]\subset [1,x]$ 
with $y_2 \geq x(\log x)^{-E}$.
Writing $\widetilde W = \widetilde W(x)$,
an application of \eqref{eq:d-F} with $C:=2E + \kappa + 4>E$ 
shows that the first term in \eqref{eq:major-arc} satisfies
\begin{align} \label{eq:major-arc_I}
\frac{q \widetilde{W}}{|I|}
\sum_{\substack{  m \in I \\ m \equiv A ~(q \widetilde{W})}} f_x(m) 
&= \frac{q \widetilde{W}}{y_2}
\sum_{\substack{ y_1 < m \leq y_1 +y_2 \\ m \equiv A ~(q \widetilde{W})}} f_x(m)
\\
\nonumber 
&=\frac{y_1+y_2}{y_2} 
S_{f_x}\Big(y_1+y_2;q \widetilde{W},A\Big)
-\frac{y_1}{y_2} S_{f_x}\Big(y_1;q  \widetilde{W},A\Big) \\
\nonumber
&=\frac{y_1+y_2}{y_2} 
S_{f_x}\Big(x;q \widetilde{W},A\Big)
-\frac{y_1}{y_2} S_{f_x}\Big(y_1;q  \widetilde{W},A\Big) \\
\nonumber
&\qquad \qquad \qquad \qquad+ 
O\bigg(\frac{\varphi_C(x)}{\log x} \frac{\widetilde W q}{\phi(\widetilde W q)}
   \prod_{\substack{p<x\\p\nmid \widetilde W q}}
   \left(1 + \frac{|f(p)|}{p} \right) \bigg).
\end{align}
We now split into two cases.
If, on the one hand, 
$x>y_1 > y_2(\log x)^{-E - \kappa -4} > x(\log x)^{-2E - \kappa-4}$, then 
\eqref{eq:d-F} shows that 
$S_{f_x}(y_1;q  \widetilde{W},A)$ can be replaced by 
$S_{f_x}(x;q  \widetilde{W},A)$ in the final expression in 
\eqref{eq:major-arc_I} so that \eqref{eq:major-arc_I} is seen to equal
$$
S_{f_x}\Big(x;q \widetilde{W},A\Big) 
+ O\Bigg(
 \frac{\varphi_C(x)}{\log x}
 \frac{\widetilde W q}{\phi(\widetilde Wq)}
 \prod_{\substack{p<x\\p\nmid \widetilde Wq}}
 \left(1 + \frac{|f(p)|}{p} \Bigg)
\right).
$$
In this case, \eqref{eq:major-arc} follows with 
$\varphi'=\varphi_C + \varphi'' + (\log w(x))^{-1}$
from \eqref{eq:major-arc'} and Corollary \ref{c:a'}.

If, on the other hand, $y_1 \leq y_2(\log x)^{-E-\kappa-4}$, then 
$\frac{y_1+y_2}{y_2} = (1 + O\left((\log x)^{-E-\kappa-4}\right))$.
Since $\phi(q\widetilde{W}) \leq q\widetilde{W} \leq (\log x)^{E + \kappa + 2}$, we further have
\begin{align*}
S_{f_x}\Big(y_1;q  \widetilde{W},A\Big)
&= \frac{q  \widetilde{W}}{y_1}
  \sum_{\substack{n \leq y_1 \\ n \equiv A \Mod{q  \widetilde{W}}}} f(n)
\leq q \widetilde{W} \sum_{\substack{n \leq y_1 \\ n \equiv A \Mod{q\widetilde{W}}}} \frac{f(n)}{n} \\
&\leq q \widetilde{W}
     \prod_{\substack{p \leq y\\ p \nmid q \widetilde W}}
     \left(1 + \frac{|f(p)|}{p} + \frac{H^2}{p^2 (1-H/p)} \right)\\
&\ll q \widetilde{W}
     \prod_{\substack{p \leq y\\ p \nmid q \widetilde W}}
     \left(1 + \frac{|f(p)|}{p} \right)
\ll (\log x)^{E+\kappa +2} \frac{q \widetilde{W}}{\phi(q \widetilde{W})} 
     \prod_{\substack{p \leq y\\ p \nmid q \widetilde W}}
     \left(1 + \frac{|f(p)|}{p} \right),
\end{align*}
which implies that
\begin{align*}
\frac{y_1}{y_2} S_{f_x}\Big(y_1;q  \widetilde{W},A\Big)
\leq (\log x)^{-E-\kappa-4} S_{f_x}\Big(y_1;q  \widetilde{W},A\Big)
&\ll (\log x)^{-2} 
\frac{\widetilde Wq}{\phi(\widetilde Wq)}
\prod_{\substack{p<x\\p\nmid \widetilde Wq}}
\left(1 + \frac{|f(p)|}{p} \right).
\end{align*}
Thus, in this case, \eqref{eq:major-arc_I} equals
$$
S_{f_x}(x;q\widetilde{W},A)
+ O\Bigg(\frac{\varphi_C(x)+(\log x)^{-1}}{\log x}
\frac{\widetilde Wq}{\phi(\widetilde Wq)}
\prod_{\substack{p<x\\p\nmid \widetilde Wq}}
\left(1 + \frac{|f(p)|}{p} \right)
\Bigg),
$$
and an application of \eqref{eq:major-arc'} yields \eqref{eq:major-arc} 
with $\varphi'(x) = \varphi_C(x)+ \varphi''(x) +(\log w(x))^{-1}$, when taking 
into account Corollary \ref{c:a'}.
\end{proof}

Following the above reduction, we now proceed to analyse the difference of the 
two mean values that appear in \eqref{eq:major-arc'}. 
\begin{lemma}[Restricted character sum] \label{l:restricted-sum}
Let $g : \NN \to \CC$ be an arithmetic function, not necessarily 
multiplicative, let $\widetilde{W}, q, A \geq 1$ be integers
and suppose that $\gcd(A,q\widetilde{W})=1$. 
If $y \geq 1$, then 
\begin{align} \label{eq:reduced-chi-sum}
S_{g}(y;\widetilde W, A) - S_{g}(y;q \widetilde W, A)
= \frac{q \widetilde{W}}{y}
\frac{1}{\phi(q \widetilde{W})}
\starsum_{\chi \Mod{q \widetilde{W}}} 
\chi(A)
\sum_{n\leq y} g(n) \bar\chi(n), 
\end{align}
where $\starsum$ indicates the restriction of the sum to characters that are 
not 
induced from characters $\Mod{\widetilde{W}}$.
\end{lemma}
\begin{proof}
We have
\begin{align} \label{eq:starcharsum}
\nonumber 
&S_g(y;\widetilde W, A) - S_g(y;q \widetilde W, A)\\
\nonumber
&=\frac{\widetilde W}{y} \Bigg(
\sum_{\substack{n\leq y \\ n \equiv A \Mod{\widetilde{W}}}} g(n)
- q
\sum_{\substack{n\leq y \\ n \equiv A \Mod{q \widetilde{W}}}} g(n)
\Bigg)
\\
\nonumber
&= \frac{1}{y}
\frac{\widetilde W}{\phi(q \widetilde{W})}
\sum_{\chi \Mod{q \widetilde{W}}} 
\Bigg( 
\sum_{\substack{A' \Mod{q \widetilde{W}} \\ A \equiv A' \Mod{\widetilde{W}}}}
\chi(A')
- q\chi(A)
\Bigg)
\sum_{n\leq y} g(n) \bar\chi(n) \\
&= \frac{1}{y}\frac{\widetilde{W}}{\phi(q \widetilde{W})}
\starsum_{\chi \Mod{q \widetilde{W}}} 
\Bigg( 
\sum_{\substack{A' \Mod{q \widetilde{W}} \\ A \equiv A' \Mod{\widetilde{W}}}}
\chi(A')
- q\chi(A)
\Bigg)
\sum_{n\leq y} g(n) \bar\chi(n),
\end{align}
where $\starsum$ indicates the restriction of the sum to characters that are 
not 
induced from characters $\Mod{\widetilde{W}}$; for all other characters we have 
$\chi(A') = \chi(A)$ and the difference in the brackets above is zero.
It remains to show that the sum over $A'$ in \eqref{eq:starcharsum} vanishes.
However,
\begin{align*}
 \sum_{\substack{A' \Mod{q \widetilde{W}}\\ A \equiv A' \Mod{\widetilde{W}}}}
\chi(A')
 =\frac{1}{\phi(\widetilde{W})}
 \sum_{\chi' \Mod{\widetilde{W}}}
 \bar{\chi'}(A) 
 \sum_{A' \Mod{q \widetilde{W}}} \chi(A') \chi'(A')=0,
\end{align*}
since $\chi \chi'$ is a non-trivial character modulo $q \widetilde{W}$.
Thus the lemma follows.
\end{proof}

Finally, we aim to exploit the fact that the character sum on the right hand 
side of \eqref{eq:reduced-chi-sum} is restricted by invoking Corollary 
\ref{c:GS}.

\begin{proof}[Proof of Proposition \ref{p:major-arc}]
Let $\eps := \frac{1}{2}\min(1,\alpha/(2H))$, 
$k:=\lceil \eps^{-2}\rceil$ and 
$k' = k \lceil \log_2 (4H) \rceil$, as in the statement of 
Corollary \ref{c:GS}.
Setting $C' = (E+1)3^{rk'+1}$, we let $\mathcal{E}$ denote the union of
the sets of characters defined by Corollary \ref{c:GS} when applied with
$C=C'$ to each of the $r$ functions $f_x \in \mathcal{M}_H$ for $f \in \{f_1, \dots, f_r\}$.

Our aim is to find a suitable integer $\widetilde W(x)$ so that, if 
$\widetilde W = \widetilde W(x)$ and $q \leq (\log x)^E$, then none of 
the characters that appear in the restricted character sum 
\eqref{eq:reduced-chi-sum} is induced by a character from the set 
$\mathcal{E}$.
To do so, we construct a finite sequence of integers 
$W_0(x), W_1(x), \dots$ with the property 
$$W_i(x) \leq W(x)^{2^i}(\log x)^{3^iE}$$ as follows.
Let $W_0(x) = W(x)$ and suppose we have already defined 
$W_i(x)$ for all $0 \leq i \leq j$.
Consider the set of integers in the interval 
$I_{j} = [W_{j}(x), W_{j}(x)(\log x)^E]$.
If there exists a character $\chi \in \mathcal{E}$ whose conductor $c_{\chi}$ satisfies 
$c_{\chi} \nmid W_{j}(x)$ but $c_{\chi} < W_{j}(x)(\log x)^E$, then we choose one such 
character $\chi$ and define $W_{j+1}(x) := c_{\chi} W_{j}(x)$. 
Note that
$$
W_{j+1}(x) 
< W_{j}(x)^2(\log x)^E 
< W(x)^{2\cdot2^j}(\log x)^{(2\cdot3^j+1)E} 
< W(x)^{2^{j+1}}(\log x)^{3^{j+1}E}.$$
If there is no such $\chi \in \mathcal{E}$, then we stop and set
$\widetilde W(x) = W_{j}(x)$.
Since $\#\mathcal{E} \leq r k'$, this process stops after at most $r k'$ steps 
and, thus, 
$\widetilde W(x) \leq W(x)^{2^{rk'}}(\log x)^{3^{rk'}E} \leq 
(\log x)^{2^{rk'+ 1} + 3^{rk'}E}$
and
$$
\widetilde W(x) q < (\log x^{1/(8H)})^C
$$
for all $q \leq (\log x)^E$ and sufficiently large $x$.

Our construction ensures that there exists no character 
$\chi \Mod{q \widetilde W(x)}$ with $q \leq (\log x)^E$ that is induced by an 
element from $\mathcal{E}$ but \emph{not} induced from a character 
$\Mod{\widetilde W(x)}$.
Since the sum \eqref{eq:reduced-chi-sum} is restricted to those characters 
modulo $q \widetilde W(x)$ that are not induced from characters modulo
$\widetilde W(x)$, we may apply Corollary \ref{c:GS} with $\mathcal{C}$ given 
by this restricted set of characters and with $Q=q\widetilde W(x)$.
This application shows that whenever $1 \leq q \leq (\log x)^E$ and 
$x^{1/2} < y \leq x$, then
\begin{align*}
&\frac{1}{y}
\starsum_{\chi \Mod{q \widetilde{W}}} 
\chi(A) \sum_{n\leq y} f_x(n) \bar\chi(n) 
\ll_{C,H,\alpha} 
\frac{1}{(\log x)^{1+\alpha/(3H)}}
\exp \bigg( \sum_{\substack{p\leq x\\p\nmid q \widetilde W}} 
\frac{|f(p)|}{p}\bigg).
\end{align*}
In combination with Lemma \ref{l:restricted-sum} for $g=f_x$, 
this yields \eqref{eq:major-arc'} for $\kappa = C'-2$ and with 
$\varphi''(x)=(\log x)^{-\alpha/(3H)}$.
Hence, Lemma \ref{l:major-arc} implies the result with
$\kappa = C'-2 \ll_{E,H,r,\alpha} 1$.
\end{proof}

The estimate \eqref{eq:major-arc} will be referred to as `the major arc 
estimate'.
We will show in Section \ref{ss:reduction} that despite the restriction to 
invertible residues $A \in (\ZZ/q \widetilde{W}\ZZ)^*$, the estimate 
\eqref{eq:major-arc} 
implies that $f(\widetilde{W}n+A) - S_f(x;\widetilde{W},A)$ is
orthogonal to periodic sequences of period at most $(\log x)^E$, 
for every $A \in (\ZZ/\widetilde{W}\ZZ)^*$.
This information will be used in combination with a factorisation theorem to 
reduce the task of proving non-correlation for  
$(f(\widetilde{W}n+A) - S_f(x;\widetilde{W},A))$ with general nilsequences to the
case where the nilsequence enjoys certain equidistribution properties and the
Lipschitz function satisfies, in particular, $\int_{G/\Gamma} F = 0$.

\section{The non-correlation result} \label{s:non-corr}
This section contains a precise statement of the main result, 
which, informally speaking, shows the following.
Given $E \geq 1$ and a multiplicative function $f \in \mathcal{F}_{H}$, 
let 
$\widetilde W(x)$ be the function from Proposition \ref{p:major-arc}. 
Then for every  residue $A \in (\ZZ/\widetilde{W}(N)\ZZ)^*$ 
and for parameters $N$ and $T$ such that $N^{1-o(1)} \ll T \ll N$,
the sequence $(f(\widetilde{W}n+A) - S_f(N;\widetilde{W},A))_{n \leq T}$ is 
orthogonal to any given polynomial nilsequence, provided $E$ is sufficiently 
large with respect to $H$, $\alpha_f$ and data related to the nilsequence.
In Section \ref{ss:reduction} we carry out a standard reduction of the main 
result to an equidistributed version, modelled on \cite[\S2]{GT-nilmobius}.

\addtocontents{toc}{\SkipTocEntry}
\subsection{Statement of the main result}
\label{ss:statement}
We begin by recalling the definition of a polynomial nilsequence and related 
notions from \cite{GT-polyorbits}.
For this, let $G$ be a connected, simply connected, nilpotent Lie group.
In accordance with \cite{GT-polyorbits}, we define a filtration $G_{\bullet}$ 
on $G$ to be a finite sequence of closed connected subgroups
$$G = G_0 = G_1 \geq G_2 \geq \dots \geq G_d \geq G_{d+1} = \{\id_G\}$$
with the property that for all pairs $(i,j)$ with $0\leq i,j \leq d$, 
the commutator group $[G_i,G_j]$ is a subgroup of $G_{i+j}$, where we set 
$G_{i+j} =  \{\id_G\}$ if $i+j > d+1$.
The degree of $G_{\bullet}$ is defined to be the largest index $j$ for which 
$G_j$ is non-trivial.
Since $G$ is nilpotent, the lower central series, defined by $G_1=G$ and 
$G_{i+1} = [G,G_i]$ for $i \geq 1$, terminates after finitely many steps.
Setting $G_0=G$, this series defines a filtration.
If $s$ denotes the degree of this filtration, then the Lie group $G$ is called 
$s$-step nilpotent. 
One can show that $s$ is the smallest possible degree that a filtration of $G$ 
can have.

Let $g:\ZZ \to G$ be a sequence with values in $G$ and define for every 
$h \in \ZZ$, the discrete derivative
$\partial_h g(n) = g(n+h) g(n)^{-1}$.
Then, following \cite[Definition 1.8]{GT-polyorbits}, the set 
$\mathrm{poly}(\ZZ,G_{\bullet})$ of polynomial sequences with coefficients in 
$G_{\bullet}$ is defined to be the set of all sequences $g:\ZZ \to G$ for which
every $i$-th derivative takes values in $G_i$, i.e. for which 
$\partial_{h_i} \dots \partial_{h_1} g(n) \in G_i$ for all 
$i \in \{0, \dots, d+1\}$ and for all $n, h_1, \dots, h_i \in \ZZ$.

To define polynomial \emph{nil}sequences, let $\Gamma < G$ be a discrete 
co-compact subgroup. 
Then the compact quotient $G/\Gamma$ is called a nilmanifold.
Any Mal'cev basis $\mathcal{X}$ (see \cite[\S2]{GT-polyorbits} for a 
definition) for $G/\Gamma$ gives rise to a metric $d_{\mathcal{X}}$ on 
$G/\Gamma$ as described in \cite[Definition 2.2]{GT-polyorbits}.
This metric allows us to define Lipschitz functions on $G/\Gamma$ as the 
set of functions $F:G/\Gamma \to \CC$ for which the Lipschitz 
norm (cf. \cite[Definition 1.2]{GT-polyorbits})
$$
\|F\|_{\mathrm{Lip}} =
\|F\|_{\infty}
+ \sup_{x,y \in G/\Gamma} \frac{|F(x)-F(y)|}{d_{\mathcal{X}}(x,y)}
$$
is finite.
If $F$ is a $1$-bounded Lipschitz function, then 
$(F(g(n)\Gamma))_{n \in \ZZ}$ is called a (polynomial) nilsequence.

We are now ready to state the main result:
\begin{theorem}\label{t:non-corr}
Let $E,H,d,m_G \geq 1$ be integers and let $f \in \mathcal{F}_{H,n^{it}}$.
Let $N$ be a positive integer parameter and let 
$\widetilde W = \widetilde W (N)$ be the 
integer produced by Proposition \ref{p:major-arc} for the function $f$ when 
applied with the given values of $E,H$ and with $x=N$.
Let $A \in \NN$ be such that $0 < A < \widetilde{W}$ and 
$\gcd(\widetilde W,A)=1$. 
Suppose further that $T$ satisfies $N/(\log N)^{E/2} \ll T \ll N$ and that 
$T,N > e^e$.
Let $G/\Gamma$ be a nilmanifold of dimension $m_G$ together 
with a filtration $G_{\bullet}$ of $G$ of degree $d$ and let
$g \in \mathrm{poly}(\ZZ,G_{\bullet})$ a polynomial sequence.
Suppose that $G/\Gamma$ has a $M_0$-rational Mal'cev basis adapted to 
$G_{\bullet}$ for some $M_0 \geq 2$ and let $G/\Gamma$ be equipped with the 
metric defined by this basis.
Let $F:G/\Gamma \to \CC$ be a $1$-bounded Lipschitz function.
Then, provided $E \geq 1$ is sufficiently large with respect to 
$d$, $m_G$, $\alpha_f$ and $H$, we have
\begin{align}\label{eq:main}
&\left|
\frac{\widetilde{W}}{T}
\sum_{n\leq T/\widetilde{W}}
\bigg(f(\widetilde{W}n+A)-
(\widetilde{W}n+A)^{it_N} S_{f(n)n^{-it_N}}(N;\widetilde{W},A)\bigg)
F(g(n)\Gamma)
\right| \ll_{d,m_G,\alpha_f,H} \\
\nonumber
&\quad
\Bigg\{ \varphi'(N) + \frac{1}{\log w(N)} 
+\frac{M_0^{O_{d,m_G}(1)} }{(\log \log T)^{1/(4^{d+1} \dim G)}}
\Bigg\}
\frac{1+ \|F\|_{\mathrm{Lip}}}{\log T}
\frac{\widetilde W}{\phi(\widetilde W)}
\prod_{\substack{p\leq N\\ p\nmid \widetilde 
W(N)}}\left(1+\frac{|f(p)|}{p}\right),
\end{align}
where $t_N \in [-2 \log N,2 \log N]$ is, as in Proposition \ref{p:major-arc}, 
given by Definition \ref{def:F_H(x)} with $C= 2E + \kappa + 4$
(in particular, $t_N=0$ if $f \in \mathcal{F}_H$),
and where $\varphi'$ is given by \eqref{eq:major-arc}.
\end{theorem}

\begin{rem*}
 Partial summation, when combined with the estimate \eqref{eq:d-F}, which 
holds with the same value of $C$ as above for the function $n \mapsto f(n)n^{-it_N}$, shows that
\begin{align*}
S_{f(n)n^{-it_N}}&(N;\widetilde{W},A) = (1+it_N) N^{-it_N} S_{f}(N;\widetilde{W},A)\\
&+O\bigg( |t_N|(1+|t_N|)\bigg((\log N)^{-E + O(H)} 
+\frac{\vphi_C(N)}{\log N} \frac{\widetilde{W}}{\phi(\widetilde{W})} 
\prod_{\substack{p<N , p \nmid \widetilde{W}}} \left(1 + \frac{|f(p)|}{p} \right)\bigg)\bigg),
\end{align*}
where $\vphi_C$ is as in \eqref{eq:d-F}.
Thus, if $E \gg_{H, \alpha_f} 1$ is sufficiently large and $|t_N|^2 = o(\vphi_C(N))$, then we may replace
the term $(\widetilde{W}n+A)^{it_N} S_{f(n)n^{-it_N}}(N;\widetilde{W},A)$ in the statement above, by 
$(1+it_N)\Big( \frac{\widetilde{W}n+A}{N}\Big)^{it_N} S_{f}(N;\widetilde{W},A)$.
\end{rem*}

\addtocontents{toc}{\SkipTocEntry}
\subsection{Reduction of Theorem \ref{t:non-corr} to the equidistributed case}
\label{ss:reduction}

Proceeding similarly as in \S2 of Green and Tao \cite{GT-nilmobius}, we will 
reduce Theorem \ref{t:non-corr} to a special case that involves only 
equidistributed polynomial sequences.
Let us begin by recalling the quantitative notion of equidistribution and total 
equidistribution for polynomial sequences that was introduced in 
\cite[Definition 1.2]{GT-polyorbits}.

\begin{definition}
 Let $G/\Gamma$ be a nilmanifold equipped with Haar measure, 
 let $\delta>0$ and let $N \in \NN$.
 A finite sequence $g:\{1, \dots, N\} \to G$ is called $\delta$-equidistributed 
in $G/\Gamma$ if 
$$
\Big|\frac{1}{N} \sum_{n \leq N} F(g(n)\Gamma) - \int_{G/\Gamma} F \Big| 
\leq \delta \|F\|_{\mathrm{Lip}}
$$
for all Lipschitz functions $F:G/\Gamma \to \CC$.
It is called totally $\delta$-equidistributed if, moreover, 
$$
\Big|\frac{1}{\# P} \sum_{n\in P} F(g(n)\Gamma) - \int_{G/\Gamma} F \Big| 
\leq \delta \|F\|_{\mathrm{Lip}}
$$
for all Lipschitz functions $F:G/\Gamma \to \CC$ and progressions 
$P \subset \{1, \dots, N\}$ of length $\# P\geq \delta N$.
\end{definition}

The tool that makes a reduction to equidistributed polynomial sequences work 
is the following factorisation 
theorem \cite[Theorem 1.19]{GT-polyorbits} due to Green and Tao:

\begin{lemma}[Factorisation lemma, Green--Tao \cite{GT-polyorbits}] 
\label{l:factorisation}
Let $m$ and $d$ be positive integers, and let $M_0, N, B >1$ be real numbers. 
Let $G/\Gamma$ be an $m$-dimensional nilmanifold together with a
filtration $G_{\bullet}$ of degree $d$. 
Suppose that $\mathcal{X}$ is an $M_0$-rational Mal'cev basis adapted to 
$G_{\bullet}$ and let $g \in \mathrm{poly}(\ZZ, G_{\bullet})$ be a polynomial 
sequence.
Then there is an integer $M$ with $M_0 \ll M \ll M_0^{O_{B,m,d}(1)}$,
a rational subgroup $G' \subseteq G$, a Mal'cev basis $\mathcal{X'}$ for 
$G'/\Gamma'$ in which each element is an $M$-rational combination of the elements 
of $\mathcal{X}$, and a decomposition $g = \eps g' \gamma$ into
polynomial sequences $\eps, g', \gamma \in \mathrm{poly}(\ZZ, G_{\bullet})$ with 
the following properties:
\begin{enumerate}
 \item $\eps: \ZZ \to G$ is $(M,N)$-smooth\footnote{The notion of 
 smoothness was defined in \cite[Def.\ 1.18]{GT-polyorbits}. 
 A sequence $(\eps (n))_{n\in \ZZ}$ is said to be \emph{$(M,N)$-smooth} if 
 both $d_{\mathcal{X}}(\eps(n),\id_G) \leq N$ and 
 $d_{\mathcal{X}}(\eps(n),\eps(n-1)) \leq M/N$ hold 
 for all $1 \leq n \leq N$.};
 \item
 $g': \ZZ \to G'$ takes values in $G'$ and the finite sequence 
 $(g'(n)\Gamma')_{n \leq T}$ is totally $M^{-B}$-equidistributed in 
 $G'\Gamma/\Gamma$ using the metric $d_{\mathcal{X'}}$ on $G'\Gamma/\Gamma$;
 \item $\gamma: \ZZ \to G$ is an $M$-rational\footnote{A sequence 
 $\gamma: \ZZ \to G$ is said to be 
 \emph{$M$-rational} if for each $n$ there is $0<r_n\leq M$ such that 
 $(\gamma(n))^{r_n} \in \Gamma$; see \cite[Def.\ 1.17]{GT-polyorbits}.} 
 sequence and the sequence $(\gamma(n)\Gamma)_{n \in \ZZ}$ is periodic with 
 period at most $M$.
\end{enumerate}
\end{lemma}

The following proposition handles the special case of Theorem \ref{t:non-corr} 
where the polynomial sequence is equidistributed.

\begin{proposition}[Non-correlation, equidistributed case]
\label{p:equid-non-corr}
Let $E, H, m_G, d \geq 1$ be integers and suppose that $f \in \mathcal{M}_H$. 
Let $N$ and $T$ be integer parameters satisfying $N^{1-o(1)} \ll T \ll N$ and
let $\delta= \delta(N) \in (0,1/2)$ depend on $N$ in such a way 
that 
$$\log N \leq \delta(N)^{-1} \leq (\log N)^{E}.$$
Let $G/\Gamma$ be a nilmanifold of dimension $m_G$ together with a 
filtration $G_{\bullet}$ of degree $d$, and suppose that $\mathcal X$ is a 
$\frac{1}{\delta(N)}$-rational Mal'cev basis adapted to $G_{\bullet}$.
This basis gives rise to the metric $d_{\mathcal X}$.
Let $Q=Q(N) \leq (\log N)^{E}$ be an integer that is divisible by $W(N)$ and 
let $0 \leq A < Q$ be an integer such that
$A \in (\ZZ/Q\ZZ)^*$.

Then there is $E_0 \geq 1$, depending on $d$, $m_G$ and $H$, such that the 
following holds provided $E$ is sufficiently large with respect to $d$, $m_G$ and 
$H$: 

Let $g \in \mathrm{poly}(\ZZ,G_{\bullet})$ be any polynomial sequence 
such that the finite sequence 
$$(g(n)\Gamma)_{n\leq T/Q}$$ is 
totally $\delta(N)^{E_0}$-equidistributed.
Let $F:G/\Gamma \to \CC$ be any $1$-bounded Lipschitz function such that
$\int_{G/\Gamma}F=0$, and let $I \subset \{1, \dots, T/Q\}$ 
be any discrete interval of length at least $T/(Q (\log N)^E)$.
Then:
\begin{align} \label{eq:prop-bound}
&\bigg|
\frac{Q}{T}
\sum_{n \in I} 
f\left(Qn+A\right)
F(g(n)\Gamma)
\bigg| \ll_{d,m_G, \alpha_f, H,E} \\
\nonumber
&
\left\{(\log \log T)^{-1/(2^{2d+3} \dim G)} 
+ \frac{\delta(N)^{-10^d \dim G}}{(\log \log T)^{1/2^{d+2}}}\right\}
\frac{1 +\|F\|_{\mathrm{Lip}}}{\log N}
\frac{Q}{\phi(Q)}
\prod_{\substack{p\leq N\\p\nmid Q}}
\left(1 + \frac{|f(p)|}{p}\right).
\end{align}
\end{proposition}

\begin{proof}[Proof of Theorem \ref{t:non-corr} assuming Proposition 
\ref{p:equid-non-corr}]
We loosely follow the strategy of \cite[\S 2]{GT-nilmobius}.
In view of the final error term in \eqref{eq:main}, we may assume 
that $M_0 \leq \log N$, as the theorem holds trivially otherwise.
This implies that $\mathcal{X}$ is a $(\log N)$-rational Mal'cev basis. 
Applying the factorisation lemma from above with $T$ replaced by 
$T/\widetilde{W}$, with $M_0=\log N$, and with a parameter $B>1$ that will be 
determined in course of the proof (as parameter $E_0$ in an application of
Proposition \ref{p:equid-non-corr}), we obtain a factorisation of $g$ as $\eps g' 
\gamma$ with properties (1)--(3) from Lemma \ref{l:factorisation}.
In particular, there is $M$ such that 
$\log N \leq M \leq (\log N)^{O_{B,m_G,d}(1)}$ and such that $g'$ takes 
values in 
a $M$-rational subgroup $G'$ of $G$ and is $M^{-B}$-equidistributed in 
$G'\Gamma/\Gamma$.
Our first aim is to decompose the summation range of $n$ in \eqref{eq:main} 
into subprogressions on which the three functions $\gamma$, $\eps$ and 
$(\widetilde W n + A)^{it_N}$ are all almost constant.

Since $\gamma$ is periodic with some period $a \leq M$, the function
$n \mapsto \gamma(an + b)$ is constant for every $b$, that is, 
$\gamma$ is constant on every progression 
$$P_{a,b}:=\{  n \in [1, T/\widetilde{W}] : n \equiv b \Mod{a}\},$$
where $0\leq b < a$.
Let $\gamma_b$ denote the value $\gamma$ takes on $P_{a,b}$ and
note that $$|P_{a,b}| \geq T/(2a\widetilde{W}) \geq T/(2M\widetilde{W}).$$
Let $g'_{a,b} : \ZZ \to G'$ be defined via
$$g'_{a,b}(n) = g'(an + b).$$
Since $(g'(n)\Gamma)_{n \leq T/\widetilde{W}}$ is totally 
$M^{-B}$-equidistributed in $G'\Gamma/\Gamma$, it is clear that every 
finite subsequence $(g'_{a,b}(n)\Gamma)_{n \leq T/(Ca\widetilde{W})}$
is $M^{-B/2}$-equidistributed if $a$, $b$ and $C$ are such that
both $0 \leq b < a \leq M$ and $C>0$ and, furthermore, $M^{B/2}>Ca$ 
hold.

Let $R \geq 1$ be an integer that will be chosen later depending on $d$ 
and $\dim G$.
By splitting each progression $P_{a,b}$ into $\ll M (\log \log N)^{1/R}$ 
pieces $P_{a,b}^{(j)}$ of diameter bounded by 
$\ll T/(M\widetilde{W} (\log \log N)^{1/R}),$ we may also arrange for 
$\eps$ and, simultaneously, for $(\widetilde W n + A)^{it_N}$ to be almost 
constant.
More precisely, the fact that $\eps$ is $(M,T/\widetilde{W})$-smooth implies 
that
$$
d_{\mathcal{X}}(\eps(n),\eps(n')) 
\leq |n-n'| M \widetilde{W}  T^{-1}
\ll (\log \log N)^{-1/R}
$$
for all $n,n' \leq T/\widetilde{W}$ with 
$|n-n'| \ll T/(M \widetilde{W} (\log \log N)^{1/R})$.
By choosing $B$ sufficiently large, 
we may ensure that $M^{B/2} \geq M \log \log N$ and, hence, that the 
equidistribution properties of $g'_{a,b}$ are preserved on the new bounded 
diameter pieces of $P_{a,b}$.
Let $\mathcal{P}$ denote the collection of all progressions $P_{a,b}^{(j)}$ in 
our decomposition.

Since $F$ is a Lipschitz function and since $d_{\mathcal{X}}$ is right-invariant 
(cf.\ \cite[Appendix A]{GT-polyorbits}), we deduce that
\begin{align}\label{eq:Lipschitz-application}
\nonumber
|F(\eps(n) g'(n) \gamma(n))
  -F(\eps(n') g'(n) \gamma(n))|
&\leq 
(1 + \|F\|)
d(\eps(n),\eps(n')) \\
&\ll 
(1 + \|F\|) (\log \log N)^{-1/R}
\end{align}
for all $n,n' \in P_{a,b}$ with 
$|n-n'| \ll T/(M\widetilde{W} (\log \log N)^{1/R} )$.
Thus, this bound holds in particular for any $n,n' \in P_{a,b}^{(j)}$.

To ensure that $(\widetilde W n + A)^{it_N}$ is almost constant on the bounded 
parameter progressions $P_{a,b}^{(j)}$ that we consider, let 
$\mathcal{P}' \subset \mathcal{P}$ denote the subset of progressions  
$P_{a,b}^{(j)}$ that are completely contained in the interval
$[T/(\widetilde{W} (\log \log N)^{1/(2R)}),T/\widetilde{W}]$.
Observe that the contribution of all other progressions 
$P_{a,b}^{(j)} \in \mathcal{P} \setminus \mathcal{P}'$ to 
\eqref{eq:main} may be bounded by
\begin{align*}
(\log \log N)^{-1/(2R)}
\frac{\widetilde W(N)}{\phi(\widetilde W(N))}
\frac{2}{\log T} 
\prod_{\substack{p\leq N\\ p\nmid 
\widetilde W(N)}}\left(1+\frac{|f(p)|}{p}\right),
\end{align*}
where we used Shiu's bound \eqref{eq:shiu} together with fact that 
we are only summing over $n \leq T/(\widetilde{W}(\log \log N)^{1/(2R)})$.
Since $\frac{\log \log N}{\log \log \log N} < w(N) \leq \log \log N$ and since
$1 \leq R \ll_{d, \dim G} 1$, we have
\begin{equation}\label{eq:loglogN^-1/R}
(\log \log N)^{-1/(2R)} \ll_{d,\dim G} (\log w(N))^{-1}, 
\end{equation}
which implies that the above contribution is negligible when compared to the 
bound in \eqref{eq:main}.
For every remaining progression $P_{a,b}^{(j)} \in \mathcal{P}'$, the diameter 
is now short compared to the size of the endpoints and we have
\begin{align*}
\log (\widetilde W n+ A) 
&= \log (\widetilde W n'+ A)
  + \log \frac{\widetilde W (n' + n - n') + A}{\widetilde W n'+ A} \\
&= \log (\widetilde W n'+ A) 
  + \log \bigg(1 + 
  O\bigg( \frac{1}{M(\log \log N)^{1/(2R)}} \bigg) \bigg) \\
&= \log (\widetilde W n'+ A) 
  +  
  O\bigg( \frac{1}{M(\log \log N)^{1/(2R)}} \bigg)\\
\end{align*}
for all $n,n' \in P_{a,b}^{(j)}$.
Since $|t_N| \leq 2 \log N$ and $M \geq \log N$, we deduce that 
\begin{align}\label{eq:n^it-factor}
\nonumber
 (\widetilde W n+ A )^{it_N} 
 &= (\widetilde W n'+ A )^{it_N} 
    \exp\bigg( O\Big( \frac{\log N}{M(\log \log N)^{1/(2R)}} \Big)\bigg) \\
\nonumber    
 &= (\widetilde W n'+ A )^{it_N} (1 + O((\log \log N)^{-1/(2R)})) \\
 &= (\widetilde W n'+ A )^{it_N} + O((\log\log N)^{-1/(2R)})
\end{align}
for all $n,n' \in P_{a,b}^{(j)}$.

Let us fix one element $n_{b,j}$ for each progression 
$P_{a,b}^{(j)} \in \mathcal{P}'$.
As we will show next, it will be sufficient to bound the correlation
\begin{align}\label{eq:P-abj-corr}
&\bigg|
\sum_{n \in P_{a,b}^{(j)}} 
\bigg(
f(\widetilde{W}n +A) 
- (\widetilde{W}n_{b,j}+A)^{it_N}
S_{f(n)n^{-it_N}}(N;\widetilde{W},A)
\bigg)
F(\eps(n_{b,j}) g'(n) \gamma_{b})\Gamma)
\bigg| = \\ 
\nonumber
&\bigg| 
\sum_{\substack{n: \\ an+b \in P_{a,b}^{(j)}}} 
\bigg(
f(\widetilde{W}(an+b)+A) 
- 
(\widetilde{W}n_{b,j}+A)^{it_N}
S_{f(n)n^{-it_N}}(N;\widetilde{W},A)
\bigg)
F(\eps(n_{b,j}) g'_{a,b}(n) \gamma_{b})\Gamma)
\bigg| 
\end{align}
for each bounded diameter piece $P_{a,b}^{(j)} \in \mathcal{P}'$.
Indeed, the estimates \eqref{eq:Lipschitz-application} and 
\eqref{eq:n^it-factor} applied with $n'= n_{b,j}$ 
to each such progression, show that the error term incurred from this 
reduction satisfies
\begin{align*}
\nonumber
&\bigg|
\sum_{P_{a,b}^{(j)} \in \mathcal{P}'}
\sum_{n \in P_{a,b}^{(j)}}
\bigg\{ \bigg(f(\widetilde{W}n +A) - 
(\widetilde{W}n+A)^{it_N}
S_{f(n)n^{-it_N}}(N;\widetilde{W},A) \bigg)
F(\eps(n) g'(n) \gamma(n)) \\
& \hspace*{2.6cm} -
\bigg(f(\widetilde{W}n +A) - 
(\widetilde{W}n_{b,j}+A)^{it_N}
S_{f(n)n^{-it_N}}(N;\widetilde{W},A) \bigg)
F(\eps(n_{b,j}) g'(n) \gamma_b)\bigg\}\bigg|\\
&\leq \sum_{P_{a,b}^{(j)} \in \mathcal{P}'}
\sum_{n \in P_{a,b}^{(j)}}
|f(\widetilde{W}n +A)| 
\Big|F(\eps(n) g'(n) \gamma(n))
  -F(\eps(n_{b,j}) g'(n) \gamma_b)\Big|\\
&\quad
+\sum_{P_{a,b}^{(j)} \in \mathcal{P}'}
\sum_{n \in P_{a,b}^{(j)}}
\Big|(\widetilde{W}n+A)^{it_N} \Big| 
\Big|F(\eps(n) g'(n) \gamma(n))
  -F(\eps(n_{b,j}) g'(n) \gamma_b)\Big|
 S_{|f|}(N;\widetilde{W},A) 
  \\  
&\quad 
+\sum_{P_{a,b}^{(j)} \in \mathcal{P}'}
\sum_{n \in P_{a,b}^{(j)}}
\Big| 
( \widetilde{W}n+A)^{it_N}
-( \widetilde{W}n_{b,j}+A)^{it_N}
\Big|
\Big|F(\eps(n_{b,j}) g'(n) \gamma_b)\Big|
S_{|f|}(N;\widetilde{W},A)
\\
&\ll 
\frac{T}{\widetilde W(N)}
\frac{(1 + \|F\|)S_{|f|}(N;\widetilde{W},A)}{(\log \log N)^{1/R}}.
\end{align*}
By Shiu's bound \eqref{eq:shiu}, this in turn is bounded above by
\begin{align}\label{eq:eps-error}
 \ll 
 \frac{T}{\widetilde W(N)}
 \frac{(1 + \|F\|)}{(\log \log N)^{1/R}}
 \frac{\widetilde W(N)}{\phi(\widetilde W(N))} \frac{1}{\log N}
\exp\bigg(\sum_{w(N) < p \leq N } 
\frac{|f(p)|}{p}\bigg)~.
\end{align}
Taking into account \eqref{eq:loglogN^-1/R}, the error term 
\eqref{eq:eps-error} is acceptable in view of the bound in 
\eqref{eq:main}.

We aim to estimate the correlation \eqref{eq:P-abj-corr} with the help of 
Proposition \ref{p:equid-non-corr}.
This task will be carried out in four steps, the first of which will be to 
bound the contribution from non-invertible residues 
$\widetilde{W}b + A \Mod{\widetilde{W}a}$ to which Proposition 
\ref{p:equid-non-corr} does not apply.
The two subsequent steps consist of checking the various assumptions of 
Proposition \ref{p:equid-non-corr}, while the fourth step contains the actual 
application of the proposition.

Before we start, we record a final estimate that will be used throughout 
the rest of the proof.
Note that the common difference of $P_{a,b}^{(j)}$ satisfies 
$a \leq M \ll (\log N)^{O_{d,m_G,B}(1)}$, which is bounded above by 
$(\log N)^E$, provided $E$ is sufficiently large in terms of $d$, $m_G$ and 
$B$. 

\paragraph*{\em Step 1: Non-invertible residues}
We seek to bound the contribution to \eqref{eq:main} of all progressions 
$P_{a,b}^{(j)} \in \mathcal{P}'$ with $\gcd(\widetilde{W}b + 
A,\widetilde{W}a)>1$.
Let $a' = \prod_{p \nmid \widetilde W(N)} p^{v_p(a)}$, so that
$\widetilde W(N)$ is invertible modulo $a'$.
Since $\gcd(A,\widetilde{W})=1$, it suffices to check whether $b$ satisfies
$\gcd(\widetilde{W}b + A,a')>1$.
Thus, the contribution we seek to bound takes the form
\begin{align*}
  &\frac{\widetilde{W}}{T}
 \sum_{d|a',\,d>1} 
 \sum_{\substack{b < a:\\ \gcd(\widetilde{W}b + A,a')=d~}}
 \sum_{\substack{n < T/\widetilde{W}  \\ n\equiv b \Mod{a}}} 
 \Big\{|f(\widetilde{W}n+A)| + S_{|f|}(N;\widetilde W,A) \Big\}.
\end{align*}
The contribution from the terms involving $S_{|f|}(N;\widetilde W,A)$ is
bounded by
\begin{align} \label{eq:non-triv-residues-1}
 &\ll S_{|f|}(N;\widetilde W,A)
 \sum_{d|a',\,d>1} 
 \sum_{\substack{b < a:\\ \gcd(\widetilde{W}b + A,a')=d~}} \frac{1}{a} \\
 \nonumber
 &\ll S_{|f|}(N;\widetilde W,A)
 \sum_{d|a',\,d>1} \frac{1}{a} \frac{a}{a'}
 \phi\Big(\frac{a'}{d}\Big) \\
 \nonumber
 &\ll S_{|f|}(N;\widetilde W,A)
 \sum_{d|a',\,d>1} \frac{1}{d},
\end{align}
where we used the fact that $\widetilde W(N)$ is invertible modulo $a'$.
In a similar fashion, we may bound the contribution from those terms involving
$|f(\widetilde{W}n+A)|$ as follows:
\begin{align*}
 &\frac{\widetilde{W}}{T}
 \sum_{d|a',\,d>1} 
 \sum_{\substack{b < a:\\ \gcd(\widetilde{W}b + A,a')=d~}}
 \sum_{\substack{n < T/\widetilde{W}  \\ n\equiv b \Mod{a}}} 
 |f(\widetilde{W}n+A)| \\
 & \leq \sum_{d|a',\,d>1} 
 \sum_{\substack{b < a:\\ \gcd(\widetilde{W}b + A,a')=d~}}
 \frac{|f(d)|}{a}
 ~S_{|f|} 
  \left(\frac{T}{d};\frac{\widetilde{W}a}{d},\frac{\widetilde{W}b+A}{d}\right)
 \\
 & \leq \sum_{d|a',\,d>1}
 \frac{|f(d)|}{a}
 \frac{a}{a'}
 \phi\Big(\frac{a'}{d}\Big)
 ~S_{|f|}
  \left(\frac{T}{d};\frac{\widetilde{W}a}{d},\frac{\widetilde{W}b+A}{d}\right)
 \\
  & \ll \sum_{d|a',\,d>1} 
 \frac{|f(d)|}{a'}
 \phi\Big(\frac{a'}{d}\Big)
 \frac{\widetilde{W}a/d}{\phi(\widetilde{W}a/d)}
 \frac{1}{\log (T/d)}
 \exp \bigg( \sum_{ \substack{ p < T/d \\ p \nmid \widetilde W a'/d}} 
 \frac{|f(p)|}{p} \bigg) \\
 & \leq \sum_{d|a',\,d>1} 
 \frac{|f(d)|}{a'}
 \phi\Big(\frac{a'}{d}\Big)
 \frac{a'/d}{\phi(a'/d)}
 \frac{\widetilde W}{\phi(\widetilde W)}
 \frac{1}{\log (T/d)}
 \exp \bigg( \sum_{ \substack{ p < T \\ p \nmid \widetilde W }} 
 \frac{|f(p)|}{p} \bigg) \\
& \ll 
 \frac{\widetilde W}{\phi(\widetilde W)}
 \frac{1}{\log T}
 \exp \bigg( \sum_{ \substack{ p < T \\ p \nmid \widetilde W}} 
 \frac{|f(p)|}{p} \bigg)
 \sum_{d|a',\,d>1} 
 \frac{|f(d)|}{d}, 
\end{align*}
where we made use of \eqref{eq:shiu} and of the fact that
$d\leq a \leq (\log N)^{E}$ so that
$\log (T/d) \geq (\log T)/2$ once $N$ and, hence, $T$ are sufficiently large.
Observe that the final sums in each of the two bounds above are similar.
We restrict attention to bounding the latter of them.
Assuming that the lower bound $w(N)$ on prime divisors $p|a'$ is sufficiently 
large with respect to $H$, we have
\begin{align*}
 \sum_{d|a',\,d>1} \frac{|f(d)|}{d} 
&\leq \prod_{p|a'}\left(1+\frac{H}{p} + \frac{H^2}{p^2} +  \dots \right)-1 
 \leq \prod_{p|a'}\left(1 + \frac{H}{p}\right)
                   \left(1 + \frac{H^2}{p^2(1 - \frac{H}{p})}\right)-1\\
&\leq \exp \bigg( \sum_{p|a'} \frac{2 H^2}{p^2} \bigg) 
      \prod_{p|a'}\left(1 + \frac{H}{p}\right) -1                
\leq \left(1+\frac{4H^2}{w(N)}\right)
\prod_{p|a'}\left(1 + \frac{H}{p}\right) - 1 \\
& \ll_{E,H} \frac{4H^2}{w(N)}+\frac{1}{\log w(N)}
 \ll_{E,H} \frac{1}{\log w(N)},
\end{align*}
where we applied Lemma \ref{l:a'} with $a$ replaced by $a'$ to estimate the 
product over $p|a'$.

Bounding the inner sum in \eqref{eq:non-triv-residues-1} in a similar 
fashion and applying \eqref{eq:shiu} to estimate $S_{|f|}(N; \widetilde W,A)$, 
we deduce that the total contribution of non-invertible residues 
$\widetilde{W}b + A \Mod{\widetilde{W}a}$ to \eqref{eq:main} is at most
$$
O_{d,m_G,B,H}\left(\frac{1}{\log w(N)}
 \frac{\widetilde W}{\phi(\widetilde W)}
 \frac{1}{\log T}
 \exp \left( \sum_{w(N)< p < T } 
 \frac{|f(p)|}{p} \right)\right),
$$
which has been taken care of in \eqref{eq:main}. 
This leaves us to considering the case where the value of $b$ does not impose an 
obstruction to applying Proposition \ref{p:equid-non-corr}.

\paragraph*{\em Step 2: Checking the initial conditions of Proposition 
\ref{p:equid-non-corr}}
The central assumption of Proposition \ref{p:equid-non-corr} concerns the 
equidistribution of the polynomial sequence it is applied to.
To verify this assumption for the sequence that appears in
\eqref{eq:P-abj-corr}, it is necessary 
to show that the conjugated sequence 
$h^*: n \mapsto \gamma_{b}^{-1} g'_{a,b}(n) \gamma_{b}$ 
is, in fact, a polynomial sequence and that it inherits the equidistribution 
properties of $g'_{a,b}(n)$.
Both these questions have been addressed in \cite[\S2]{GT-nilmobius} in a way we 
can directly build on:
Let $H=\gamma_{b}^{-1} G' \gamma_{b}$ and define
$H_{\bullet} = \gamma_{b}^{-1} (G')_{\bullet} \gamma_{b}$.
Let $\Lambda = \Gamma \cap H$ and define
$F_{b,j}: H/\Lambda \to \RR$ via
$$
F_{b,j}(x \Lambda) 
= F(\eps(n_{b,j}) \gamma_{b} x \Gamma).
$$
Then $h^* \in \mathrm{poly}(\ZZ,H_{\bullet})$ and the correlation 
\eqref{eq:P-abj-corr} that we seek to bound takes the form
\begin{align} \label{eq:F-b-j}
\bigg|
\sum_{\substack{n: (an+b) \in  P_{a,b}^{(j)}}} 
\bigg(
f(\widetilde{W}(an+b)+A) - 
(\widetilde{W}n_{b,j}+A)^{it_N}
S_{f(n)n^{-it_N}}(N;\widetilde{W},A)
\bigg)
F_{b,j}(h^*(n)\Lambda) \bigg|.
\end{align}
The `Claim' from the end of \cite[\S2]{GT-nilmobius} guarantees the existence of  
a Mal'cev basis $\mathcal{Y}$ for $H/\Lambda$ adapted to $H_{\bullet}$ 
such that each basis element $Y_i$ is a $M^{O(1)}$-rational combination of basis
elements $X_i$. 
Thus, there is $C'=O(1)$ such that $\mathcal{Y}$ is $M^{C'}$-rational.
Furthermore, it implies that there is $c'>0$, depending only on the dimension of 
$G$ and the degree of $G_{\bullet}$, such that whenever $B$ is sufficiently large 
the sequence 
\begin{equation}\label{eq:h-seq}
 (h^*(n) \Lambda)_{n \leq T/(a\widetilde{W})}
\end{equation}
is totally $M^{-c'B/2 + O(1)}$-equidistributed in $H/\Lambda$, equipped with 
the metric $d_{\mathcal{Y}}$ induced by $\mathcal{Y}$.
Taking $B$ sufficiently large, we may assume that the sequence \eqref{eq:h-seq} is
totally $M^{-c'B/4}$-equidistributed.
Finally, the `Claim' also provides the bound 
$\|F_{b,j}\|_{\mathrm{Lip}} \leq M^{C''} \|F\|_{\mathrm{Lip}}$ for some 
$C''=O(1)$.
This shows that all conditions of Proportion \ref{p:equid-non-corr} are satisfied 
except for $\int_{H/\Lambda}F_{b,j}=0$.

\paragraph*{\em Step 3: The final condition}
The final condition that needs to be arranged for before we can apply 
Proposition \ref{p:equid-non-corr} to \eqref{eq:F-b-j} is that  
$\int_{H/\Lambda}F_{b,j}=0$.
This is where the major arc condition \eqref{eq:major-arc} is needed, which in 
turn requires that $\gcd(\widetilde{W}b+A,\widetilde{W}a)=1$.
To ensure that the integral over the test function is zero, we decompose 
$F_{b,j}(x\Lambda)$ as 
$$(F_{b,j}(x\Lambda) - \mu_{b,j}) + \mu_{b,j},$$
where $\mu_{b,j}:=\int_{H/\Lambda}F_{b,j}$. 
The expression in brackets represents a new test function that
we can apply the proposition with, and we will show next that the contribution 
from the final term above is small provided 
$f(\widetilde{W}n+A)- 
(\widetilde{W}n_{b,j}+A)^{it_N}
S_{f(n)n^{-it_N}}(N;\widetilde{W},A)$ does 
not correlate with the characteristic function $\1_{P_{a,b}^{(j)}}$ of the 
corresponding progression $P_{a,b}^{(j)}$.

To start with, recall that $T \geq N/(\log N)^{E/2}$, that the common difference of 
$P_{a,b}^{(j)}$ satisfies $a \leq (\log N)^{E}$ and that the length of $P_{a,b}^{(j)}$ is 
bounded below by
$$|P_{a,b}^{(j)}|
\geq T/(2aM\widetilde{W} (\log \log N)^{1/R}) 
\gg T/(a\widetilde{W} (\log N)^{E/2}) \gg N/(a\widetilde{W} (\log N)^{E}),$$
provided $E$ is sufficiently large in terms of $d$, $m_G$ and $B$.
Observe that condition \eqref{eq:major-arc} applies to the function
$n \mapsto f(n)n^{-it_N}$ and to all discrete intervals 
$I \subset \{1, \dots,T/\widetilde{W}\}$ of length $|I|\gg T/(\log T)^{E}$.
In particular, we may choose $q=a$ and $r=b$ and let 
$I$ be a discrete interval of length $a \widetilde W(N) |P_{a,b}^{(j)}|$ that 
contains the set $\{ \widetilde W(N) m + A: m \in P_{a,b}^{(j)}\}$.
To relate $f(n)$ to $f(n)n^{-it_N}$, we observe that the estimates 
\eqref{eq:n^it-factor} and \eqref{eq:loglogN^-1/R} imply that
$$
f(\widetilde W n + A) 
= (\widetilde W n_{b,j} + A)^{it_N} 
  f(\widetilde W n + A)
  (\widetilde W n + A)^{-it_N} 
+ O\bigg(\frac{|f(\widetilde W n + A)|}{\log w(N)}\bigg)$$ 
for all $n, n_{b,j} \in P_{a,b}^{(j)} \in \mathcal{P}'$.

By applying condition \eqref{eq:major-arc} to the main below and Shiu's 
bound \eqref{eq:shiu} in combination with Corollary \ref{c:a'}
to the error term, we obtain the uniform estimate
\begin{align*}
&\frac{1}{|P_{a,b}^{(j)}|}
 \sum_{  m \in P_{a,b}^{(j)} } f(\widetilde{W} m + A)\\
&= (\widetilde{W} m_{b,j} + A)^{it_N}~
\frac{a \widetilde{W}}{|I|}
 \sum_{\substack{  m \in I 
 \\ m \equiv \widetilde{W}b + A ~(a\widetilde{W})}} 
 f(m) m^{-it_N}+ O\bigg( \frac{1}{\log w(N)} 
 \frac{a \widetilde{W}}{|I|}
 \sum_{\substack{  m \in I 
  \\ m \equiv \widetilde{W}b + A ~(a\widetilde{W})}} 
 |f(m)|\bigg)
\\
&= 
(\widetilde{W} m_{b,j} + A)^{it_N} S_{f(n)n^{-it_N}}(N;\widetilde{W},A) \\
&\qquad + 
 O\bigg( \bigg(\varphi'(N) + \frac{1}{\log w(N)} \bigg) 
 \frac{1}{\log N} \frac{\widetilde{W}}{\phi(\widetilde{W})} 
 \prod_{\substack{p<N \\ p \nmid \widetilde{W}}} \left(1 + \frac{|f(p)|}{p} \right)
 \bigg), 
\end{align*}
valid for all $P_{a,b}^{(j)} \in \mathcal{P}'$.

Let, as above, $\mu_{b,j}=\int_{H/\Lambda}F_{b,j}$, and note that 
$\mu_{b,j} \ll 1$.
Thus, the error term incurred by replacing for each $P_{a,b}^{(j)}$ 
with $\gcd(\widetilde{W}b+A,\widetilde{W}a)=1$ the factor $F_{b,j}(h(n)\Lambda)$ 
in \eqref{eq:F-b-j} by
$(F_{b,j}(h(n)\Lambda) - \mu_{b,j})$ is bounded as 
follows:
\begin{align*}
&\bigg|\frac{\widetilde{W}}{T}
\sum_{\substack{{P}_{a,b}^{(j)} \in \mathcal{P}' 
\\ \gcd(\widetilde{W} b + A,a)=1 }}
\mu_{b,j}
\sum_{n\in {P}_{a,b}^{(j)}}
\Big(f(\widetilde{W}n+A)-
(\widetilde{W}m_{b,j}+A)^{it_N}
S_{f(n)n^{-it_N}}(N;\widetilde{W},A)\Big) 
\bigg|\\
&\ll
\frac{\widetilde{W}}{T} 
\sum_{{P}_{a,b}^{(j)} \in \mathcal{P}'} |P_{a,b}^{(j)}|
\bigg(\varphi'(N) + \frac{1}{\log w(N)} \bigg) 
\frac{1}{\log N} \frac{\widetilde{W}}{\phi(\widetilde{W})} 
\prod_{\substack{p<N \\ p \nmid \widetilde{W}}} \left(1 + \frac{|f(p)|}{p} \right)\\
&\ll \bigg(\varphi'(N) + \frac{1}{\log w(N)} \bigg) 
\frac{1}{\log N} \frac{\widetilde{W}}{\phi(\widetilde{W})} 
\prod_{\substack{p<N \\ p \nmid \widetilde{W}}} \left(1 + \frac{|f(p)|}{p} \right),
\end{align*}
where $\varphi'$ is the function defined in Remark \ref{rem:phi'}.
This error term has been taken care of in the bound \eqref{eq:main}.

\paragraph*{\em Step 4: Application of Proposition \ref{p:equid-non-corr}}
The application of Proposition \ref{p:equid-non-corr} to \eqref{eq:F-b-j}
will give rise to the third error term in \eqref{eq:main}.
In view of the work carried out in Steps 1--3, we may now 
assume that $\gcd(\widetilde{W}b+A,\widetilde{W}a)=1$ 
and that $\int_{H/\Lambda}F_{b,j} = 0$ holds, and
apply Proposition \ref{p:equid-non-corr}
with:
\begin{itemize}
 \item $g=h$, $Q=\widetilde W a$, $I=\{n: an+b \in P_{a,b}^{(j)}\}$,
 \item  with a function $\delta:\NN \to \RR$ such that 
$\delta(N)=M^{-C'} (= M_0^{O_{d,m_G,B}(1)})$, which 
ensures that $\mathcal{Y}$ is $\frac{1}{\delta(N)}$-rational, 
 \item with $E$ sufficiently large to ensure that 
$M^{C'} < (\log N)^E$, which in particular means that $E$ depends on $B$,
and
 \item $E_0=c'B/4=O_{d, m_G}(B)$ for some value of $B$ that is sufficiently 
large to ensure that \eqref{eq:h-seq} is totally $M^{-c'B/4}$-equidistributed 
in 
$H/\Lambda$ (cf.\ Step 2) and that is also sufficiently large for Proposition 
\ref{p:equid-non-corr} to apply with the above choice of $E_0$.
\end{itemize}
Since there are $\ll aM(\log \log N)^{1/R}$ intervals $P_{a,b}^{(j)}$ in the 
decomposition $\mathcal{P}' \subset \mathcal{P}$, this yields the bound
\begin{align} \label{eq:deduction-final-bd}
\nonumber
\sum_{P_{a,b}^{(j)} \in \mathcal{P}'}
&\bigg|
\sum_{\substack{n:\, (an+b) \\ \in P_{a,b}^{(j)}}} 
\bigg(
f(\widetilde{W}(an+b)+A) - 
(\widetilde{W}n_{b,j}+A)^{it_N}
S_{f(n)n^{-it_N}}(N;\widetilde{W},A)
\bigg)
F_{b,j}(h^*(n)\Lambda) \bigg|\\
\nonumber
&\ll aM(\log \log N)^{1/R}
     \frac{1 + M^{O(1)} \|F\|}{\log T}
     \frac{T}{\widetilde{W}a}
     \frac{\widetilde W a}{\phi(\widetilde W a)}
     \prod_{\substack{p\leq N\\p\nmid \widetilde Wa}}
     \left(1 + \frac{|f(p)|}{p}\right)
     \mathcal{N} 
\\
&\ll 
     M^{O(1)}
     (\log \log N)^{1/R} 
     \frac{1 + \|F\|}{\log T}
     \frac{T}{\widetilde{W}}
     \frac{\widetilde W a}{\phi(\widetilde W a)}
     \prod_{\substack{p\leq N\\p\nmid \widetilde W}}
     \left(1 + \frac{|f(p)|}{p}\right)
     \mathcal{N},
\end{align}
where the implied constant depends on $d,m_G,\alpha_f,H$ and $B$, and where
\begin{align*}
\mathcal{N}=
(\log \log T)^{-1/(2^{2d+3} \dim G)} 
+ \frac{M^{10^d \dim G}}{(\log \log T)^{1/2^{d+2}}}
\ll \frac{M^{10^d \dim G}}{(\log \log T)^{1/(2^{2d+3} \dim G)}}.
\end{align*}
Finally, we invoke Corollary \ref{c:a'} to remove the dependence on $a$ from 
\eqref{eq:deduction-final-bd}.
We complete the deduction of Theorem \ref{t:non-corr} by setting 
$R = 2^{2d+3} \dim G$ and comparing the bound arising from 
\eqref{eq:deduction-final-bd} with the third term in \eqref{eq:main}.
\end{proof}
It remains to establish Proposition \ref{p:equid-non-corr}.

\section{Linear subsequences of equidistributed nilsequences}
\label{s:linear-subsecs}
Our aim in this section is to study the equidistribution properties of 
families 
$$
\Big\{ (g(Dn+D')\Gamma)_{n\leq T/D} : D \in [K,2K) \Big\}
$$
of linear subsequences of an equidistributed sequence $(g(n)\Gamma)_{n \leq T}$,
where $D$ runs through dyadic intervals $[K,2K)$ for $K \leq T^{1-1/H}$. 
This result will only be needed in the case of unbounded multiplicative 
functions, which allows us to assume that $H>1$ in this section.

We begin by recalling some essential definitions and notation. 
Let $P: \ZZ \to \RR/\ZZ$ be a polynomial of degree at most $d$ and let 
$\alpha_0, \dots, \alpha_d \in \RR/\ZZ$ be defined via
$$
P(n) 
= \alpha_0 + \alpha_1 \binom{n}{1} + \dots + \alpha_d \binom{n}{d}.$$
Then the \emph{smoothness norm} of $g$ with respect to $T$ is defined 
(c.f.\ Green--Tao \cite[Def.\ 2.7]{GT-polyorbits}) as
$$
\|P\|_{C^{\infty}[T]}
= \sup_{1 \leq j \leq d} T^j \|\alpha_j\|_{\RR/\ZZ}.
$$
If $\beta_0, \dots, \beta_d \in \RR/\ZZ$ are defined via
$$
P(n) = \beta_dn^d + \dots + \beta_1 n + \beta_0,
$$
then (cf.\ \cite[equation (14.3)]{lmr}) the smoothness norm is bounded 
above by a similar expression in terms of the $\beta_i$, namely
\begin{equation} \label{eq:alpha-beta}
\|P\|_{C^{\infty}[T]} 
\ll_d
\sup_{1 \leq j \leq d} T^j \|j!\beta_j\|_{\RR/\ZZ}
\ll_d \sup_{1 \leq j \leq d} T^j \|\beta_j\|_{\RR/\ZZ}.
\end{equation}
On the other hand, \cite[Lemma 3.2]{GT-polyorbits} shows that there is 
a positive integer $q \ll_d 1$ such that
$$
\|q \beta_j\|_{\RR/\ZZ} \ll T^{-j} \|P\|_{C^{\infty}[T]}.
$$
Apart from smoothness norms, we also require the notion of a 
horizontal character as defined in 
\cite[Definition 1.5]{GT-polyorbits}.
A continuous additive homomorphism $\eta: G \to \RR/\ZZ$ is called a
\emph{horizontal character} if it annihilates $\Gamma$.
In order to formulate quantitative results, one defines a height function $|\eta|$
for these characters.
A definition of this height, called the \emph{modulus} of $\eta$, may be 
found in \cite[Definition 2.6]{GT-polyorbits}.
All that we require to know about these heights is that there are at most 
$M^{O(1)}$ horizontal characters $\eta:G\to \RR/\ZZ$ of modulus 
$|\eta| \leq M$.

The interest in smoothness norms and horizontal characters lies in Green and
Tao's `quantitative Leibman Theorem':
\begin{proposition}[Green--Tao, Theorem 2.9 of \cite{GT-polyorbits}]
\label{p:leibman}
Let $m_G$ and $d$ be non-negative integers, let $0 < \delta < 1/2$ and let
$N \geq 1$.
Suppose that $G/\Gamma$ is an $m_G$-dimensional nilmanifold together with a 
filtration $G_{\bullet}$ of degree $d$ and that $\mathcal{X}$ is a 
$\frac{1}{\delta}$-rational Mal'cev basis adapted to $G_{\bullet}$. 
Suppose that $g \in \mathrm{poly}(\ZZ, G_{\bullet})$.
If $(g(n)\Gamma)_{n \leq N}$ is not $\delta$-equidistributed, then there 
is a non-trivial horizontal character $\eta$ with 
$0 < |\eta| \ll \delta^{-O_{d,m_G}(1)}$ such that
$$\|\eta \circ g \|_{C^{\infty}[N]} \ll \delta^{-O_{d,m_G}(1)} .$$
\end{proposition}

The following lemma shows that for polynomial sequences the notions 
of equidistribution and total equidistribution are equivalent with a polynomial 
dependence in the equidistribution parameter.
\begin{lemma} \label{l:equi/totally-equi}
 Let $N$ and $A$ be positive integers and let $\delta: \NN \to [0,1]$ be a 
 function that satisfies $\delta(x)^{-t} \ll_t x$ for all $t>0$.
 Suppose that $G$ has a $\frac{1}{\delta(N)}$-rational Mal'cev basis adapted to 
 the filtration $G_{\bullet}$. 
 Suppose that $g \in \mathrm{poly}(\ZZ,G_{\bullet})$ is a polynomial sequence 
 such that $(g(n)\Gamma)_{n\leq N}$ is $\delta(N)^A$-equidistributed.
 Then there is  $1\leq B \ll_{d,m_G} 1$
 such that $(g(n)\Gamma)_{n\leq N}$ is totally 
 $\delta(N)^{A/B}$-equidistributed, provided $A/B>1$ and provided $N$ is 
 sufficiently large.
\end{lemma}
\begin{rem}
 The Green--Tao factorisation theorem (cf.\ property (3) of Lemma 
 \ref{l:factorisation}) usually allows one to arrange for $A > B$ to hold.
\end{rem}

\begin{proof}
 We allow all implied constants to depend on $d$ and $m_G$.
 Let $B\geq1$ and suppose that $(g(n)\Gamma)_{n\leq N}$ fails to be totally 
 $\delta(N)^{A/B}$-equidistributed.
 Then there is a subprogression $P=\{ \ell n + b : 0 \leq n \leq m-1\}$ of
 $\{1, \dots, N\}$ of length $m > \delta(N)^{A/B}N$ such
 that the sequence $(\tilde g(n))_{0 \leq n < m}$, where
 $\tilde g(n)=g(\ell n + b)$, fails to be 
 $\delta(N)^{A/B}$-equidistributed.
 Provided $A>B$, Proposition \ref{p:leibman} implies that there is a non-trivial 
 horizontal character $\eta: G \to \RR/\ZZ$ of modulus
 $|\eta| < \delta(N)^{-O({A/B})}$ such that 
 $$
 \|\eta \circ \tilde g\|_{C^{\infty}[m]} \ll \delta(N)^{-O({A/B})}.
 $$ 
 The lower bound on $m$ implies that this is equivalent to the assertion 
 $$
 \|\eta \circ \tilde g\|_{C^{\infty}[N]} \ll \delta(N)^{-O({A/B})},
 $$ 
 where we recall that the implied constant may depend on $d$.
 
Observing that $\eta \circ g$ is a polynomial of degree at most $d$, let
$\eta \circ g(n)= \beta_d n^d + \dots + \beta_0$.
Then
 $$
 \eta \circ \tilde g(n) 
 = \sum_{i=0}^d n^i 
   \sum_{j=i}^d \beta_j \binom{j}{i}\ell^i b^{j-i},
 $$
and, hence, 
 $$
 \sup_{1 \leq i \leq d} N^i 
 \left \|\sum_{j=i}^d \beta_j \binom{j}{i}\ell^i b^{j-i} \right\|
 \ll \delta(N)^{-O({A/B})}. 
 $$
This yields the bound
 \begin{equation} \label{eq:beta-downwards}
 \left \|\sum_{j=i}^d \beta_j \binom{j}{i}\ell^i b^{j-i} \right\|
 \ll N^{-i} \delta(N)^{-O({A/B})} 
 \end{equation}
 for $1 \leq i \leq d$.
 Note that the lower bound on $m$ implies that $\ell < \delta(N)^{-A/B}$.
 Using a downwards induction argument, we aim to show that
 \begin{equation}\label{eq:beta-induction}
 \|\ell^d \beta_j \| \ll N^{-j} \delta(N)^{-O(A/B)}  
 \end{equation}
 for all $1\leq j \leq d$.
 For $j=d$, this is clear from the above.
 Suppose \eqref{eq:beta-induction} holds for all $j>i$.
 For each $i<j$ we then, in particular, have that 
 $$
 \left\|\ell^{d}  \beta_j \binom{j}{i} b^{j-i}\right\|
 \ll_d  \|\ell^{d} \beta_j \| b^{j-i}
 \ll_d  N^{-j} \delta(N)^{-O(A/B)} b^{j-i}
 \ll_d  N^{-i} \delta(N)^{-O(A/B)}.
 $$
 Using the fact that $\delta(N)^{-t} \ll_t N$ for all $t>0$, we deduce 
 that \eqref{eq:beta-induction} holds for $j=i$ from the above bounds and 
 from \eqref{eq:beta-downwards}. 
 This shows that there is a non-trivial horizontal character, namely
 $\ell^d \eta$, of modulus at most $\delta(N)^{-O(A/B)}$, such that 
 $$
 \|\ell^d \eta \circ g \|_{C^{\infty}[N]}
 \ll \sup_{1 \leq i \leq d} N^{i}  \|\ell^d \beta_i \|_{\RR/\ZZ}
 \ll \delta(N)^{-O({A/B})},
 $$
 where we made use of \eqref{eq:alpha-beta}.
 Choosing $B$ sufficiently large in terms of $m$ and $d$, 
 \cite[Proposition 14.2(b)]{lmr} implies that $g$ is not 
 $\delta(N)^A$-equidistributed, which is a contradiction.
\end{proof}

We are now ready to address the equidistribution properties of linear 
subsequences.

\begin{proposition} \label{p:linear-subsecs}
Let $H>1$, let $N$ and $T$ be as before and let $E_1 \geq 1$.
Let $(A_D)_{D \in \NN}$ be a sequence of integers such that 
$|A_D| \leq D$ for every $D \in \NN$.
Further, let $\delta: \NN \to (0,1)$ be a function that satisfies 
$\delta(x)^{-t} \ll_t x$ for all $t>0$.
Suppose $G/\Gamma$ has a $\frac{1}{\delta(N)}$-rational Mal'cev basis adapted to 
a filtration $G_{\bullet}$ of degree $d$.
Let $g \in \mathrm{poly}(G_{\bullet},\ZZ)$ be a polynomial sequence and suppose  
that the finite sequence $(g(n)\Gamma)_{n\leq T}$ is totally 
$\delta(T)^{E_1}$-equidistributed in $G/\Gamma$.
Then there is a constant $c_1 \in (0,1)$, depending only on $d$ and 
$m_G := \dim G$, 
such that the following assertion holds for all integers 
$$K \in [(\log T)^{\log \log T}, T^{1-1/H}],$$
provided $c_1E_1 \geq 1$.

Write $g_D(n) = g(Dn + A_D)$ and let $\mathcal{B}_{K}$ denote the set of integers 
$D \in [K,2K)$ for which 
$$(g_D(n)\Gamma)_{n \leq T/D}$$
fails to be totally $\delta(T)^{c_1 E_1}$-equidistributed. 
Then
$$
\# \mathcal{B}_{K}
\ll K \delta(T)^{c_1 E_1}.
$$
\end{proposition}
\begin{proof}
Let $K\in [(\log T)^{\log \log T}, T^{1-1/H}]$ be a fixed integer and let 
$c_1 > 0$ to be determined in the course of the proof.
Suppose that $E_1>1/c_1$.
Lemma \ref{l:equi/totally-equi} implies that for every 
$D \in \mathcal{B}_{K}$, the sequence 
$(g_D(n)\Gamma)_{n\leq T/ D}$
fails to be $\delta(T)^{c_1E_1B}$-equidistributed on $G/\Gamma$ for some $B>0$ 
only depending on $d$ and $m_G$.
We continue to allow implied constants to depend on $d$ and $m_G$.
By Proposition \ref{p:leibman}, there is a non-trivial horizontal character
$\eta_D: G \to \RR/\ZZ$ of magnitude $|\eta_D| \ll \delta(T)^{-O(c_1E_1)}$ such 
that
\begin{align} \label{eq:smoothness-D}
\| \eta_D \circ g_D\|_{C^{\infty}[T/D]} \ll 
\delta(T)^{-O(c_1E_1)}.
\end{align}
For each non-trivial horizontal character $\eta: G \to \RR/\ZZ$ we define the 
set
$$
\mathcal{D}_{\eta} = 
\left\{ D \in \mathcal{B}_K : \eta_D = \eta \right\}.
$$
Note that this set is empty unless $|\eta| \ll \delta(T)^{-O(c_1E_1)}$.
Suppose that
$$
\# \mathcal{B}_K \geq K \delta(T)^{c_1E_1}.
$$
By the pigeon hole principle, there is some $\eta$ of modulus
$|\eta| \ll \delta(T)^{-O(c_1E_1)}$ such that
$$
\# \mathcal{D}_{\eta} \geq K \delta(T)^{O(c_1E_1)}.
$$
Suppose
$$
\eta \circ g(n) 
= \beta_d n^d + \dots \beta_1 n + \beta_0 
$$
and let
$$
\eta \circ g_D(n) 
= \alpha_d^{(D)} n^d + \dots + \alpha_1^{(D)} n + \alpha_0^{(D)}
$$
for any $D \in \mathcal{B}_K$.
The quantities $\alpha_j^{(D)}$ and $\beta_j$ are linked through the relation
\begin{align} \label{eq:rec}
 \alpha_j^{(D)} = 
 D^{j} \sum_{i=j}^d \binom{i}{j} A_D^{i-j} \beta_i
\end{align}
for each $1 \leq j \leq d$.
Thus, the bound \eqref{eq:smoothness-D} on the smoothness norm 
asserts that
\begin{align} \label{eq:smoothness-D-2}
\sup_{1 \leq j \leq d} 
\frac{T^j}{K^j} 
\| \alpha_j^{(D)}  \|
\ll \delta(T)^{-O(c_1E_1)}. 
\end{align}
With a downwards induction we deduce from \eqref{eq:smoothness-D-2} and 
\eqref{eq:rec} that
\begin{align}\label{eq:pre-waring-D}
\sup_{1 \leq j \leq d} 
\frac{T^j}{K^j} 
\left\| D^j \beta_j \right\|
\ll \delta(T)^{-O(c_1E_1)}. 
\end{align}
The bound \eqref{eq:pre-waring-D} provides information on rational approximations 
of $D^j \beta_j$ for many values of $D$. 
Our next aim is to use this information in order to deduce information on rational 
approximations of the $\beta_j$ themselves.
To achieve this, we employ the Waring trick that appeared in the \emph{Type I} 
sums analysis in \cite[\S3]{GT-nilmobius}, and begin by recalling the two 
lemmas that this trick rests upon.
The first one is a recurrence result, \cite[Lemma 3.2]{GT-polyorbits}.
\begin{lemma}[Green--Tao \cite{GT-polyorbits}]\label{l:GT-rec}
Let $\alpha \in \RR$, $0< \delta <1/2$ and $0 < \sigma < \delta/2$, 
and let $I \subseteq \RR/\ZZ$ be an interval of length $\sigma$ such that 
$\alpha n \in I$ for at least $\delta N$ values of $n$, $1 \leq n \leq N$. 
Then there is some $k \in \ZZ$ with 
$0 < |k| \ll \delta^{-O(1)}$ such that 
$\| k\alpha \| \ll \sigma \delta^{-O(1)}/N$.
\end{lemma}
The second, \cite[Lemma 3.3]{GT-nilmobius}, is a consequence of the asymptotic 
formula in Waring's problem.
\begin{lemma}[Green--Tao \cite{GT-nilmobius}] \label{l:GT-waring}
Let $K \geq 1$ be an integer, and suppose that 
$S \subseteq \{1, \dots, K\}$ is a set of size $\alpha K$.
Suppose that $t\geq 2^j + 1$. Then $ \gg_{j,t} \alpha^{2t}K^j$ integers in 
the interval $[1,tK^j]$ can be written in the form 
$k_1^j + \dots + k_t^j$, $k_1, \dots, k_t \in S$.
\end{lemma}
Returning to the proof of Proposition \ref{p:linear-subsecs}, let us consider the
set
\begin{align*}
\overline{\mathcal{D}}_j
= \left\{ m \leq s(2K)^j:
\begin{array}{l}
m =  D_1^j + \dots + D_s^j \\
D_1, \dots, D_s \in \mathcal{D}_{\eta}
\end{array}
\right\} 
\end{align*}
for some $s \geq 2^j+1$.
Each element $m$ of this set satisfies
\begin{align}\label{eq:bd-waring-D}
\|\beta_j m \|
\ll \delta(T)^{-O(c_1)} (K/T)^j, \quad 1\leq j \leq d,  
\end{align}
in view of \eqref{eq:pre-waring-D}.
Thus, Lemma \ref{l:GT-waring} implies that there are
$$
\# \overline{\mathcal{D}}_j
\gg \delta(T)^{O(c_1E_1)} K^j 
$$ 
elements in this set.
In view of the restrictions on $K$ and the assumptions on the function 
$\delta(x)$, the conditions of Lemma \ref{l:GT-rec} (on $\sigma$ and $\delta$) 
are satisfied provided $T$ is sufficiently large.
We conclude that there is an integer $k_j$ such that
$$1 \leq k_j 
\ll \delta(T)^{-O(c_1E_1)}
$$ 
and such that 
\begin{align*} 
\| k_j \beta_j \|
\ll 
 \delta(T)^{-O(c_1E_1)} 
 T^{-j}.
\end{align*}
Thus
\begin{align}\label{eq:alpha_j-kappa_j-D}
\beta_j = \frac{a_j}{\kappa_j} + \tilde\beta_j,
\end{align}
where $\kappa_j| k_j $, $\gcd(a_j,\kappa_j)=1$
and 
$$0 \leq \tilde\beta_j \ll 
 \delta(T)^{-O(c_1E_1)}  
 T^{-j}.$$
Hence,
\begin{align}\label{eq:kappa_j-D}
\| \kappa_j \beta_j \|
\ll \delta(T)^{-O(c_1E_1)} 
 T^{-j}.
\end{align}
Let $\kappa=\lcm(\kappa_1, \dots, \kappa_d)$ and set 
$\tilde \eta = \kappa \eta$.
We proceed as in \cite[\S3]{GT-nilmobius}: 
The above implies that
$$
\|\tilde \eta \circ g (n)\|_{\RR/\ZZ} \ll
\delta(T)^{-O(c_1E_1)} n/T,
$$
which is small provided $n$ is not too large.
Indeed, if $T'= \delta(T)^{c_1E_1C}T$ for some sufficiently large constant 
$C \geq 1$, only depending on $d$ and $m_G$, and if $n \in \{1, \dots , T'\}$, 
then
$$
\|\tilde \eta \circ g (n)\|_{\RR/\ZZ} \leq 1/10.
$$
Let $\chi: \RR/\ZZ \to [-1,1]$ be a function of bounded Lipschitz norm that 
equals $1$ on $[-\frac{1}{10},\frac{1}{10}]$ and satisfies $\int_{\RR/\ZZ} 
\chi(t)\d t = 0$.
Then, by setting $F := \chi \circ \tilde \eta$, we obtain a Lipschitz function 
$F:G/\Gamma \to [-1,1]$ that satisfies $\int_{G/\Gamma} F = 0$ and 
$\|F\|_{\mathrm{Lip}} \ll \delta(T)^{-O(c_1E_1)}$.
Choosing, finally, $c_1$ sufficiently small, only depending on $d$ and $m_G$, we 
may ensure that
$$\|F\|_{\mathrm{Lip}} < \delta(T)^{-E_1}$$
and, moreover, that
$$
T'> \delta(T)^{E_1}T.
$$
This choice of $T'$, $F$ and $c_1$ implies that
$$
\Big| \frac{1}{T'} \sum_{1 \leq n \leq T'} F(g(n)\Gamma) \Big|
= 1 > \delta(T)^{E_1}  \|F\|_{\mathrm{Lip}},
$$
which contradicts the fact that
$(g(n)\Gamma)_{n\leq T}$ is totally $\delta(T)^{E_1}$-equidistributed.
This completes the proof of the proposition.
\end{proof}

\section{Equidistribution of product nilsequences} \label{s:products}
In this section we prove, building on material and techniques from 
\cite[\S 3]{GT-nilmobius}, a result on the equidistribution of 
products of nilsequences which will allow us to perform 
applications of the Cauchy--Schwarz inequality in 
Section \ref{s:proof-of-prop}.
The specific form of the result is adjusted to the requirements of 
Section \ref{s:proof-of-prop}.

We begin by introducing the product sequences we shall be interested in.
Suppose $g \in \mathrm{poly}(G_{\bullet},\ZZ)$ is a polynomial 
sequence.
This is equivalent to the assertion that there exists an integer $k$, 
elements $a_1,\dots, a_k$ of $G$, and integral polynomials 
$P_1 ,\dots, P_k \in \ZZ[X]$ such that
$$g(n)=a_1^{P_1(n)}a_2^{P_2(n)} \dots a_k^{P_k(n)}.$$
Then, for any pair of integers $(m, m')$, the sequence
$n \mapsto (g(mn),g(m'n)^{-1})$ is a polynomial sequence on $G\times G$ 
that may be represented by 
$$(g(mn),g(m'n)^{-1})
= \left(\prod_{i=1}^k (a_i,1)^{P_i(mn)}\right)
\left(\prod_{i=1}^k(1,a_i)^{P_i(m'n)}\right)^{-1}.$$
The horizontal torus of $G \times G$ arises as the direct product
$G/\Gamma[G,G]\times G/\Gamma[G,G]$ of horizontal tori for $G$.
Let $\pi:G \to G/\Gamma[G,G]$ be the natural projection map. 
Any horizontal character on $G \times G$ restricts to a horizontal
character on each of its factors. 
Thus, it takes the form 
$\eta \oplus \eta'(g_1,g_2):= \eta(g_1) + \eta'(g_2)$ for horizontal 
characters $\eta,\eta'$ of $G$.
The following proposition will be applied in the proof of Proposition 
\ref{p:equid-non-corr} to sequences $g=g_D$ for unexceptional $D$ in the sense of 
Proposition \eqref{p:linear-subsecs}.
\begin{proposition}\label{p:equid}
Let $N$ and $T$ be as before and let $E_2 \geq 1$.
Let $(\tilde D_m)_{m \in \NN}$ be a sequence of integers satisfying
$|\tilde D_m| < m$ for every $m \in \NN$.
Further, let $\delta: \NN \to (0,1)$ be a function that satisfies 
$\delta(x)^{-t} \ll_t x$ for all $t>0$.
Suppose $G/\Gamma$ has a $\frac{1}{\delta(T)}$-rational Mal'cev basis adapted 
to a filtration $G_{\bullet}$ of degree $d$. 
Let $P \subset \{1,\dots, T \}$ be a discrete interval.
Suppose $F: G/\Gamma \to \CC$ is a $1$-bounded function of bounded Lipschitz 
norm $\|F\|_{\mathrm{Lip}}$ and suppose that $\int_{G/\Gamma} F = 0$.
Let $g \in \mathrm{poly}(G_{\bullet},\ZZ)$ and suppose that
the finite sequence $(g(n)\Gamma)_{n\leq T}$ is 
totally $\delta(T)^{E_2}$-equidistributed in $G/\Gamma$.
Then there is a constant $c_2 \in (0,1)$, only depending on $d$ and 
$m_G:= \dim G$, such that the following assertion holds for all integers 
$$K \in 
\left[\exp \left( (\log \log T)^2 \right), 
\exp \left(\frac{1}{H}\left(\log T -(\log T)^{1/U}\right)\right)\right],$$
where $1 < U \ll 1$, provided $c_2E_2\geq1$.

Let $\mathcal{E}_{K}$ denote the set of integer pairs $(m,m') \in (K,2K]^2$
such that the discrete interval
$$
I_{m,m'}
=\left\{n \in \NN : 
\begin{array}{c}
nm + \tilde D_m \in P, \\
nm' + \tilde D_{m'} \in P
\end{array}
\right\}
$$
has length at least $$\# I_{m,m'}> \delta(N)^{c_2E_2} T/K,$$
and such that
\begin{align*}
\bigg|
\sum_{\substack{
n \in I_{m,m'}}}
F\Big(g\Big(mn + \tilde D_m \Big)\Gamma\Big)
\overline{F\Big(g\Big(m'n+\tilde D_{m'}\Big)\Gamma\Big)} 
\bigg|
> (1 + \|F\|_{\mathrm{Lip}}) \delta(T)^{c_2E_2}~\#I_{m,m'}
\end{align*}
holds.
Then, 
$$
\# \mathcal{E}_{K} 
< K^2 \delta(T)^{O(c_2E_2)}, 
$$
uniformly for all $K$ as above.
\end{proposition}

\begin{rem}
Using a trivial bound when $\# I_{m,m'} \leq \delta(N)^{c_2E_2} T/K$, we
deduce that
\begin{align*}
\bigg|
\sum_{n \in I_{m,m'}}
F\Big(g\Big(mn + \tilde D_m \Big)\Gamma\Big)
\overline{F\Big(g\Big(m'n+\tilde D_{m'}\Big)\Gamma\Big)} 
\bigg|
<
\frac{(1 + \|F\|_{\mathrm{Lip}}) \delta(T)^{c_2E_2} T}{K}
\end{align*}
for all $(m,m')\in (K,2K]^2\setminus  \mathcal{E}_{K}$.
\end{rem}

\begin{rem}\label{r:primes}
 The above proposition essentially continues to hold when the variables $(m,m')$ 
are restricted to pairs of primes.
Due to a suitable choice of a cut-off parameter, $X$, that appears in Section 
\ref{ss:MV}, we will not need this variant of the proposition (cf. Section 
\ref{ss:small-primes}) and only provide a very brief account of it at the very end 
of this section.
\end{rem}

\begin{proof}
To begin with, we endow $G/\Gamma \times G/\Gamma$ with a 
metric by setting 
$$d((x,y),(x',y'))=d_{G/\Gamma}(x,x')+d_{G/\Gamma}(y,y').$$
Let $\tilde F: G/\Gamma \times G/\Gamma \to \CC$ be defined via
$\tilde F (\gamma,\gamma') = F(\gamma) \overline{F(\gamma')}$.
This is a Lipschitz function.
Indeed, the fact that $F$ and $\overline{F}$ are $1$-bounded Lipschitz 
functions allows us to deduce that 
$\|\tilde F\|_{\mathrm{Lip}} \leq \|F\|_{\mathrm{Lip}}$.
Let $g_{m,m'}: \NN \to G \times G$ be the polynomial sequence
defined by 
$$g_{m,m'}(n)= (g(nm+\tilde D_m),g(nm'+\tilde D_{m'})).$$
Furthermore, we write $\Gamma'=\Gamma \times \Gamma$.
Then $\tilde F$ satisfies
$$
\int_{G/\Gamma \times G/\Gamma}
 \tilde F(\gamma,\gamma')
 \d(\gamma,\gamma')
=
 \int_{G/\Gamma}
 F(\gamma)
 \int_{G/\Gamma}
 \overline{F(\gamma')}
 \d\gamma'
 \d\gamma
= 0.
$$
Now, suppose that 
$$K \in 
\left[\exp \left( (\log \log T)^2 \right), 
\exp \left(H^{-1}(\log T -(\log T)^{1/U})\right)\right]
$$
and that 
$$
\mathcal{E}_{K} \geq  K^2 \delta(T)^{c_2E_2}. 
$$
For each pair $(m,m') \in \mathcal{E}_{K}$, we have
\begin{align} \label{eq:tech-lowerbound}
 \Bigg|\sum_{ n \in I_{m,m'}} 
  \tilde{F} (g_{m,m'}(n) \Gamma) \Bigg|
>  (1 + \|F\|_{\mathrm{Lip}})\delta(T)^{c_2E_2} \# I_{m,m'}.
\end{align}
Thus, for every pair $(m,m') \in \mathcal{E}_K$, the corresponding sequence
$$(\tilde{F} (g_{m,m'}(n) \Gamma))_{n \leq T/ \max(m,m')}$$ 
fails to be totally $\delta(T)^{c_2E_2}$-equidistributed.
Lemma \ref{l:equi/totally-equi} implies that this finite sequence also fails to 
be $\delta(T)^{c_2E_2B}$-equidistributed for some $B\geq1$ that only depends on 
$d$ 
and $m_G$.
All implied constants in the sequel will be allowed to depend on $d$ and $m_G$, 
without explicit mentioning.
By \cite[Theorem 2.9]{GT-polyorbits}\footnote{The 
$\frac{1}{\delta(T)}$-rational 
Mal'cev basis for $G/\Gamma$ induces one for $G/\Gamma \times G/\Gamma$. 
Thus \cite[Theorem 2.9]{GT-polyorbits} is applicable.}, 
there is for each pair $(m,m') \in \mathcal{E}_K$ a non-trivial horizontal 
character 
$$\tilde \eta_{m,m'}
 = \eta_{m,m'} \oplus \eta'_{m,m'}: G\times G \to \RR/\ZZ$$ 
of magnitude $\ll \delta(T)^{-O(c_2E_2)}$ such that
\begin{equation}\label{eq:smoothness-bound}
 \| \tilde \eta_{m,m'} \circ \tilde g_{m,m'}\|_{
 C^{\infty}[T/\max(m,m')]} 
\ll \delta(T)^{-O(c_2E_2)}.
\end{equation}
Given any non-trivial horizontal character 
$\tilde \eta : G \times G \to \RR/\ZZ$, we define the set 
$$
\mathcal{M}_{\eta} 
= \left\{ (m,m') \in \mathcal{E}_{K} \mid 
\tilde \eta_{m,m'} = \tilde \eta
\right\}.
$$
This set is empty unless $|\tilde \eta| \ll \delta(T)^{-O(c_2E_2)}$.
Pigeonholing over all non-trivial $\tilde \eta$ of modulus bounded by 
$\delta(T)^{-O(c_2E_2)}$, 
we find that there is some $\tilde \eta$ amongst them for which 
$$
\# \mathcal{M}_{\tilde \eta} > K^2 \delta(T)^{O(c_2E_2)}. 
$$
Let us fix such a character $\tilde \eta = \eta \oplus \eta'$ and suppose 
without loss of generality that the component $\eta$ is non-trivial.
Suppose
$$
\tilde \eta \circ (g(n),g(n'))
= (\alpha_d n^d + \alpha'_d n'^d) + \dots 
+ (\alpha_1 n   + \alpha'_1 n')
+ (\alpha_0     + \alpha'_0 )$$
and define for $(m,m') \in \mathcal{E}_K$ the coefficients $\alpha_j(m,m')$, 
$1 \leq j \leq d$, via
$$
\tilde \eta \circ g_{m,m'}(n)
= \alpha_d(m,m') n^d + \dots 
+ \alpha_1(m,m') n
+ \alpha_0(m,m').
$$
Then the bound \eqref{eq:smoothness-bound} on the smoothness norm 
asserts that
\begin{align} \label{eq:smoothness-2}
\sup_{1 \leq j \leq d} 
\frac{T^j}{K^j} 
\left\| \alpha_j (m,m') \right\|
\ll \delta(T)^{-O(c_2E_2)}. 
\end{align}
Observe that each $\alpha_j(m,m')$, $1 \leq j \leq d$, satisfies
\begin{align} \label{eq:rec-1}
 \alpha_j(m,m') = 
 \sum_{i=j}^d 
 \binom{i}{j} 
 \Big(\tilde D_m^{i-j} \alpha_i m^{j} 
 + \tilde D_{m'}^{i-j} \alpha'_i m'^j \Big).
\end{align}
We now aim to show with a downwards induction starting from $j=d$ that
\begin{align}\label{eq:alpha_j-kappa_j}
\alpha_j = \frac{a_j}{\kappa_j} + \tilde\alpha_j,
\end{align}
where 
$1 \leq \kappa_j \ll \delta(T)^{-O(c_2E_2)}$, 
$\gcd(a_j,\kappa_j)=1$, 
and 
\begin{align}\label{eq:tilde-alpha-j}
 \tilde \alpha_j \ll \delta(T)^{-O(c_2E_2)} T^{-j}.
\end{align}
Suppose $j_0 \leq d$ and that the above holds for all $j>j_0$.
Set $k_{j_0}=\lcm(\kappa_{j_0+1}, \dots, \kappa_d)$ if $j_0<d$, and $k_{j_0}=1$ 
when $j_0=d$.
Note that $k_{j_0} \ll \delta(T)^{-O(c_2E_2)}$.

Pigeonholing, we find that there is $\tilde m'$ such that 
$m'=\tilde m'$ for $\gg K\delta(T)^{O(c_2E_2)}$ pairs 
$(m,m') \in \mathcal{M}_{\tilde \eta}$.
Amongst these there are furthermore 
$\gg K\delta(T)^{O(c_2E_2)}$ values of $m$ that belong to the 
same fixed residue class modulo $k_{j_0}$. 
Denote this set of integers $m$ by $\mathcal{M}'$.
Suppose $m=k_{j_0}m_1 + m_0 \in \mathcal{M}'$.
Letting $\{x\}$ denote the fractional part of $x \in \RR$, we then have
\begin{align*}
\{\tilde D_m^{i-j_0} \alpha_i m^{j_0}\}
= 
\Big\{
 \tilde D_m^{i-j_0} \tilde \alpha_i m^{j_0} 
 + \frac{a_i m_0^{j_0}}{\kappa_i} 
\Big\}, \qquad (i \geq j_0),
\end{align*}
where, in view of \eqref{eq:tilde-alpha-j},
$$
\tilde D_m^{i-j_0} \tilde \alpha_i m^{j_0} 
\ll \delta(T)^{-O(c_2E_2)}  K^{i} T^{-i}.
$$
Since $m_0$ is fixed, it thus follows from \eqref{eq:smoothness-2}, 
\eqref{eq:rec-1}, \eqref{eq:alpha_j-kappa_j} and the above bound that as 
$m$ varies over $\mathcal{M}'$, the value of
$$
\|\alpha_{j_0} m^{j_0}\|
$$
lies in a fixed interval of length 
$\ll \delta(T)^{-O(c_2E_2)} K^{j_0} T^{-j_0}$.

We aim to make use of this information in combination with the Waring trick
from \cite[\S 3]{GT-nilmobius} that was already employed in Section 
\ref{s:linear-subsecs}.
For this purpose, we consider the set of integers 
\begin{align*}
\mathcal{M}^*=\left\{ m \leq s(2K)^{j_0}:
\begin{array}{l}
m = m_1^{j_0} + \dots + m_s^{j_0} \\
m_1, \dots, m_s \in \mathcal{M}'
\end{array}
\right\}
\end{align*}
with $s \geq 2^{j_0}+1$.
For each element $m \in \mathcal{M}^*$ of this set,
$\|\alpha_{j_0} m\|$ lies in an interval of length
$\ll_s \delta(T)^{-O(c_2E_2)} K^{j_0} T^{-j_0}$.
Furthermore, Lemma \ref{l:GT-waring} implies that 
$\# \mathcal{M}^* \gg \delta(T)^{O(c_2E_2)} K^{j_0}$.
The restrictions on the size of $K$ and the assumptions on the 
function $\delta$ imply that the conditions of Lemma \ref{l:GT-rec} are 
satisfied once $T$ is sufficiently large. 
Thus, assuming $T$ is sufficiently large, there is an integer
$1 \leq \kappa'_{j_0} \ll \delta(T)^{-O(c_2E_2)}$ such that 
\begin{align*} 
\| \kappa'_{j_0} \alpha_{j_0} \|
\ll \delta(T)^{-O(c_2E_2)} T^{-j_0},
\end{align*}
i.e.
\begin{align*}
\alpha_{j_0} = \frac{a_{j_0}}{\kappa_{j_0}} + \tilde\alpha_{j_0},
\end{align*}
where $\kappa_{j_0}| \kappa'_{j_0} $, $\gcd(a_{j_0},\kappa_{j_0})=1$
and $\tilde \alpha_{j_0} \ll \delta(T)^{-O(c_2E_2)} T^{-j_0}$,
as claimed.

In particular, we have
\begin{align}\label{eq:kappa_j}
\| \kappa_{j} \alpha_{j} \|
\ll \delta(T)^{-O(c_2E_2)} T^{-j}
\end{align}
for $1\leq j \leq d$.
Proceeding as in \cite[\S3]{GT-nilmobius}, 
let $\kappa=\lcm(\kappa_1, \dots, \kappa_d)$ and set 
$\tilde \eta = \kappa \eta$.
Then \eqref{eq:kappa_j} implies that
$$
\|\tilde \eta \circ g\|_{C^{\infty}[T]} =
\sup_{1\leq j \leq d} T^j \| \kappa \alpha_{j} \|
\ll \delta(T)^{-O(c_2E_2)},
$$
which in turn shows that
$$
\|\tilde \eta \circ g (n)\|_{\RR/\ZZ} \ll
\delta(T)^{-O(c_2E_2)}n/T
$$
for every $n \in \{1, \dots, T\}$.
Note that the latter bound can be controlled by restricting $n$ to a smaller 
range.
For this, set $T'= \delta(T)^{c_2E_2C}T$ for some constant $C \geq 1$ depending 
only on $d$ and $m_G$, chosen sufficiently large to guarantee that
$$
\|\tilde \eta \circ g (n)\|_{\RR/\ZZ} \leq 1/10,
$$
whenever $n \in \{1, \dots , T'\}$.
Let $\chi: \RR/\ZZ \to [-1,1]$ be a function of bounded Lipschitz norm that 
equals $1$ on $[-\frac{1}{10},\frac{1}{10}]$ and satisfies $\int_{\RR/\ZZ} 
\chi(t)\d t = 0$.
Then, by setting $F := \chi \circ \tilde \eta$, we obtain a function 
$F:G/\Gamma \to [-1,1]$ such that $\int_{G/\Gamma} F = 0$ and 
$\|F\|_{\mathrm{Lip}} \ll \delta(T)^{-O(c_2E_2)}$.
Choosing $c_2$ sufficiently small, we may ensure that
$$\|F\|_{\mathrm{Lip}} < \delta(T)^{-E_2}$$
and, moreover, that
$$
T'> \delta(T)^{E_2}T.
$$
The quantities $T'$, $F$ and $c_2$ are chosen in such a way that
$$
\Big| \frac{1}{T'} \sum_{1 \leq n \leq T'} F(g(n)\Gamma) \Big|
= 1 > \delta(T)^{E_2}  \|F\|_{\mathrm{Lip}},
$$
This contradicts the fact that
$(g(n)\Gamma)_{n\leq T}$ is totally $\delta(T)^{E_2}$-equidistributed and 
completes the proof of the proposition.
\end{proof}

\subsection{Restriction of Proposition \ref{p:equid} to pairs of primes}
We end this section by making the contents of Remark \ref{r:primes} more precise.
The variables $(m,m')$ in Proposition \ref{p:equid} can without much additional 
effort be restricted to range over pairs of primes.
It is clear that in the above proof all applications of the pigeonhole principle 
that involve the parameters $m$ and $m'$ have to be restricted to the set of 
primes.
The only true difference lies in the application of Waring's result:
Lemma \ref{l:GT-waring} needs to be replaced by the following one.
\begin{lemma}
Let $K \geq 1$ be an integer and let 
$S \subset \{1, \dots, K\} \cap \mathcal{P}$ be a subset of the primes 
less than $K$.
Suppose $\#S= \alpha \frac{K}{\log K}$.
Let $s \geq 2^k + 1$.
Let $X \subset \{1, \dots, sK^k\}$ denote the set of integers that are 
representable as $p_1^k + \dots + p_s^k$ with $p_1, \dots, p_s \in S$.
Then 
$$|X|\gg_{k,s} \alpha^{2s} K^k,$$
as $K \to \infty$.
\end{lemma}
\begin{proof}
Let $I_s(N)$ denote the number of solutions to the equation
$$p_1^k + \dots + p_s^k = N$$
in positive prime numbers $p_1, \dots, p_s$.
Hua's asymptotic formula \cite[Theorem 11]{hua} for the 
Waring--Goldbach problem implies that
$$I_s(N) \ll_{k,s} \frac{N^{s/k-1}}{(\log N)^{s}}.$$
Thus, for $1 \leq n \leq sK^k$, we have
$$
I_s(n) \ll_{k,s} \frac{K^{s-k}}{(\log K)^{s}}.
$$
Hence,
$$
\alpha^{2s} \frac{K^{2s}}{(\log K)^{2s}}
=\bigg( \sum_{n=1}^{sK^k} I_s(n)\bigg)^2
\leq |X| \sum_{n} I_s^2(n)
\ll_{k,s} |X| K^k\frac{K^{2s-2k}}{(\log K)^{2s}} 
\ll_{k,s} |X| \frac{K^{2s-k}}{(\log K)^{2s}}. 
$$
Rearranging completes the proof of the lemma.
\end{proof}

\section{Proof of Proposition \ref{p:equid-non-corr}}
\label{s:proof-of-prop}
In this section we prove Proposition \ref{p:equid-non-corr} by invoking the 
possibly trivial Dirichlet decomposition from Lemma \ref{l:dirichlet}.
Let $f \in \mathcal{M}_H$, let $h,h'$ be as in Lemma \ref{l:dirichlet} and let  
$f=f_1 * \dots * f_H$ with $f_i=h$ for $i<H$ and $f_H=h*h'$.
We are given integers $Q$ and $A$ such that $0 \leq A < Q \leq (\log N)^E$ and 
such that $A \in (\ZZ/Q\ZZ)^*$. 
Recall that $g \in \mathrm{poly}(\ZZ,G_{\bullet})$ is a polynomial sequence 
with the property that
$(g(n)\Gamma)_{n \leq T/Q}$ is totally 
$\delta(N)^{E_0}$-equidistributed in $G/\Gamma$.
Let $I \subset \{1,\dots,T/Q\}$ be a discrete interval of length at 
least $T/(Q (\log N)^E)$.
Our aim is to bound above the expression
\begin{equation} \label{eq:proof-general-aim}
\bigg| 
\frac{Q}{T}\sum_{n \in I} 
f(Qn+A)F(g(n)\Gamma)
\bigg|.
\end{equation}
If $H=1$, then we may write this expression as
\begin{equation} \label{eq:H=1-case}
\bigg|
\frac{Q}{T}\sum_{n \in I} 
f(Qn+ D')F(g(Dn+D'')\Gamma)
\bigg|,
\end{equation}
where $D=1$, $D'= A$ and $D'' = 0$.
The aim of the next two sections is to show that in the case where $H>1$ 
and the Dirichlet decomposition is non-trivial, the task of bounding 
\eqref{eq:proof-general-aim} can be reduced to that of bounding an expression 
similar to \eqref{eq:H=1-case}, but with $f$ replaced by one of the $f_i$.

\subsection{Reduction by hyperbola method}
\label{ss:hyperbola}
Taking into account that $f=f_1 * \dots * f_H$, the correlation from 
Proposition \ref{p:equid-non-corr} may be written as
\begin{align}\label{eq:full-sum}
& \frac{Q}{T}
\sum_{n \leq T/Q} \1_I(n)f(Qn+A)F(g(n)\Gamma) 
=
\\
\nonumber
& \frac{Q}{T}
 \sum_{\substack{d_1 \dots d_H \leq T \\ 
                 d_1 \dots d_H \equiv  A \\ \Mod{Q} 
 }}
  f_1(d_1) f_2(d_2) \dots f_H(d_H)
  F\left(g\left(\frac{d_1 \dots d_H -A}{Q} 
   \right)\Gamma \right) 
  \1_{P}\left(d_1 \dots d_H\right),
\end{align}
where $P$ is the finite progression defined via $P=QI+A$.
Our first step is to split this summation via inclusion-exclusion into 
a finite sum of weighted correlations of individual factors $f_i$ with a 
nilsequence.
To describe these weighted correlations,  
let $i \in \{1, \dots, H\}$. 
For every $j \not=i$, let $d_j$ be a fixed positive integer and 
write $D_i:=\prod_{j \not= i} d_j$. 
Let $a_i \in [0, T/D_i)$ be an integer.
Weighted correlations involving $f_i$ will then take the form:
\begin{align}\label{eq:fi-corr}
&  \frac{Q}{T}  
  \bigg(\prod_{j \not= i} f_j(d_j)\bigg)
  \sum_{\substack{a_i < d_i \leq T/D_i 
        \\ d_iD_i \equiv A \Mod{Q}}} 
   \1_{P}(d_iD_i)
   f_i(d_i) 
   F\left(g\left(
   \frac{d_i D_i -A}{Q} \right)\Gamma 
   \right) \\
\nonumber
&= \frac{Q}{T}
  \bigg(\prod_{j \not= i} f_j(d_j)\bigg)
  \sum_{\frac{a_i - D'_i}{Q} < n \leq 
        \frac{T-D'_i}{D_iQ}}
   f_i(Q n + D'_i) 
   F\Big(g\left(D_i n + D''_i \right)\Gamma \Big)
   \1_I(D_i n + D''_i ),
\end{align}
for suitable integers $D'_i,D''_i$, determined by the values of $D_i \Mod{Q}$ 
and $A$.
In order to bound correlations of the form \eqref{eq:fi-corr}, we need to ensure 
that $d_i$ runs over a sufficiently long range, which will be achieved by 
arranging for $D_i \leq T^{1-1/H}$ to hold.

Let $\tau=T^{1-1/H}$ and note that $D_i = D_j \frac{d_j}{d_i}$. 
Hence,
$$
D_i > \tau \iff d_j > \frac{\tau d_i}{D_j}.
$$
With the help of this equivalence, the function
$\1: \ZZ^H \to 1$ can be decomposed as follows.
Suppose $d_1 \dots d_H \leq T$.
Then
\begin{align*}
\1(d_1,&\dots, d_H) \\
&= \1_{D_1 \leq \tau} + \1_{D_1 > \tau}
 \Big(\1_{D_2 \leq \tau} 
 +\1_{D_2 > \tau}
 \Big(\1_{D_3 \leq \tau}
  + \dots 
 \Big(\1_{D_H \leq \tau} + \1_{D_H > \tau}\Big) \dots \Big) \Big) \\
&=\1_{D_1 \leq \tau} + \1_{D_1 > \tau}
 \Big(\1_{D_2 \leq \tau} \1_{d_2 > \frac{\tau d_1}{D_2}}
 +\1_{D_2 > \tau}
 \Big(\1_{D_3 \leq \tau}
  \1_{d_3 > \frac{\tau \max(d_1,d_2)}{D_3}}
  + \dots\\
&\qquad \qquad \qquad \qquad \qquad \qquad \qquad \qquad \dots
  + \1_{D_{H-1} > \tau}
 \Big(\1_{D_H \leq \tau} + \1_{D_H > \tau}\Big) \dots \Big) \Big)\\
&= \1_{D_1 \leq \tau} 
  +\1_{D_2 \leq \tau} \1_{d_2 > \frac{\tau d_1}{D_2}}
  +\1_{D_3 \leq \tau} \1_{d_3 > \frac{\tau \max(d_1,d_2)}{D_3}}
  + \dots 
  +\1_{D_H \leq \tau} 
   \1_{d_H > \frac{\tau \max(d_1, \dots, d_{H-1})}{D_H}}.
\end{align*}
Thus,
$$
 \sum_{d_1 \dots d_H < T}
=\sum_{i=1}^H
 \sum_{D \leq T^{1-1/H}}
 \sum_{\substack{d_1, \dots \widehat{d_{i}} \dots, d_H \\ D_i=D}}
 \sum_{\substack{d_i \leq T/D_i \\
       d_i > \tau \max(d_1, \dots, d_{i-1})/D_i}}.
$$
This shows that the original summation \eqref{eq:full-sum} may be 
decomposed as a sum of summations of the shape \eqref{eq:fi-corr} while only 
increasing the total number of terms by a factor of order $O(H)$.
Expressing, if necessary, the summation range
$$(\tau \max(d_1, \dots, d_{i-1})/D_i, T/D_i)$$ of $d_i$ as the 
difference of two intervals starting from $1$, we can ensure that $d_i$ 
runs over an interval of length $\gg T/D_i \gg T^{1/H}$.
The correlation now decomposes as:
\begin{align}\label{eq:pre-decomposition}
\nonumber
\frac{Q}{T}
&
 \sum_{\substack{d_1 \dots d_H \leq T \\ 
                 d_1 \dots d_H \equiv  A \\
                 \Mod{Q} }}
  f_1(d_1) f_2(d_2) \dots f_H(d_H)
  F\left(g\left(\frac{d_1 \dots d_H -A}{Q} 
\right)\Gamma \right) 
  \1_{P}\left(d_1 \dots d_H\right) \\
&\quad \leq~
     \sum_{i=1}^H
     \sum_{k=0}^{\frac{1-1/H}{\log2}\log T}
     \sum_{\substack{D \sim 2^k \\ (D,Q)=1}}
     \sum_{\substack{d_1, \dots \widehat{d_{i}} \dots, d_H \\ D_i=D}}     
     \bigg( \prod_{j \not= i} \frac{|f_j(d_j)|}{d_i} \bigg) 
     \\
\nonumber     
&\qquad \qquad \qquad \qquad
     \Bigg|\frac{DQ}{T}
     \sum_{\substack{n \leq T/D \\
       n > \frac{\tau}{D} \max(d_1, \dots, d_{i-1}) \\ Dn + D'' \in I}}
     \hspace*{-.5cm}  
     f_i(Qn+D')F(g(Dn+D''))
     \Bigg|.
\end{align}
Observe that \eqref{eq:H=1-case} can be regarded as the special case $H=1$ and 
$D=1$ of this bound.
Our next aim is to analyse the innermost sum of \eqref{eq:pre-decomposition}
as $D\sim 2^k$ varies.
Setting $E_1 = E_0$, we deduce from Proposition \ref{p:linear-subsecs} that 
whenever $2^k \in [\exp ((\log \log T)^2),(\log T)^{1-1/H}]$ then there is a set 
$\mathcal{B}_{2^k}$ of cardinality at most $O(\delta(N)^{c_1 E_0}2^k)$ such that 
for each $D \sim 2^k$ with $D \not\in \mathcal{B}_{2^k}$ the sequence
$$(g_D(n)\Gamma)_{n\leq T/Q}, \qquad 
g_D(n):=g(Dn+D''),$$
is totally $\delta(N)^{c_1 E_0}$-equidistributed.
Before turning to the case of $D \not\in \mathcal{B}_{2^k}$, we 
bound the total contribution from exceptional $D$, that is, 
from $D \in \mathcal{B}_{2^k}$ and from $D \leq \exp((\log \log T)^2)$. 

\subsection{Contribution from exceptional $D$}
\label{ss:D-contribution}

Let $\mathcal{B}_{2^k}$ denote the exceptional set from the previous section.
\begin{lemma}\label{l:exceptional-D}
Whenever $E_0$ is sufficiently large in terms of $d$, $m_G$ and $H$, 
the following two estimates hold:
\begin{align*}
\sum_{\frac{(\log \log T)^2}{\log 2} < k \leq \frac{(1-1/H)\log T}{\log 2} }
&\sum_{D \in \mathcal{B}_{2^k}}
\sum_{\substack{d_1 \dots d_H \leq T \\ 
                 d_1 \dots d_H \equiv  A  \Mod{Q} \\ 
                 D_i=D }}
 |f_1(d_1) f_2(d_2) \dots f_H(d_H)|
 \1_{P}\left(d_1 \dots d_H\right) \\
&\ll_t \frac{T}{Q} \frac{1}{(\log T)^2},
\end{align*}
and
\begin{align*}
& \sum_{\substack{D \leq \exp((\log \log T)^2) \\ \gcd(D,W)=1}}
\sum_{\substack{d_1 \dots d_H \leq T \\ 
                 d_1 \dots d_H \equiv  A  \Mod{Q} \\ 
                 D_i=D }}
 |f_1(d_1) f_2(d_2) \dots f_H(d_H)|
 \1_{P}\left(d_1 \dots d_H\right) \\
&\ll (\log \log T)^{2H}  \frac{T}{Q} 
  \frac{Q}{\phi(Q)}
 \frac{1}{\log T} \prod_{\substack{p\leq T \\ p \nmid Q}}
 \left(1 + \frac{|f(p)|}{Hp} \right) .
\end{align*}
\end{lemma}
Before we prove this lemma, let us consider its contribution to the bound in 
Proposition~\ref{p:equid-non-corr}.
The contribution from the first part is easily seen to be negligible.
Regarding the second part, recall that $H>1$ and note that 
by property (2) of Definition \ref{def:M}, we have
$$
\prod_{\substack{Q< p\leq T}}
\left(1 + \frac{(H-1)|f(p)|}{Hp} \right)
\gg \left(\frac{\log T}{E\log \log T}\right)^{(H-1)\alpha_f/H}, 
$$
where the exponent is positive.
Thus, the bound in the second part saves a power of $\log x$ when compared 
with the bound in \eqref{eq:prop-bound} and is therefore also negligible.

\begin{proof}
Set 
$$\bar f_i (n) := |f_1 * \dots \widehat f_i \dots * f_H(n)|.$$
Then $\bar f_i = |h^{*(H-1)}*h'|$ or $|h^{*(H-1)}|$ and it follows 
from \eqref{eq:h'-bound} and the properties of $h$ that 
$\bar f_i (n) \leq (CH)^{\Omega(n)}$ for some constant $C$.
This implies a second moment bound of the form
$$
\sum_{\substack{n \leq x \\ \gcd(n,Q)=1}} \bar f_i (n)^2 
\leq x \sum_{\substack{n \leq x \\ \gcd(n,Q)=1}} \frac{\bar f_i (n)^2}{n}
\leq x \prod_{w(N) < p \leq x} \left(1 - \frac{(CH)^2}{p}\right)^{-1}
\leq x (\log x)^{O(H^2)}.
$$
Similarly, we have 
$$
\sum_{\substack{n \leq x \\ \gcd(n,Q)=1}} 
|f_i(n)| \ll x (\log x)^{O(H)}.
$$
Since Proposition \ref{p:linear-subsecs} provides the bound 
$\# \mathcal{B}^*_{2^k} 
\ll \delta(N)^{c_1E_0} 2^k,$
a trivial application of the Cauchy--Schwarz inequality yields
\begin{align*}
 \sum_{D \in \mathcal{B}^*_{2^k}}
  \bar f_i(D)
  \sum_{\substack{n \leq T/D 
        \\ nD \equiv A  \Mod{Q} 
        \\ nD \in P}} |f_i(n)| 
&\leq
  \sum_{\substack{n \leq T/2^k\\ \gcd(n,Q)=1}} |f_i(n)|
  \sum_{\substack{D \in \mathcal{B}^*_{2^k} \\ \gcd(n,Q)=1}}   
  \bar f_i (D) \\
&\leq   \sum_{\substack{n \leq T/2^k \\ \gcd(n,Q)=1}} |f_i(n)| 2^k 
\delta(N)^{c_1E_0} k^{O(H^2)} \\
&\leq   T (\log T)^{O(H)} \delta(N)^{c_1E_0} k^{O(H^2)}.
\end{align*}
Recall that $c_1$ only depends on $d$ and $m_G$, and that by assumption of 
Proposition \ref{p:equid-non-corr} we have $\delta(N) \ll (\log T)^{-1}$.
Since the summation in $k$ has length at most $\log T$ and since 
$k^{O(H^2)} < (\log T)^{O_H(1)}$ for each $k$, the first part of the lemma 
follows by choosing $E_0$ sufficiently large in terms of $d$, $m_G$ and $H$.

Concerning the second part, we have
\begin{align*}
\sum_{\substack{D \leq \exp((\log \log T)^2) \\ \gcd(D,Q)=1}}
 \bar f_i(D) 
 \sum_{\substack{n \leq T/D
       \\ nD \equiv A  \Mod{Q}
       \\ nD \in P}} |f_i(n)| 
\leq 
 \sum_{\substack{D \leq \exp((\log \log T)^2) \\ \gcd(D,Q)=1}}
 \bar f_i(D) 
 \sum_{\substack{n \leq T/D
       \\ n \equiv A\bar D \\ \Mod{Q}}} |f_i(n)|,
\end{align*}
where $\bar D D \equiv 1 \Mod{Q}$.
Since $\log (T/D) \asymp \log T$ and $T/D < T$, Shiu's bound \eqref{eq:shiu} 
yields the upper bound
\begin{equation} \label{eq:small-D-bound}
 \ll 
 \sum_{\substack{D \leq \exp((\log \log T)^2) \\ \gcd(D,Q)=1}}
 \frac{\bar f_i(D)}{D}
 \frac{T}{Q}
 \frac{Q}{\phi(Q)}
 \frac{1}{\log T} \prod_{\substack{p\leq T \\ p \nmid Q}}
 \left(1 + \frac{|f_i(p)|}{p} \right).
\end{equation}
The outer sum satisfies
\begin{align*}
 \sum_{\substack{D \leq \exp((\log \log T)^2) \\ \gcd(D,Q)=1}}
 \frac{\bar f_i(D)}{D}
&\ll
\prod_{w(N)< p \leq \exp((\log \log T)^2)}
\bigg(
1 + \frac{|f(p)|}{Hp} \bigg)^{H-1} \\
&\ll \exp\left((H-1) \sum_{\substack{p \leq  \exp((\log \log T)^2)}} 
 \frac{1}{p}\right) 
\ll (\log \log T)^{2H},
\end{align*}
which completes the proof of the second part.
\end{proof}

\subsection{Montgomery--Vaughan approach}\label{ss:MV}
Since $\mathcal{M}_1 \subset \mathcal{M}_H$ for all $H>1$, it suffices to prove 
Proposition \ref{p:equid-non-corr} for $H>1$.
Since the task of bounding \eqref{eq:H=1-case} for $D=1$ presents an 
easier special case of the task of bounding the inner sum of 
\eqref{eq:pre-decomposition} for unexceptional $D$ when $H>1$, a proof for 
the $H=1$ case may, however, be extracted from the argument below. 
In fact, most of the argument directly applies when setting $H=D=1$.
The main differences leading to simplifications are that 
\begin{enumerate}
 \item if $H=D=1$, one can, instead of later referring to the results from 
Section \ref{s:linear-subsecs}, directly work with the equidistribution 
properties of the given polynomial sequence $g$, and
\item the extra work of handling the outer sums in \eqref{eq:pre-decomposition} 
is not required when $H=D=1$.
\end{enumerate}

From now on we assume that $H>1$ and that $D$ is unexceptional, that is 
$D \sim 2^k$ for $k \geq (\log \log T)^2 / \log 2$ 
and $D \not\in \mathcal{B}_{2^k}$, where $\mathcal{B}_{2^k}$ is the exceptional 
set from Section \ref{ss:hyperbola}.
To bound the inner sum of \eqref{eq:pre-decomposition} for unexceptional $D$, we
employ the strategy of Montgomery and Vaughan \cite{mv77} outlined in Section 
\ref{s:outline}, and begin by introducing a factor $\log n$ into the average.
This will later allow us to reduce matters to understanding 
equidistribution along sequences defined in terms of primes.
We set $h=f_i$.
We caution that this is \emph{not} the function $h$ from Lemma \ref{l:dirichlet}, 
but could either be $h$ or $h*h'$ in the notation of the lemma. 

Cauchy--Schwarz and several integral comparisons show that
\begin{align*}
&\sum_{n \leq T/(DQ)}
 \1_I(Dn + D'')
 h(Qn+D') F(g(Dn+D'')\Gamma) 
 \log\bigg(\frac{T/D}{Qn+D'} \bigg) \\
&\leq 
 \bigg(\sum_{n \leq T/(DQ)} 
 \Big(\log (T/(DQ)) - \log n \Big)^2 
\bigg)^{1/2}
 \bigg(
 \sum_{n \leq T/(DQ)} 
 h^2(Qn+D')
 \bigg)^{1/2} \\
&\ll \frac{T}{DQ} 
 \sqrt{ \frac{DQ}{T} \sum_{n\leq T/(DQ)} 
 h^2(Qn+D')},
\end{align*}
and hence, invoking $D \leq T^{1-1/H}$,
\begin{align} \label{eq:initial-CS}
&\frac{DQ}{T}
 \sum_{\substack{n \leq T/(DQ) \\ Dn+D'' \in I}}  
 h(Qn+D') F(g(Dn+D'')\Gamma) \\
\nonumber 
&\ll_H~ \frac{1}{\log T} 
 \sqrt{\frac{DQ}{T} \sum_{n\leq \frac{T}{DQ}} 
 h^2(Qn+D')}
\\
\nonumber
&\quad + \frac{1}{\log T}
 \Bigg| \frac{DQ}{T}
  \sum_{\substack{n\leq T/(DQ) \\Dn+D'' \in I}}
  h(Qn+D') F(g(Dn+D'') \Gamma) 
  \log (Qn+D') 
  \Bigg|.
\end{align}
Lemma \ref{l:dirichlet} shows that the contribution of the first 
term in this bound to \eqref{eq:pre-decomposition} is at most
$$ 
O_H\Bigg( \frac{1}{(\log T)^{1/2}}
\frac{Q}{\phi(Q)} \frac{1}{\log x} 
\prod_{\substack{p\leq x \\ p \nmid Q}}
\left(1 + \frac{|f(p)|}{p} \right)\Bigg),
$$ 
which is negligible in view of the bound stated in Proposition 
\ref{p:equid-non-corr}.
It remains to estimate the second term from \eqref{eq:initial-CS}.
For this, it will be convenient to abbreviate
$$g_D(n):=g(Dn+D''),$$
and to introduce the two finite progressions
\begin{equation}\label{eq:I_D}
I_D=\{n: Dn+D'' \in I\}
\quad \text{and} \quad
P_D=\Big\{n: \frac{n - D'}{Q}\in I_D \Big\}.
\end{equation}
Since $\log n = \sum_{m|n} \Lambda(m)$, our task is to bound
\begin{align}\label{eq:task}
\frac{DQ}{T\log T}
\Bigg|
\sum_{\substack{mn\leq T/D \\ mn \equiv D' \Mod{Q}}} 
\1_{P_D}(nm) h(nm) \Lambda(m) 
F\Big(g_D\Big(\frac{nm-D'}{Q}\Big) \Gamma\Big) 
\Bigg|.
\end{align}
To further simplify this expression we now show that one can, at the expense of 
a small error term, restrict the summation in \eqref{eq:task} to pairs $(m,n)$ of 
co-prime integers for which $m=p$ is prime.
To see this, recall that $F$ is $1$-bounded and observe that
\begin{align*}
 \sum_{\substack{ nm \leq T/D  
       \\ \Omega(m)\geq 2 \text{ or } \gcd(n,m)>1 
       \\ mn \equiv D' \Mod{Q} }}
 |h(nm)| \Lambda (m)
& \leq
 2 \sum_{p}
 \sum_{k \geq 2} 
 k \log p 
 \sum_{\substack{n\leq T/D, p^k\|n \\ n \equiv D' \Mod{Q} }} 
 |h(n)| \\
& \leq 2 \sum_{p > w(N)}
 \sum_{k \geq 2}
 H^k k \log p
 \sum_{\substack{n\leq T/(Dp^k) \\ p^kn \equiv D' \Mod{Q} }} 
 |h(n)| .
\end{align*}
If $p^k \leq (T/D)^{1/2}$, then Shiu's bound \eqref{eq:shiu} implies for the 
inner sum:
$$
 \sum_{\substack{n\leq T/(Dp^k) \\ p^kn \equiv D' \Mod{Q} }} 
 |h(n)|
 \ll \frac{1}{p^k} 
 \frac{T}{D} \frac{1}{\phi(Q)} 
 \frac{1}{\log T} 
 \prod_{\substack{p\leq T/D \\ p \nmid Q}} 
 \left(1 + \frac{|h(p)|}{p}\right).
$$
If $N$ is sufficiently large, then $H \log p \ll p^{1/4}$ for all $p > w(N)$ and 
thus
\begin{align*}
 \sum_{p > w(N)}
 \sum_{\substack{k \geq 2 \\ p^k \leq (T/D)^{1/2}}}
 \frac{H^k \log p^k}{p^k}
 \ll \sum_{p > w(N)}
 \frac{1}{p^{2-1/2}}
 \ll 
 \frac{1}{w(N)^{1/2}}.
\end{align*}
Combining the last three steps, the contribution to \eqref{eq:task} 
from the terms $p^k \leq (T/D)^{1/2}$ is seen to be bounded by
$$ 
 \ll \frac{1}{w(N)^{1/2}\log T}
 \frac{Q}{\phi(Q)} \frac{1}{\log T} 
 \prod_{\substack{p\leq T/D \\ p \nmid Q}} 
 \left(1 + \frac{|h(p)|}{p}\right).
$$

Turning towards the case of $p^k > (T/D)^{1/2}$, note first that, 
provided $N$ is sufficiently large so that $w(N)>H$, then:
\begin{align*}
 \sum_{\substack{n\leq T/(Dp^k) \\ p^kn \equiv D' \Mod{Q} }} 
 |h(n)| 
&\leq  \frac{T}{Dp^k}
 \sum_{\substack{n\leq T/(Dp^k) \\ \gcd(n,Q)=1}} 
 \frac{|h(n)|}{n}\\
&\leq \frac{T}{Dp^k}
 \prod_{w(N) < p' \leq T/(Dp^k)}
 \left(1 - \frac{H}{p'} \right)^{-1}
\leq 
 \frac{T}{D p^k} \left(\log_+ \frac{T}{D p^k}\right)^{O(H)},
\end{align*} 
where $\log_+(x)= \max\{\log x, 0\}$ for $x>0$, as usual.
Assuming, again, that $H \log p \ll p^{1/4}$ for all $p > w(N)$, 
the remaining sum over $p^k > (T/D)^{1/2}$ therefore satisfies:
\begin{align*}
& \frac{DQ}{T\log T} \sum_{p > w(N)}
 \sum_{\substack{k \geq 2 \\ p^k > (T/D)^{1/2}}}
 H^k \log p^k
 \sum_{\substack{n\leq T/(Dp^k) \\ p^kn \equiv D' \Mod{Q} }} 
 |h(n)| \\
& \leq \frac{DQ}{T\log T} \sum_{p > w(N)}
  \sum_{\substack{k \geq 2 \\ p^k > (T/D)^{1/2}}}
 H^k \log p^k
 \frac{T}{D p^k} \left(\log_+ \frac{T}{D p^k}\right)^{O(H)}\\
&\ll \frac{Q}{\log T} \left(\log 
\frac{T}{D}\right)^{O(H)}
 \sum_{p > w(N)} 
 \sum_{\substack{k \geq 2 \\ p^k > (T/D)^{1/2}}}
 \frac{H^k \log p^k}{p^k} \\
&\ll \frac{Q}{\log T} 
 \left(\log \frac{T}{D}\right)^{O(H)}
 \sum_{p > w(N)} 
 \sum_{\substack{k \geq 2 \\ p^k > (T/D)^{1/2}}}
 p^{-k(1-1/4)} \\ 
&\ll \frac{Q}{\log T}
 \left(\log \frac{T}{D}\right)^{O(H)}
 \sum_{p > w(N)} 
 p^{-2 + 1/2}p^{1/4} \left(\frac{T}{D}\right)^{-1/4}\\
&\ll \frac{Q}{\log T}
\left(\frac{T}{D}\right)^{-1/4} \left(\log \frac{T}{D}\right)^{O(H)}\\
&\ll T^{-\frac{1}{8H}},
\end{align*}
This contribution is dominated by that of the smaller prime powers above.

Thus, the total contribution to \eqref{eq:pre-decomposition} of pairs $(m,n)$ 
that are not of the form $(m,p)$, where $p$ is prime that does not divide $m$, is 
bounded by
\begin{align*}
& \frac{1}{\log T} 
  \sum_{i=1}^t \sum_k \sum_{D \sim 2^k} \sum_{d_1 \dots \hat{d_i} \dots d_t = D}
  \bigg(\prod_{j \not= i} \frac{|f_j(d_j)|}{d_i}\bigg)
  \frac{Q}{\phi(Q)} \frac{1}{\log T} 
 \prod_{\substack{p\leq T/D \\ p \nmid Q}} 
 \left(1 + \frac{|h(p)|}{p}\right)\\
&\leq \frac{1}{w(N)^{1/2} \log T}
 \frac{Q}{\phi(Q)} \frac{1}{\log T} 
 \prod_{\substack{p\leq T \\ p \nmid Q}} 
 \left(1 + \frac{|f(p)|}{p}\right),
\end{align*}
which is negligible in view of the bound claimed in Proposition 
\ref{p:equid-non-corr}.

This reduces the task of proving Proposition \ref{p:equid-non-corr} to that of 
bounding the expression
\begin{equation}\label{eq:newtask}
 \frac{DQ}{T}
 \Bigg|
 \sum_{\substack{mp\leq T/D \\ mp \equiv D' \Mod{Q}}} 
 \1_{P_D}(mp) h(m) h(p) \Lambda(p) 
 F\Big(g\Big(\frac{pm-D'}{Q}\Big) \Gamma\Big)
 \Bigg|.
\end{equation}

\subsection{Decomposing the summation range}\label{ss:decomposition}
We prepare the analysis of \eqref{eq:newtask} by first splitting the summation 
into large and small divisors with respect to the parameter 
$$X=X(D)
=\left(\frac{T}{D}\right)^{1-1/\left(\log\frac{T}{D}\right)^{\frac{U-1}{U}}},
$$
for a fixed integer $U \geq 4$.
With this choice of $X$ we obtain
\begin{align}\label{eq:two-sums}
& \frac{Q D}{T}\sum_{\substack{m<X \\ \gcd(m,Q)=1}} \nonumber
 \sum_{\substack{p \leq T/(m D)\\ p \equiv D' \bar{m} \Mod{Q}}} 
 \1_{P_D}(mp) h(m) h(p) \Lambda(p) 
 F\Big(g\Big(\frac{pm-D'}{Q}\Big) \Gamma\Big) \\
&+ \frac{QD}{T}
 \sum_{\substack{m>X \\ \gcd(m,Q)=1}} 
 \sum_{\substack{p \leq T/(mD) \\ p \equiv D' \bar{m} \Mod{Q}}} 
 \1_{P_D}(mp) h(m) h(p) \Lambda(p) 
 F\Big(g\Big(\frac{pm-D'}{Q}\Big) \Gamma\Big).
\end{align}
In order to analyse these expressions, we dyadically decompose in each of the two 
terms the sum with shorter summation range.
The cut-off parameter $X$ is chosen in such a way that one of the dyadic 
decompositions is of short length, depending on $U$.
Indeed, we have $$\log_2 X \sim \log_2 (T/D),$$ while
$$
\log_2 \frac{T}{DX} = \frac{(\log \frac{T}{D})^{1/U}}{\log 2}.
$$
Let $$T_0 = \exp((\log\log T)^2).$$
Then the two sums from \eqref{eq:two-sums} decompose as
\begin{align}\label{eq:two-sums-large-p}
&\frac{QD}{T} \sum_{\substack{m<T_0 \\ \gcd(m,Q)=1}} 
 \sum_{\substack{p \leq T/(mD) \\ p \equiv D' \bar{m} \Mod{Q}}} 
 \1_{P_D}(mp) h(m) h(p) \Lambda(p) 
 F\Big(g_D\Big(\frac{pm-D'}{Q}\Big) \Gamma\Big) \\
\nonumber
&+ \frac{QD}{T}\sum_{j=1}^{\log_2 (X/T_0)} 
 \sum_{\substack{m \sim 2^{-j} X \\ \gcd(m,Q)=1}}
 \sum_{\substack{p \leq T/(mD) \\ p \equiv D' \bar{m} \Mod{Q}}} 
 \1_{P_D}(mp) h(m) h(p) \Lambda(p) 
 F\Big(g_D\Big(\frac{pm-D'}{Q}\Big) \Gamma\Big) 
\end{align}
and
\begin{align}\label{eq:two-sums-small-p}
&
 \frac{QD}{T}
\Bigg\{
 \sum_{\substack{m>X \\ \gcd(m,Q)=1}}
 \sum_{\substack{p \leq \min(T/(mD),T_0) 
      \\ p \equiv D' \bar{m} \Mod{Q}}}
 \1_{P_D}(mp) h(m) h(p) \Lambda(p) 
 F\Big(g_D\Big(\frac{pm-D'}{Q}\Big) \Gamma\Big) \\
\nonumber
&+
 \sum_{j=1}^{\log_2 (T/(XDT_0))} \hspace*{-.5cm}
 \sum_{\substack{m>X \\ \gcd(m,Q)=1}} 
 \sum_{\substack{p \sim 2^{-j}T/(XD) \\ p \equiv D' \bar{m} \Mod{Q}}}
 \1_{pm<T/D}
 \1_{P_D}(mp) h(m) h(p) \Lambda(p) 
 F\Big(g_D\Big(\frac{pm-D'}{Q}\Big) \Gamma\Big) 
\Bigg\}.
\end{align}
We now analyse the contribution from these four sums to 
\eqref{eq:pre-decomposition} in turn, 
beginning with the two short sums up to $T_0$, which are both straightforward to
bound.
The main work goes into handling the large primes case corresponding to the long 
sum in \eqref{eq:two-sums-large-p}.
Here we will make use of the results from Sections \ref{s:linear-subsecs} and 
\ref{s:products}.
The long sum from \eqref{eq:two-sums-small-p} will, again, be straightforward to 
handle due to the above choice of the parameter $X$.

\subsection{Short sums}
The following lemma provides straightforward bounds on the contribution of 
the short sums in \eqref{eq:two-sums-large-p} and \eqref{eq:two-sums-small-p} to 
\eqref{eq:pre-decomposition}.

\begin{lemma} \label{l:short-sums}
Writing $\bar{f}_i(n) = |f_1 * \dots * \widehat{f_i} * \dots * f_H(n)|$, we have
\begin{align}\label{eq:short-sum-bd1}
\nonumber
&\sum_{\substack{D \leq T^{1-1/H}\\(D,Q)=1}}
\frac{\bar{f}_i(D)}{\log T}
 \bigg| 
 \frac{Q}{T}
 \sum_{\substack{m<T_0 \\ \gcd(m,Q)=1}} 
 \sum_{\substack{p \leq T/(mD) \\ p \equiv D' \bar{m} \\ \Mod{Q}}} 
 \1_{P_D}(mp) h(m) h(p) \Lambda(p) 
 F\Big(g_D\Big(\frac{pm-D'}{Q}\Big) \Gamma\Big)\bigg| \\
&\ll (\log \log T)^{2} 
 \frac{1}{\log T}
 \frac{Q}{\phi(Q)}
 \exp 
  \bigg( \frac{H-1}{H}
  \sum_{\substack{p \leq T \\ p\nmid Q }} \frac{|f(p)|}{p} 
  \bigg)
\end{align}
and 
\begin{align}\label{eq:short-sum-bd2}
\nonumber
&\sum_{\substack{D \leq T^{1-1/H}\\(D,Q)=1}}
\frac{\bar{f}_i(D)}{\log T}
 \bigg|\frac{Q}{T} 
 \sum_{\substack{m>X \\ \gcd(m,Q)=1}} 
 \sum_{\substack{p \leq \min(\frac{T}{Dm},T_0) 
       \\ p \equiv D \bar{m} \Mod{Q}}}
 \1_{P_D}(mp) h(m) h(p) \Lambda(p) 
 F\Big(g_D\Big(\frac{pm-D'}{Q}\Big) \Gamma\Big)\bigg| \\
&\ll \frac{(\log \log T)^2}{\log T}
 \frac{1}{\log T} \frac{Q}{\phi(Q)} 
 \prod_{\substack{p\leq T\\ p\nmid Q}}
 \left(1+ \frac{|f(p)|}{p}\right).
\end{align}
\end{lemma}
\begin{rem*}
 Note that both these bounds are negligible when compared to 
\eqref{eq:prop-bound}. In the first case this follows from property (2) of 
Definition \ref{def:M}.
\end{rem*}

\begin{proof}
The short sum in \eqref{eq:two-sums-large-p} satisfies
\begin{align*}
&\bigg|
 \frac{QD}{T} 
 \sum_{\substack{m<T_0 \\ \gcd(m,Q)=1}} 
 \sum_{\substack{p \leq T/(mD) \\ p \equiv D \bar{m} \Mod{Q}}} 
 \1_{P_D}(mp) h(m) h(p) \Lambda(p) 
 F\Big(g_D\Big(\frac{pm-D'}{Q}\Big) \Gamma\Big)\bigg| \\
&\ll 
 \sum_{\substack{m<T_0 \\ \gcd(m,Q)=1}}
 |h(m)|
 \frac{QD}{T} 
 \sum_{\substack{p \leq T/(mD) \\ p \equiv D \bar{m} \Mod{Q}}} 
 \Lambda(p) \\ 
&\ll 
     \frac{Q}{\phi(Q)}
     \sum_{\substack{m < T_0 \\ \gcd(m,Q)=1}} \frac{|h(m)|}{m}
\ll 
     \frac{Q}{\phi(Q)}
     \frac{1}{\log T}
     \exp \bigg( \sum_{w(N)< p < T_0 } \frac{1}{p} \bigg)\\
&\ll (\log \log T)^{2}
     \frac{Q}{\phi(Q)}.    
\end{align*}
Thus, the left hand side of \eqref{eq:short-sum-bd1} is bounded by
\begin{align*}
 (\log \log T)^{2}
 \frac{Q}{\phi(Q)}
 \frac{1}{\log T}
 \sum_{\substack{D \leq T^{1-1/H}\\(D,Q)=1}} 
 \frac{\bar{f}_i(D)}{D}. 
\end{align*}
The claimed bound now follows since
\begin{align*}
  \sum_{\substack{D \leq T^{1-1/H}\\(D,Q)=1}} 
 \frac{\bar{f}_i(D)}{D} 
\ll 
 \exp \bigg( \sum_{\substack{ p \leq T \\ p\nmid Q}} 
    \frac{|f_1(p) + \dots + f_H(p) - f_i(p)|}{p} \bigg)
=  \exp \bigg( \frac{H-1}{H} 
    \sum_{\substack{ p \leq T \\ p\nmid Q}} 
    \frac{|f(p)|}{p} \bigg),
\end{align*}
recalling the definition of the functions $f_j$ from \eqref{eq:h(p)/H}.

The short sum in \eqref{eq:two-sums-small-p} is bounded by
\begin{align*}
&
 \bigg|
 \frac{QD}{T}\sum_{\substack{m>X \\ \gcd(m,Q)=1}} 
 \sum_{\substack{p \leq \min(T/(mD),T_0) 
       \\ p \equiv D \bar{m} \Mod{Q}}}
 \1_{P_D}(mp) h(m) h(p) \Lambda(p) 
 F\Big(g_D\Big(\frac{pm-D'}{Q}\Big) \Gamma\Big)\bigg| \\
&\ll 
 \sum_{\substack{w(N) < p < T_0 }} 
 \frac{\Lambda(p)}{p} \max_{A' \in (\ZZ/Q\ZZ)^*} 
 S_{|h|}\left(\frac{T}{pD};Q,A'\right) 
\\
&\ll 
 \sum_{\substack{w(N) < p < T_0 }}
 \frac{\Lambda(p)}{p} 
 \frac{1}{\log T} \frac{Q}{\phi(Q)} \prod_{\substack{p\leq T\\ p \nmid Q}}
 \left( 1 + \frac{|h(p)|}{p} \right)
 \\
&\ll \frac{\log T_0}{\log T} 
 \frac{Q}{\phi(Q)} 
 \prod_{\substack{p\leq T\\ p \nmid Q}}
 \left( 1 + \frac{|h(p)|}{p} \right),
\end{align*}
where we used \eqref{eq:shiu}.
This shows that left hand side of \eqref{eq:short-sum-bd2} is bounded by
$$
\frac{\log T_0}{(\log T)^2} 
 \frac{Q}{\phi(Q)} 
 \sum_{\substack{D \leq T^{1-1/H}\\(D,Q)=1}}
 \frac{\bar{f}_i(D)}{D}
 \prod_{\substack{p\leq T\\ p \nmid Q}}
 \left( 1 + \frac{|h(p)|}{p} \right) . 
$$
Recall from Section \ref{ss:decomposition} that $\log T_0 = (\log \log T)^2$.
To finish the proof of \eqref{eq:short-sum-bd2}, recall also that
$\bar{f}_i(p)=(H-1) f(p)/H$ and $h(p) = f(p)/H$, that $|f(p)| \leq H$, and that
$|\bar{f}_i(p^k)| \leq (CH)^k$ for some positive constant $C$.
Assuming that $N$ is sufficiently large to ensure that $w(N)> 2CH$, we then have
\begin{align*}
&\sum_{\substack{D \leq T^{1-1/H}\\(D,Q)=1}}
 \frac{\bar{f}_i(D)}{D} 
 \prod_{\substack{p\leq T\\ p \nmid Q}}
 \left( 1 + \frac{|h(p)|}{p} \right)\\
&\leq
 \prod_{\substack{w(N)<p\leq T \\ p\nmid q}}
 \left( 1 + \frac{|f(p)|}{Hp} \right)
 \left( 1 + \frac{(H-1)|f(p)|}{Hp} 
 + \frac{(CH)^2}{p^2}\left(1-\frac{CH}{p}\right)^{-1}\right)\\
&\leq
 \exp\left( \sum_{w(N)<p\leq T} 
 \frac{2(CH)^2}{p^2}  + \frac{H-1}{p^2} \right)
 \prod_{\substack{w(N) < p \leq T \\ p\nmid q}}
 \left( 1 + \frac{|f(p)|}{p} \right)
\ll
 \prod_{\substack{ p \leq T \\ p\nmid Q}}
 \left( 1 + \frac{|f(p)|}{p} \right) ,
\end{align*}
which completes the proof.
\end{proof}

\subsection{Large primes} \label{ss:large-primes}
In this subsection we will finally apply the results from \S \ref{s:products} to 
bound the contribution of the dyadic parts of \eqref{eq:two-sums-large-p} to 
\eqref{eq:pre-decomposition}.
More precisely, we prove:

\begin{lemma}[Contribution from large primes] \label{l:large-primes}
Under the assumptions of Proposition \ref{p:equid-non-corr}, the following holds.
Let $q$ be as in the Proposition \ref{p:equid-non-corr}, recall the 
definition of $P_D$ from \eqref{eq:I_D}, and let $E^{\sharp}_h(T,D,j)$ denote the 
expression
\begin{align*}
\bigg|
 \frac{DQ}{T}  
 \sum_{\substack{m \sim 2^{-j} X \\ \gcd(m,Q)=1}}
 \sum_{\substack{p \leq T/(mD) \\ p \equiv D' \bar{m} \Mod{Q}}} 
 \1_{P_D}(mp) h(m) h(p) \Lambda(p) 
 F\Big(g_D\Big(\frac{pm-D'}{Q}\Big) \Gamma\Big)
 \bigg|.
\end{align*}
Then, provided the parameter $E_0$ from Proposition \ref{p:equid-non-corr} is 
sufficiently large depending on $d$, $m_G$ and $H$, we have
\begin{align} \label{eq:large-primes}
&\sum_{i=1}^H
 \sum_{k=\frac{(\log \log T)^2}{\log 2}}^{(1-1/H)\frac{\log T}{\log2}}
 \sum_{\substack{D \sim 2^k\\ (D,Q)=1}}
 \1_{D \not\in \mathcal{B}_{2^k}}
 \sum_{\substack{d_1, \dots \widehat{d_{i}} \dots, d_H \\ D_i=D}}
 \bigg( \prod_{i' \not= i} \frac{|f_{i'}(d_{i'})|}{d_{i'}} \bigg) 
 \sum_{j=0}^{\log_2(X/T_0)} 
 \frac{E^{\sharp}_{f_i}(T,D,j)}{\log T}\\
\nonumber
&\ll
  \left(
  (\log \log T)^{-1/(2^{s+2} \dim G)} 
  + \frac{\delta(N)^{-10^s \dim G}}{(\log \log T)^{1/2}}
  \right)
  \frac{1+\|F\|_{\mathrm{Lip}}}{\log T}
  \frac{Q}{\phi(Q)}
  \prod_{\substack{p\leq T\\ p\nmid Q}}
  \left(1+\frac{|f(p)|}{p}\right),
\end{align}
where the implied constant may depend on $d$, $m_G$, $\alpha_f$, $E$ and $H$.
\end{lemma}
\begin{rem*}
Note that this contribution agrees with the bound \eqref{eq:prop-bound}.
\end{rem*}
The remainder of this subsection is concerned with the proof of 
\eqref{eq:large-primes}.
Considering $E^{\sharp}_h(T,D,j)$ for a fixed value of $j$, 
$1 \leq j \leq \log_2 \frac{X}{T_0}$, the Cauchy--Schwarz inequality yields
\begin{align} \label{eq:two-sums-1}
& 
 \sum_{\substack{m \sim 2^{-j}X \\ \gcd(m,Q)=1 }}
 \sum_{\substack{p \leq 2^{j} T/(XD) 
       \\ p \equiv D \bar{m} \Mod{Q} \\ mp \in P_D}} 
 \1_{mp\leq N} h(m) h(p) \Lambda(p) 
 F\Big(g_D\Big(\frac{pm-D'}{Q}\Big) \Gamma\Big)\\
\nonumber
&\leq 
 \bigg(
  \sum_{p \leq 2^j T/(XD)}
  |h(p)|^2 \Lambda(p)
 \bigg)^{1/2} 
 \Bigg( \frac{Q}{\phi(Q)} 
  \sum_{A' \in (\ZZ/Q\ZZ)^*} 
  \sum_{\substack{m,m' \sim 2^{-j}X 
  \\ m \equiv m' \equiv A' \Mod{Q} 
  \\ \gcd(Q,m)=1}}
  h(m)h(m') \\
\nonumber
& \qquad \qquad \qquad
  \frac{\phi(Q)}{Q} 
  \sum_{\substack{p \leq T/(D\max(m,m')) \\ pA' \equiv D' \Mod{Q} 
   \\ pm, pm' \in P_D}} 
  \Lambda(p)
  F\Big(g_D\Big(\frac{pm-D'}{Q}\Big) \Gamma\Big)
  \overline{ F\Big(g_D\Big(\frac{pm'-D'}{Q}\Big) \Gamma\Big)}
 \Bigg)^{1/2}.
\end{align}
The first factor is easily seen to equal $O(2^jT/(XD))$, since $h(p) \ll_H 1$ at 
primes.
To estimate the second factor, we seek to employ the orthogonality of the 
`$W$-tricked von Mangoldt function' with nilsequences, combined with the 
fact that for most pairs $(m,m')$ the product nilsequence that appears in the 
above expression is equidistributed (cf.\ Proposition \ref{p:equid}).
For this purpose, let us make the change of variables $p=Qn+D'_m$ 
in the inner sum of the second factor, where 
$D'_m$ is such that $D'_m \equiv D' \bar{m} \Mod{Q}$.
This yields 
\begin{align} \label{eq:cov}
& \frac{\phi(Q)}{Q}
  \sum_{\substack{p \leq T/(D\max(m,m')) \\ pA' \equiv D' \Mod{Q} 
   \\ pm,pm' \in P_D}} 
  \Lambda(p)
  F\Big(g_D\Big(\frac{pm-D'}{Q}\Big) \Gamma\Big)
  \overline{ F\Big(g_D\Big(\frac{pm'-D'}{Q}\Big) \Gamma\Big)} \\
\nonumber
&=
  \sum_{\substack{n \leq T/(QD\max(m,m')) 
   \\ nm+\widetilde D_{m}, nm'+\widetilde D_{m'} \in I_D}} 
  \frac{\phi(Q)}{Q} \Lambda(Qn+D'_m) 
  F(g_D(nm+\widetilde D_m)\Gamma) \overline{F(g_D(nm'+\widetilde D_{m'})\Gamma)},
\end{align}
for suitable values of $0\leq \widetilde D_{m}<m$, $0\leq \widetilde D_{m'}<m'$
and with $I_D=\{n: Dn + D'' \in I\}$ as defined in \eqref{eq:I_D} and $I$ as in 
the statement of Proposition \ref{p:equid-non-corr}.
Let us consider the summation range 
\begin{align*}
 I_{m,m'}
=\left\{n \in \NN : 
\begin{array}{c}
nm + \widetilde D_m \in I_D, \\
nm' + \widetilde D_{m'} \in I_D
\end{array}
\right\}
\end{align*}
in the above expression more closely.
Since $I$ is a discrete interval, $I_D$ is a discrete interval too and, for 
$m,m' \sim 2^{-j}X$, we have
$$
\#\Big\{
n \in \NN: nm+\widetilde D_{m} \in I_D
\Big\} \ll |I_D|2^{j}/X \ll |I|2^{j}/(DX) \leq T 2^j/(DX Q)
$$
and, similarly,
$
\#\{n \in \NN: nm'+\widetilde D_{m'} \in I_D\}  
\ll T 2^j/(DX Q).
$
We will now split the set
$$\{(m,m'): m,m'\sim 2^{-j}X, m \equiv m' \equiv A' \Mod{Q}\}$$
into two subsets; one containing all pairs $(m,m')$ for which
$
\#I_{m,m'} \leq \delta(N) 2^jT/(DXQ),
$
and one containing the pairs $(m,m')$ for which 
\begin{align}\label{eq:mm'-cases}
\#I_{m,m'}>\delta(N) 2^jT/(DXQ). 
\end{align}
In the former case, the trivial bound of \eqref{eq:cov} asserts that 
\begin{align*}
&\Bigg|
 \sum_{\substack{n \leq T/(QD\max(m,m')) 
   \\ n \in I_{m,m'}}} 
  \frac{\phi(Q)}{Q} \Lambda(Qn+D'_m) 
  F(g_D(nm+\widetilde D_m)\Gamma) \overline{F(g_D(nm'+\widetilde D_{m'})\Gamma)} 
  \Bigg|
\\
&\leq \frac{\delta(N)T2^j}{DXQ}.
\end{align*}
This leaves us to bound \eqref{eq:cov} in the case where \eqref{eq:mm'-cases} 
holds.

To start with, recall our assumption from the start of Section \ref{ss:MV} that 
all values of $D$ are unexceptional in the sense that $D \sim 2^k$ for some 
$k \geq (\log \log T)^2 / \log 2$ and $D \not\in \mathcal{B}_{2^k}$, where 
$\mathcal{B}_{K}$ was defined in Proposition \ref{p:linear-subsecs}.
Thus, for any fixed unexceptional value of $D$, the finite sequence 
$$
(g_D(n)\Gamma)_{n\leq T/(Dq)}
$$
is totally $\delta(N)^{c_1E_0}$-equidistributed.
Thus, applying Proposition \ref{p:equid} with $g = g_D$ and with $E_2=c_1E_0$, we 
obtain for every integer 
$$K \in \left[T_0, X \right]$$ 
an exceptional set $\mathcal{E}_{K}$ 
of size 
\begin{equation}\label{eq:E_K-bound}
 \# \mathcal{E}_{K} \ll \delta(T)^{O(c_1c_2E_0)}K^2
\end{equation}
such that for all pairs of integers 
$(m,m') \in (K,2K]^2 \setminus \mathcal{E}_{K}$
the following estimate holds:
\begin{align*}
\bigg|
\sum_{\substack{n \leq T/(QD\max(m,m')) 
   \\ nm+\widetilde D_{m}, nm'+\widetilde D_{m'} \in I_D}}
  F(g_D(nm+\widetilde D_m)\Gamma) \overline{F(g_D(nm'+\widetilde D_{m'})\Gamma)}
\bigg|
<  \frac{(1 + \|F\|_{\mathrm{Lip}}) \delta(N)^{c_1c_2E_0} T}{KQD}. 
\end{align*}
Before we continue with the analysis of \eqref{eq:cov}, we prove a quick lemma 
that will allow us to handle the contribution of exceptional sets 
$\mathcal{E}_{K}$ in the proof of Lemma \ref{l:large-primes}.
\begin{lemma}\label{l:E_K'-contribution}
Suppose $j \leq \log_2(X/T_0)$ and let $\mathcal{E}_{K}$ be the exceptional set 
obtained from Proposition \ref{p:equid} when applied with $g=g_D$.
Then, provided $E_0$ is sufficiently large, we have
 \begin{align*}
&\frac{1}{\phi(Q)} \sum_{A' \in (\ZZ/Q\ZZ)^*} 
  \sum_{\substack{m,m' \sim 2^{-j}X 
        \\ m \equiv m' \equiv A' \Mod{Q}}}
  |h(m)h(m')| \1_{(m,m') \in \mathcal{E}_{D,2^{-j}X}} \\
&\ll \delta(N)^{O(c_1c_2 E_0)} 
\Bigg(
\frac{2^{-j}X}{\phi(Q)}
    \frac{1}{\log (2^{-j}X)}
  \prod_{\substack{p\leq 2^{-j}X\\ p\nmid Q }} 
  \left(1 + \frac{|h(p)|}{p} \right)
\Bigg)^2,
\end{align*}
where $c_1$ and $c_2$ are the constants defined in Proposition 
\ref{p:linear-subsecs} and Proposition \ref{p:equid}, respectively.
\end{lemma}
\begin{proof}
In view of \eqref{eq:E_K-bound}, Cauchy--Schwarz yields
 \begin{align*}
& \frac{1}{\phi(Q)} 
  \sum_{A' \in (\ZZ/Q\ZZ)^*} 
  \sum_{\substack{m,m' \sim 2^{-j}X 
        \\ m \equiv m' \equiv A' \Mod{Q}}}
  |h(m)h(m')| \1_{(m,m') \in \mathcal{E}_{D,2^{-j}X}} \\
&\ll \frac{2^{-j}X}{\phi(Q)}  \delta(N)^{O(c_1c_2E_0)} 
\bigg( 
  \sum_{\substack{m, m' \sim 2^{-j}X \\ m \equiv m' \equiv A' 
   \Mod{Q} }}
  |h(m)|^2|h(m')|^2 \bigg)^{1/2} \\
&\ll \frac{2^{-j}X}{\phi(Q)} \delta(N)^{O(c_1c_2E_0)}  
  \sum_{\substack{m \sim 2^{-j}X \\ m \equiv A' 
   \Mod{Q} }} |h(m)|^2. 
\end{align*}
Since $2^{-j}X \geq T_0 = \exp((\log \log T)^2) \gg Q^2$, we may 
apply Shiu's bound \eqref{eq:shiu} and the trivial inequality 
$h(p)^2 \leq |h(p)|$ to obtain the upper bound
\begin{align*}   
&\ll \left(\frac{2^{-j}X}{\phi(Q)}\right)^2 
  \delta(N)^{O(c_1c_2E_0)} 
  \frac{1}{\log (2^{-j}X)}
  \prod_{\substack{p\leq 2^{-j}X\\ p\nmid Q }} 
  \left(1 + \frac{|h(p)|}{p} \right)\\
&\ll \delta(N)^{O(c_1c_2E_0)}\log (2^{-j}X)
  \Bigg( \frac{2^{-j}X}{\phi(Q)}
    \frac{1}{\log (2^{-j}X)}
  \prod_{\substack{p\leq 2^{-j}X\\ p\nmid Q }} 
  \left(1 + \frac{|h(p)|}{p} \right)
  \Bigg)^2.
\end{align*}
Recall that $X$ was defined in Section \ref{ss:decomposition} and satisfies
$X \leq T \ll N$.
Since furthermore $\delta(N) \leq (\log N)^{-1}$, any sufficiently large 
choice of $E_0$ 
guarantees that 
$$\delta(N)^{O(c_1c_2E_0)} \log (2^{-j}X) \leq \delta(N)^{O(c_1c_2E_0)}$$
holds. This completes the proof.
\end{proof}

As a final tool for the proof of Lemma \ref{l:large-primes}, we require an 
explicit bound on the correlation of the `$W$-tricked von Mangoldt 
function' with nilsequences.
The following lemma provides such bounds in our specific setting. 
We include a proof building on that of Green and Tao 
\cite[Proposition 10.2]{GT-linearprimes} in Appendix \ref{app}.
\begin{lemma} \label{l:Lambda}
Let $G/\Gamma$ be an $s$-step nilmanifold, let $G_{\bullet}$ be a filtration of 
$G$ of degree $d$ and let $\mathcal{X}$ be a $M$-rational Mal'cev basis adapted 
to it.
Let $\Lambda': \NN \to \RR$ be the restriction of the ordinary von 
Mangoldt function to primes, that is, $\Lambda'(p^k)=0$ whenever 
$k>1$.
Let $W=W(x)$, let $q'$ and $b'$ be integers such that  
$0<b'<Wq' \leq (\log x)^{E}$ and $\gcd(Wq',b')=1$ hold.
Let $\alpha \in (0,1)$. 
Then, for every $y \in [\exp((\log x)^{\alpha}), x]$ and for every
polynomial sequence $g \in \mathrm{poly}(\ZZ,G_{\bullet})$, the following 
estimate holds:
\begin{align*}
\left|\sum_{n\leq y} 
\frac{\phi(Wq')}{Wq'}\Lambda'(Wq'n+b') 
F(g(n)\Gamma) \right|
\ll_{\alpha, d, \dim G, E, \|F\|_{\mathrm{Lip}}} 
\left| \sum_{n\leq y} F(g(n)\Gamma)\right| + y\mathcal{E}(x)~,
\end{align*}
where
$$
\mathcal{E}(x) := (\log \log x)^{-1/(2^{2d+3} \dim G)} 
+ \frac{M^{O(10^d \dim G)}}{(\log \log x)^{1/2^{d+2}}}.
$$
\end{lemma}
Employing Lemma \ref{l:Lambda} for the upper endpoint of an interval $[y_0,y_1]$, 
and either a trivial estimate or the lemma for the lower endpoint,
say, depending on whether or not $y_0 \leq y_1^{1/2}$, we obtain as immediate 
consequence that
\begin{align}\label{eq:GT-Lambda-and-N-q-P}
&\left|\sum_{y_0 \leq n\leq y_1} 
\frac{\phi(Wq')}{Wq'}\Lambda'(Wq'n+b') 
F(g(n)\Gamma) \right|  \\
\nonumber
&\quad \ll_{\alpha, s, E, \|F\|_{\mathrm{Lip}}}
  \left|\sum_{n\leq y_0} F(g(n)\Gamma)\right|
+ \left|\sum_{n\leq y_1} F(g(n)\Gamma)\right| 
+ y_1^{1/2} + y_1 \mathcal{E}(x) ~.
\end{align}
for any $0 < y_0 < y_1 \leq x$ such that $y_1^{1/2} \geq \exp((\log 
x)^{\alpha})$.

This brings us back to the task of bounding \eqref{eq:cov} under the assumption 
of \eqref{eq:mm'-cases}.
We shall start by applying \eqref{eq:GT-Lambda-and-N-q-P} 
with $[y_0,y_1]=I_{m,m'}$ and $x=N=T^{1+o(1)}$.
To do so, note that \eqref{eq:mm'-cases} implies that
\begin{align*}
\frac{1}{2} \log y_1 
&\geq \log \Big( \delta(N) T2^j/(DX Q) \Big) \\
&\geq \log \frac{T}{DX} + j\log 2 +  \log \delta(N) - \log Q \\
&\geq (\log (T/D))^{1/U} + j \log 2 
      - \log \log N - 2E \log \log T  \\
&\geq \Big( \frac{\log T}{H} \Big)^{1/4} 
      - \log \log N 
      - 2E \log \log T  \\
&\gg_{E, H} (\log T )^{1/4},
\end{align*}
where we used the definition of $X$ and the assumptions that Proposition 
\ref{p:equid-non-corr} makes on $\delta$.
Thus, choosing $\alpha=1/5$, say, the conditions of Lemma \ref{l:Lambda} are 
satisfied for every $T$ that is sufficiently large with respect to $E$ and $H$.
Hence, \eqref{eq:GT-Lambda-and-N-q-P} yields the following estimate for the 
interval $[y_0,y_1]= I_{m,m'}$~:
\begin{align*}
\Bigg|\sum_{\substack{n \leq T/(QD\max(m,m')) 
   \\ n\in I_{m,m'}}} 
  &\frac{\phi(Q)}{Q} \Lambda(Qn+D'_m) 
  F(g_D(nm+\widetilde D_m)\Gamma) \overline{F(g_D(nm'+\widetilde D_{m'})\Gamma)} 
\Bigg|
\\
\ll_{s, E, H, \|F\|_{\mathrm{Lip}}}\quad
&\Bigg|\sum_{n \leq y_0}  
  F(g_D(nm+\widetilde D_m)\Gamma) \overline{F(g_D(nm'+\widetilde D_{m'})\Gamma)} 
\Bigg|
\\
+&
 \Bigg|
 \sum_{n \leq y_1}  
  F(g_D(nm+\widetilde D_m)\Gamma) \overline{F(g_D(nm'+\widetilde D_{m'})\Gamma)}
  \Bigg|
+\frac{T2^j}{DXQ} \mathcal{E}(T).
\end{align*}
Proposition \ref{p:equid} shows that the right hand side is small for most pairs 
$(m,m')$.
Indeed, together with Proposition \ref{p:equid}, the above implies that 
\eqref{eq:two-sums-1} is bounded above by
\begin{align*}
 \ll_{s, E, H, \|F\|_{\mathrm{Lip}}}
  \Big(\frac{T 2^j}{DX}\Big)^{1/2} 
  &\Bigg( \frac{Q}{\phi(Q)} 
  \sum_{A' \in (\ZZ/Q\ZZ)^*} 
  \sum_{\substack{m,m' \sim 2^{-j}X 
         \\ m \equiv m' \equiv A' 
          \Mod{Q}}}
  |h(m)h(m')| \times \\
 &\times \frac{T 2^j}{QDX} \bigg(\delta(N)^{O(c_1c_2E_0)}
  + \1_{(m,m') \in \mathcal{E}_{D,2^{-j}X}}
  +  
  \mathcal{E}(T)
  \bigg)
  \Bigg)^{1/2},
\end{align*}
Treating the part of this expression to which Lemma \ref{l:E_K'-contribution} 
applies separately and rewriting in the remaining part the sum over $m, m'$ as a 
square, we obtain after collecting together all the normalisation factors:
\begin{align*}
 \ll_{s, E, H, \|F\|_{\mathrm{Lip}}}~ 
 &\frac{T}{QD} 
  \Bigg\{ 
  \max_{\substack{A' \in  (\ZZ/Q\ZZ)^*}}
  \Bigg(
  \frac{Q2^j}{X}
  \sum_{\substack{m\sim 2^{-j}X 
         \\ m\equiv A' 
          (Q)}}
  |h(m)|\Bigg)^2 
  \Big(\delta(N)^{O(c_1c_2E_0)}
  +  
  \mathcal{E}(T)
  \Big)\\
  &+
 \delta(N)^{O(c_1c_2 E_0)} 
\Bigg(
\frac{Q}{\phi(Q)}
    \frac{1}{\log (2^{-j}X)}
  \prod_{\substack{p\leq 2^{-j}X\\ p\nmid Q }} 
  \left(1 + \frac{|h(p)|}{p} \right)
\Bigg)^2 \Bigg\}^{1/2}\\
 \ll_{s, E, H, \|F\|_{\mathrm{Lip}}}~
 &\frac{T}{QD}
  \Big(\delta(N)^{O(c_1c_2E_0)} + 
  \mathcal{E}(T)
  \Big)
  \Bigg(
\frac{Q}{\phi(Q)}
    \frac{1}{\log (2^{-j}X)}
  \prod_{\substack{p\leq 2^{-j}X\\ p\nmid Q }} 
  \left(1 + \frac{|h(p)|}{p} \right)
\Bigg),
\end{align*}
where we applied Shiu's bound in the last step.
Summing the above expression over $j \leq \log_2(\frac{X}{T_0})$ and taking into 
account the factor $(\log T)^{-1}$, we deduce that the inner sum in 
\eqref{eq:large-primes} is bounded by
\begin{align*}
\ll_{s, E, H, \|F\|_{\mathrm{Lip}}}&
  \Big(\delta(N)^{O(c_1c_2E_0)} + 
  \mathcal{E}(T)  \Big)
  \Bigg(
  \frac{1}{\log T}
  \frac{Q}{\phi(Q)}
  \prod_{\substack{p\leq T\\ p\nmid Q }} 
  \left(1 + \frac{|h(p)|}{p} \right)
  \Bigg)\\
  &\qquad \times
  \sum_{j=1}^{\log_2(\frac{X}{T_0})}
  \frac{1}{\log (2^{-j}X)}
  \prod_{2^{-j}X < p' < T } 
  \left(1 - \frac{|h(p')|}{p'} \right)
\end{align*}
Since $\delta(N) \leq (\log N)^{-1}$, choosing $E_0$ sufficiently large in terms 
of $d$ and $m_0$ ensures that
$$
\delta(N)^{O(c_1c_2E_0)} + \mathcal{E}(T) 
\ll (\log N)^{-1} + \mathcal{E}(T)
\ll \mathcal{E}(T).
$$
To complete the proof of Lemma \ref{l:large-primes}, it thus remains to show that 
the inner sum over $j$ in the expression above is $O_{\alpha_f}(1)$.
To see this, observe that property (2) of Definition  \ref{def:M} yields
$$
\prod_{\substack{X2^{-j} < p\leq T}}
\left(1-\frac{|h(p)|}{p}\right)
\ll \left( \frac{\log(2^{-j}X)}{\log T}\right)^{\alpha_f/H}.
$$
Thus, 
\begin{align*}
 &\sum_{j=1}^{\log_2(\frac{X}{T_0})}
  \frac{1}{\log (2^{-j}X)}
  \prod_{2^{-j}X < p' < T} 
  \left(1 - \frac{|h(p')|}{p'} \right)\\
 &\ll  \frac{1}{(\log T)^{\alpha_f/H}}
  \sum_{j=1}^{\log_2(\frac{X}{T_0})}
  \frac{1}{(\log X - j \log 2)^{1-\alpha_f/H}}
 \ll_{\alpha_f} \frac{(\log X)^{\alpha_f/H}}{(\log T)^{\alpha_f/H}}
 \ll_{\alpha_f} 1,
 \end{align*}
as required.

\subsection{Small primes} \label{ss:small-primes}
To complete the proof of Proposition \ref{p:equid-non-corr}, it remains to bound 
the contribution of the dyadic parts of \eqref{eq:two-sums-small-p}
to \eqref{eq:pre-decomposition}. 
This is achieved by the following lemma, which will be proved by a combination 
of Cauchy-Schwarz, Lemma \ref{l:dirichlet} and the choice of the parameter $X$ 
from Section \ref{ss:decomposition}.
\begin{lemma}[Contribution from small primes] \label{l:small-primes}
Let $E^{\flat}_h(T,D,j)$ denote the expression
$$
 \bigg|
 \frac{DQ}{T}
 \sum_{\substack{m>X \\ \gcd(m,Q)=1}} 
 \sum_{\substack{p \sim 2^{-j}T/(XD) 
       \\ p \equiv D' \bar{m} \Mod{Q}}}
 \1_{pm<T/D}
 \1_{P_D}(mp) h(m) h(p) \Lambda(p) 
 F\Big(g_D\Big(\frac{pm-D'}{Q}\Big) \Gamma\Big)
 \bigg|.
$$
Then
\begin{align*}
&\sum_{i=1}^H
 \sum_{k=\frac{(\log \log T)^2}{\log 2}}^{(1-1/H)\frac{\log T}{\log2}}
 \sum_{\substack{D \sim 2^k \\ (D,Q)=1}}
 \1_{D \not\in \mathcal{B}_{2^k}}
 \sum_{\substack{d_1, \dots \widehat{d_{i}} \dots, d_H \\ D_i=D}} 
 \bigg( \prod_{j \not= i} \frac{|f_j(d_j)|}{d_j} \bigg) 
 \sum_{j=0}^{\log_2(T/(XDT_0))}
 \frac{E^{\flat}_{f_i}(T,D,j)}{\log T}\\
&\ll (\log T)^{-1/4} 
 \frac{1}{\log T} \frac{\phi(Q)}{Q}
 \prod_{\substack{p\leq T \\ p \nmid Q}}
 \left(1+\frac{|f(p)|}{p}\right).
\end{align*}
\end{lemma}

\begin{proof}
Applying Cauchy--Schwarz to the expression $E^{\flat}_h(T,D,j)$ for a 
fixed value of $j$, $0 \leq j \leq \log_2(T/(XDT_0))$, we obtain
\begin{align} \label{eq:another-CS}
&\Bigg|\frac{QD}{T} 
 \sum_{\substack{m>X \\ (m,Q)=1 }} 
 \sum_{\substack{p \sim 2^{-j} T/(XD) 
       \\ p \equiv D' \bar m \Mod{Q} }} 
 1_{pm<T/D} h(m) h(p) \Lambda(p) 
 F\Big(g\Big(\frac{pm-D'}{Q}\Big) \Gamma\Big) \1_{P_D}(mp) 
 \Bigg|\\ 
\nonumber
&\leq  
 \bigg(\frac{Q}{\phi(Q)} \frac{1}{2^j X}
 \sum_{\substack{X < m < 2^j X \\ \gcd(m,Q)=1 }} |h(m)|^2
 \bigg)^{1/2}\\ \nonumber
&\quad \times
 \Bigg( \phi(Q) \Big(\frac{2^{j}XD}{T}\Big)^2
 \sum_{\substack{p,p' \sim 2^{-j} T/(XD) \\ p \equiv p' \Mod{Q} }}
 h(p)h(p') \Lambda(p) \Lambda(p') \\ \nonumber
& \qquad \qquad \qquad \qquad
 \frac{Q}{X 2^j}
 \sum_{\substack{X<m<T/(D\max(p,p')) 
       \\ mp \equiv D' \Mod{Q}
       \\ pm\in I_D}}
 F\Big(g\Big(\frac{pm-D'}{Q}\Big) \Gamma\Big)
  \overline{ F\Big(g\Big(\frac{p'm-D'}{Q}\Big) \Gamma\Big)}
 \Bigg)^{1/2}.
\end{align}
We estimate the second factor trivially as $O(1)$ by using the bounds
$|h(p)h(p')| \ll 1$ and $\|F\|_{\infty}=\|\bar F\|_{\infty} \leq 1$.
Thus, \eqref{eq:another-CS} is bounded by
\begin{align*}
 \bigg(\frac{Q}{\phi(Q)} \frac{1}{2^j X}
 \sum_{\substack{X < m < 2^j X \\ \gcd(m,Q)=1 }} |h(m)|^2
 \bigg)^{1/2}.
\end{align*}
This expression can be handled as in Lemma \ref{l:dirichlet}:
Note that $X \leq 2^j X \leq T/(DT_0)$, where
$$X= \left(\frac{T}{D}\right)^{1-1/\left(\log \frac{T}{D}\right)^{(U-1)/U}} 
\gg (T/D)^{1/2}$$ and
$$
\frac{T}{DT_0} 
= \left(\frac{T}{D}\right)^{
1 - \left(\log \log \frac{T}{D}\right)^2/\left( \log \frac{T}{D} \right)
}.
$$
Thus, Shiu's bound \eqref{eq:shiu} and the trivial inequality 
$|h(p)|^2\leq|h(p)|$ imply that
\begin{align*}
 \bigg(\frac{Q}{\phi(Q)} \frac{1}{2^j X}
 \sum_{\substack{X < m < 2^j X \\ \gcd(m,Q)=1 }} |h(m)|^2
 \bigg)^{1/2}
 &\ll 
 \bigg(
 \frac{1}{\log T} \frac{Q}{\phi(Q)} \prod_{\substack{p \leq T\\ p\nmid Q}}
 \left(1 + \frac{|h(p)|}{p}\right)\bigg)^{1/2}.
\end{align*}
The right hand side is bounded below by $(\log T)^{-1/2}$, thus the above is 
bounded by
\begin{align*}
  \ll (\log T)^{1/2} 
  \bigg(
  \frac{1}{\log T} \frac{Q}{\phi(Q)} \prod_{\substack{p \leq T\\ p\nmid Q}}
 \left(1 + \frac{|h(p)|}{p}\right) \bigg).
\end{align*}
Finally, note that the summation range in $j$ is short: it is bounded by
$$\log_2(T/(XDT_0)) \ll (\log T)^{1/U} 
\ll (\log T)^{1/4}.$$
This shows that
\begin{align*}
&\sum_{i=1}^H
 \sum_{k=1}^{(1-1/H)\frac{\log T}{\log2}}
 \sum_{\substack{D \sim 2^k\\(D,Q)=1}} 
 \1_{D \not\in \mathcal{B}_{2^k}}
 \sum_{\substack{d_1, \dots \widehat{d_{i}} \dots, d_H \\ D_i=D}} 
 \bigg( \prod_{j \not= i} \frac{|f_j(d_j)|}{d_j} \bigg) 
 \sum_{j=0}^{\log_2(T/(XDT_0))} 
 \frac{E^{\flat}_{f_i}(T,D,j)}{\log T}\\
&\ll (\log T)^{-1+\frac{1}{2}+\frac{1}{4}} 
  \sum_{i=1}^H
  \sum_{\substack{D \leq T^{1-1/H}\\(D,Q)=1}} 
   \sum_{\substack{d_1, \dots \widehat{d_{i}} \dots, d_t \\ D_i=D}} 
  \bigg( \prod_{j \not= i} \frac{|f_j(d_j)|}{d_j} \bigg) 
  \frac{1}{\log T} \frac{Q}{\phi(Q)} \prod_{\substack{p \leq T\\ p\nmid Q}}
  \left(1 + \frac{|h(p)|}{p}\right) \\
&\ll (\log T)^{-1/4} 
  \frac{1}{\log T} \frac{Q}{\phi(Q)} \prod_{\substack{p \leq T\\ p\nmid Q}}
 \left(1 + \frac{|f(p)|}{p}\right).
\end{align*}
This completes the proof of Lemma \ref{l:small-primes} as well as the proof of 
Proposition \ref{p:equid-non-corr}.
\end{proof}

\appendix
\section{Explicit bounds on the correlation of $\Lambda$ with nilsequences} 
\label{app}
The aim of this appendix is to provide a proof of Lemma \ref{l:Lambda}. 
This result is due to Green and Tao and we expect that a statement 
like Lemma \ref{l:Lambda} will eventually appear in \cite{G-notes}. 
The author is grateful to Ben Green for very helpful discussions.

The proof of Lemma \ref{l:Lambda} rests upon the decomposition of $\Lambda'$ that 
already appeared in the proof of the original result, 
\cite[Proposition 10.2]{GT-linearprimes}.
To be precise, let $\gamma \in (0,1)$ be a small positive real number that will 
later be chosen depending on the degree $d$ of the given filtration 
$G_{\bullet}$.
Further, let $\chi^{\flat} + \chi^{\sharp} = \id_{\RR}$ be a smooth 
decomposition of the identity function $\id_{\RR}: \RR \to \RR$, 
$\id_{\RR}(t):=t$, that is such that $\supp (\chi^{\sharp}) \subset (-1,1)$ and 
$\supp (\chi^{\flat}) \subset \RR \setminus [-1/2,1/2]$.
This decomposition of $\id_{\RR}$ induces the following decomposition of 
$\Lambda'$:
$$
\frac{\phi(Wq')}{Wq'}\Lambda'(Wq'n + b') - 1
= \frac{\phi(Wq')}{Wq'}\Lambda^{\flat}(Wq'n + b') +
  \Big(\frac{\phi(Wq')}{Wq'}\Lambda^{\sharp}(Wq'n + b') - 1 \Big),
$$
where, cf.\ \cite[(12.2)]{GT-linearprimes}, 
$$
\Lambda^{\sharp}(n) 
= - \log x^{\gamma} 
 \sum_{d|n} \mu(d) \chi^{\sharp} 
 \Big( \frac{\log d}{\log x^{\gamma}} \Big) 
  \qquad (|t| \geq 1 \Rightarrow \chi^{\sharp} (t) = 0)
$$
is a truncated divisor sum, where
$$
\Lambda^{\flat}(n)
= - \log x^{\gamma} 
 \sum_{d|n} \mu(d) \chi^{\flat} 
 \Big( \frac{\log d}{\log x^{\gamma}} \Big) 
  \qquad (|t| \leq 1/2 \Rightarrow \chi^{\flat} (t) = 0)
$$
is an average of $\mu (d)$ running over large divisors of $n$.
This decomposition in turn splits the correlation from Lemma \ref{l:Lambda} into 
two correlations that shall be bounded separately.

The correlation estimate of the $\Lambda^{\flat}$ term with nilsequences follows
as in \cite[\S12]{GT-linearprimes} from the non-correlation of M{\"o}bius 
with nilsequences and inherits an error term which saves a factor 
$O_A(\log x)^{-A}$ for any given $A \geq 1$ when compared to the trivial bound.
In \cite[Conjecture 8.5]{GT-linearprimes}, it was conjectured that the M\"obius 
function is orthogonal to linear nilsequence.
Since \cite[Theorem 1.1]{GT-nilmobius} proves this conjecture, not just for 
linear, but for polynomial nilsequences, it follows without any essential 
changes in the proof, that the correlation estimate \cite[eq.\ 
(12.10)]{GT-linearprimes} continues to hold for polynomial sequences.
That is to say, we have an estimate of the form
\begin{equation}\label{eq:poly-flat-corr}
 \bigg|\sum_{n \leq N} \Lambda^{\flat}(n) F(g(n)\Gamma)\bigg|
 \ll_{\|F\|_{\mathrm{Lip}},G/\Gamma,s,A} N(\log N)^{-A}.
\end{equation}

In our setting, we may express the congruence condition modulo $Wq'$ as a 
character sum
$$
\frac{\phi(Wq')}{Wq'}\Lambda^{\flat}(n)\1_{n\equiv b' \Mod{Wq'}}
= 
\EE_{\chi \Mod{Wq'}} 
\frac{\phi(Wq')}{Wq'} \Lambda^{\flat}(n) \chi(n)\bar\chi(b'). 
$$
Following \cite{GT-linearprimes}, cf.\ equation (12.8), the factor 
$F(g(n)\Gamma)$ from the statement of Lemma \ref{l:Lambda} may be reinterpreted 
as $F(g'(Wq'n+b')\Gamma)$ for a new polynomial sequence $g'$.
Reinterpreting the product $\chi(n) F(g'(n)\Gamma)$ of a character $\chi$ with 
the given nilsequence as a nilsequence itself allows us to employ 
the correlation estimate \eqref{eq:poly-flat-corr} with $N$ given by 
$xq'W \ll x (\log x)^{E}$ to handle the correlation for $\Lambda^{\flat}$.
Thanks to the saving of an arbitrary power of $\log x$ in  
\eqref{eq:poly-flat-corr}, we can compensate the factor of $Wq'$, which is 
bounded above by $(\log x)^{E}$, that we loose when passing to the character 
sums.
In total, we obtain
$$
\frac{1}{y}\sum_{n\leq y} 
\frac{\phi(Wq')}{Wq'}\Lambda^{\flat}(Wq'n+b')
F(g(n)\Gamma) 
\ll_{\|F\|_{\mathrm{Lip}},s, G/\Gamma, B}
(\log y)^{-B}
\ll_{\|F\|_{\mathrm{Lip}},s, G/\Gamma, B'}
(\log x)^{-B'}.
$$

It remains to analyse the contribution of the function 
$\lambda^{\sharp}:\NN \to \RR$, defined via
$$\lambda^{\sharp}(n) 
:= \frac{\phi(Wq')}{Wq'}\Lambda^{\sharp}(Wq'n+b') - 1.$$
This contribution satisfies the general bound
\begin{align*}
\bigg|
\frac{1}{y}
\sum_{n \leq y} 
\Big(\frac{\phi(Wq')}{Wq'}\Lambda^{\sharp}(Wq'n+b') - 1\Big) 
F(g(n)\Gamma) 
\bigg|
\leq
\left\|\lambda^{\sharp}\right\|_{U^{k+1}[y]} 
\left\|F(g(\cdot)\Gamma)\right\|_{U^{k+1}[y]^*}
\end{align*}
for every $k \geq 1$, where the dual uniformity norm is defined via
$$
\| F (g(\cdot)\Gamma)\|_{U^{k+1}[N]^*} 
 := \sup \Big\{ \Big|\frac{1}{N} \sum_{n \leq N} f(n) F (g(n)\Gamma) 
    \Big| : \|f\|_{U^{k}[N]}\leq 1 \Big\}.
$$
The main task that remains is to obtain control on the above dual uniformity 
norm for at least one value of $k$.
In \cite{GT-linearprimes}, this is achieved through
\cite[Proposition 11.2]{GT-linearprimes}, which decomposes a general 
nilsequence into an averaged nilsequence of bounded dual uniformity norm plus an 
error term that is small in the $L^{\infty}$ norm.
The proof of this decomposition uses a compactness argument and, as such, 
does not provide explicit error terms.
Central ideas for a new approach not working with compactness were indirectly 
provided by work of Eisner and Zorin-Kranich~\cite{E-ZK} on a different question. 
Eisner and Zorin-Kranich replace in their work the Lipschitz function in the 
definition of a nilsequence by a smooth function and the Lipschitz norm by a 
Sobolev norm.
Moreover, they show that certain constructions that play a central role in 
\cite{GT-polyorbits} have counterparts in the Sobolev norm setting.
Building on these observations, Green \cite{G-notes} 
proves that in the Sobolev norm setting the dual $U^{s+1}$ norm of an $s$-step 
nilsequence \emph{is} in fact bounded.
The statement of the latter result involves the following notion of Sobolev norms.
\begin{definition}[cf.\ \cite{G-notes}]
 Let $G/\Gamma$ be an $m$-dimension nilmanifold together with a Mal'cev basis 
 $\mathcal X = \{X_1, \dots, X_m\}$.
 For any $\psi \in C^{\infty}(G/\Gamma)$, set 
 $$
 \|\psi\|_{W^m,\mathcal{X}} 
 = \sup_{m' \leq m} 
 \sup_{1 \leq i_1, \dots, i_{m'} \leq m} 
 \| D_{X_{i_1}} \dots D_{X_{i_{m'}}} \psi \|_{\infty},
 $$
 where 
 $D_X \psi (g\Gamma)
 = \lim_{t\to 0}\frac{\mathrm{d}}{\mathrm{d}t} \psi(\exp(tX) g \Gamma)$.
\end{definition}

\begin{lemma}[Green \cite{G-notes}, Theorem 5.3.1]\label{l:bounded-U*-norm}
 Let $G/\Gamma$ be a $k$-step nilmanifold together with a filtration 
 $G_{\bullet}$ of degree $d \geq k$ and a $M$-rational Mal'cev basis adapted to 
it. 
 Let $g \in \mathrm{poly}(\ZZ, G_{\bullet})$ and suppose 
 $\tilde F \in C^{\infty}(G/\Gamma)$.
 Then
 \begin{align*}
 \|\tilde F (g(\cdot)\Gamma)\|_{U^{d+1}[N]^*} 
 &:= \sup \Big\{ \Big|\frac{1}{N} \sum_{n \leq N} f(n) \tilde F (g(n)\Gamma) 
    \Big| : \|f\|_{U^{d+1}[N]}\leq 1 \Big\} \\
 &\ll M^{10^d \dim G}  \| \tilde F \|_{W^{2^d \dim G},\mathcal{X}}. 
 \end{align*}
\end{lemma}
In order to apply Lemma \ref{l:bounded-U*-norm} in our situation, an 
auxiliary result is needed that allows one to pass from the Lipschitz setting to 
the Sobolev setting, i.e. to write any Lipschitz function on $G/\Gamma$ as the sum 
of a smooth function, to which Lemma \ref{l:bounded-U*-norm} can be applied, and a 
small $L^{\infty}$ error.
This is the content of the following lemma which will be proved using a 
standard smoothing trick; the author thanks Ben Green for pointing out this 
approach.
\begin{lemma}\label{l:smoothing}
 Suppose that $F:G/\Gamma \to \CC$ is a Lipschitz function and let $m$ be a 
 positive integer.
 Then there is a constant $c\in (0,1)$, only depending on $G$, such that for 
 every $\eps \in (0,c)$ there exists a function $\psi_m \in C^{\infty}(G/\Gamma)$ 
 such that
 \begin{equation}\label{eq:smoothing-L_inf}
 \|F- F*\psi_m\|_{\infty} \leq \eps (1 + \|F\|_{\mathrm{Lip}})
 \end{equation}
 and
 \begin{equation}\label{eq:smoothing-Sob}
 \| F*\psi_m \|_{W^{m}, \mathcal{X}} 
 \ll 
 (m/\eps)^{2m} M^{O(m)}.  
 \end{equation}
\end{lemma}

Taking Lemma \ref{l:smoothing} on trust for the moment, we first complete the 
proof of Lemma \ref{l:Lambda} before providing that of Lemma \ref{l:smoothing}.
Recall that the filtration $G_{\bullet}$ of the nilmanifold $G/\Gamma$ from 
Lemma \ref{l:Lambda} is of degree $d$.
The previous two lemmas allow us to reduce the proof of Lemma \ref{l:Lambda} to a 
bound on the $U^{d+1}$-norm of $\lambda^{\sharp}:\NN \to \RR$.
More precisely, we have
\begin{align} \label{eq:A-smoothing+dual}
\nonumber
&\frac{1}{y}
\sum_{n \leq y} 
\Big(\frac{\phi(Wq')}{Wq'}\Lambda^{\sharp}(Wq'n+b') - 1\Big) 
F(g(n)\Gamma) \\
\nonumber
&\ll
\eps(1 + \|F\|_{\mathrm{Lip}})
+ \frac{1}{y}
\sum_{n \leq y} 
\Big(\frac{\phi(Wq')}{Wq'}\Lambda^{\sharp}(Wq'n+b') - 1\Big) 
(F*\psi_m)(g(n)\Gamma) \\
&\ll
\eps(1 + \|F\|_{\mathrm{Lip}})
+ 
\left\|\lambda^{\sharp}\right\|_{U^{d+1}[y]} 
\left\|(F*\psi_m) (g(\cdot)\Gamma) \right\|_{U^{d+1}[y]^*}.
\end{align}
Since $\Lambda^{\sharp}$ is a truncated divisor sum, one can analyse 
its $U^{d+1}$-norm with the help of \cite[Theorem D.3]{GT-linearprimes}.
We will follow the final paragraph (`The 
correlation estimate for $\Lambda^{\sharp}$') of 
\cite[Appendix D]{GT-linearprimes} closely.

For each non-empty subset $\mathcal{B} \subset \{0,1\}^{d+1}$, let 
$$
\Psi_{\mathcal{B}}(n,\mathbf h) 
= \Big(Wq'(n+\boldsymbol\omega \cdot \mathbf h) + b'\Big)_{\omega \in 
\mathcal{B}}~,\quad (n,\mathbf h) \in \ZZ \times \ZZ^{d+1},
$$
denote the relevant system of forms.
The set of exceptional primes for this system, denoted by 
$\mathcal{P}_{\Psi_{\mathcal{B}}}$, 
is defined to be the set of all primes $p$ such that the 
reduction modulo $p$ of $\mathcal{P}_{\Psi_{\mathcal{B}}}$ contains two linearly 
dependent forms or a form that degenerates to a constant.
It is clear that whenever $x$ is sufficiently large, the set 
$\mathcal{P}_{\Psi_{\mathcal{B}}}$ consists of all prime factors of $W(x)q'$ and, 
in particular, it contains all primes up to $w(x)$.
For each prime $p$, the local factor $\beta_p^{(\mathcal{B})}$ corresponding to 
$\Psi_{\mathcal{B}}$ is defined to be
$$
\beta_p^{(\mathcal{B})}=
\frac{1}{p^{d+2}}
\sum_{(n,\mathbf{h}) \in (\ZZ/p\ZZ)^{d+2}}
\prod_{\omega \in \mathcal{B}}
\frac{p}{\phi(p)}\1_{p \,\nmid\, Wq'(n+\boldsymbol\omega \cdot \mathbf h) + b'}~.
$$
By \cite[Lemma 1.3]{GT-linearprimes}, we have 
$\beta_{p}^{(\mathcal{B})} = 1 + O_d(1/p^2)$ 
for all $p \not\in \mathcal{P}_{\Psi_{\mathcal{B}}}$, and hence
$$
\prod_{p \not\in \mathcal{P}_{\Psi_{\mathcal{B}}}}
\beta_{p}^{(\mathcal{B})} 
= 1 + O_d\left(\frac{1}{w(x)}\right)
= 1 + O_d\left(\frac{1}{\log \log x}\right),
$$
while the product of exceptional local factors satisfies 
$$
\prod_{p \in \mathcal{P}_{\Psi}}\beta_{p}^{(\mathcal{B})} 
= \prod_{p| W(x)q'} \beta_{p}^{(\mathcal{B})} 
= \left(\frac{W(x)q'}{\phi(W(x))q'}\right)^{|\mathcal{B}|},
$$
since $\gcd(W(x)q',b')=1$.

Let $K_y$ be a convex body that is contained in the hypercube $[-y,y]^{d+2}$. 
Then, \cite[Theorem D.3]{GT-linearprimes}, applied with $a_i=1$ and 
$\chi_i=\chi_{\#}$, implies that if $\gamma > 0$ is sufficiently small 
depending on $d$, then
\begin{align*}
 & \frac{1}{y^{d+2}}
   \sum_{(n,\mathbf{h}) \in K_y} 
   \prod_{\omega \in \mathcal{B}}
   \Lambda^{\sharp}\Big(
   Wq'(n+\boldsymbol\omega \cdot \mathbf h) + b'\Big) \\
  &\qquad = \frac{\vol(K_y)}{y^{d+2}} \prod_p \beta_{p}^{(\mathcal{B})}
  + O_d\bigg((\log y^{\gamma})^{-1/20} 
  \exp \bigg(
  O_d\bigg(\sum_{p \in P_{\Psi_{\mathcal{B}}}} p^{-1/2}\bigg)\bigg)\bigg).
\end{align*}
Since $Wq'\leq (\log x)^E$, we have 
$|\mathcal{P}_{\Psi_{\mathcal{B}}}| 
\ll \frac{w(x)}{\log x} + \frac{E\log \log x}{\log w(x)}$.
Recall that $w(x) \leq \log \log x$ and that
$\log y \in [(\log x)^{\alpha}, \log x]$. 
Thus,
\begin{align*}
 (\log y^{\gamma})^{-1/20} 
 \exp \bigg(O_d\bigg(\sum_{p \in \mathcal{P}_{\Psi_{\mathcal{B}}}} 
p^{-1/2}\bigg)\bigg)
& \ll  (\gamma (\log x)^{\alpha})^{-1/20} 
 \exp \Big(O_d(|P_{\Psi_{\mathcal{B}}}|)\Big) \\
& \ll_d (\log x)^{-\alpha/20} (\log x)^{O_d(E)/\log w(x)},
\end{align*}
which is $o(1)$ as $x \to \infty$.

Choosing 
$K_y = 
\{(n,\mathbf{h}): 0 < n + \boldsymbol \omega \cdot \mathbf h \leq y
\text{ for all } \boldsymbol \omega \in \{0,1\}^{d+1} \}$, we obtain
\begin{align*}
\left\|\lambda^{\sharp}\right\|_{U^{d+1}[y]}^{2^{d+1}} 
&= \frac{\vol(K_y)}{y^{d+2}}
\sum_{\mathcal{B} \subseteq \{0,1\}^{d+1} } 
(-1)^{|\mathcal{B}|}
\prod_{p \not\in P_{\Psi_{\mathcal{B}}}} \beta_{p}^{(\mathcal{B})}
+ O_{d}\Big((\log x)^{-\alpha/20 + \frac{O_d(E)}{\log w(x)}}\Big)\\
&\ll_{d} 
\frac{\vol(K_y)}{y^{d+2}}
\frac{1}{\log \log x}
+ (\log x)^{-\alpha/20 + \frac{O(E)}{\log w(x)}}\\
&\ll_{d, \alpha, E} 
\frac{1}{\log \log x}
\end{align*}
Returning to \eqref{eq:A-smoothing+dual}, it follows from the above bound,
Lemma \ref{l:bounded-U*-norm} and an application of 
Lemma \ref{l:smoothing} with $m=2^d \dim G$ and 
$\eps= (\log \log x)^{-1/(m2^{d+3})}$,   
that for $\exp((\log x)^{\alpha}) \leq y \leq x$
\begin{align*}
&\frac{1}{y}
\sum_{n \leq y} 
\Big(\frac{\phi(Wq')}{Wq'}\Lambda'(Wq'n+b') - 1\Big) F(g(n)\Gamma) \\
&\ll_{d,\alpha, E}
\frac{1 + \|F\|_{\mathrm{Lip}} }
     {(\log \log x)^{1/(2^{2d+3} \dim G)}}
+ 
\|\lambda_{\sharp} \|_{U^{d+1}[y]} 
\Big\|(F*\psi_m) (g(\cdot)\Gamma)\Big\|_{U^{d+1}[y]^*}
\\
& \ll_{d,\alpha, E} 
\frac{1 + \|F\|_{\mathrm{Lip}} }
     {(\log \log x)^{1/(2^{2d+3} \dim G)}}
+ \frac{M^{10^d \dim G} \| F*\psi_m \|_{W^{2^d \dim G},\mathcal{X}}}
       {(\log \log x)^{1/2^{d+1}}}\\
& \ll_{d,\dim G, \alpha, E} 
\frac{1 + \|F\|_{\mathrm{Lip}} }
     {(\log \log x)^{1/(2^{2d+3} \dim G)}}
+ \frac{(\log \log x)^{1/2^{d+2}} M^{O(10^d \dim G)}}{(\log \log 
x)^{1/2^{d+1}}}~,
\end{align*}
which reduces the proof of Lemma \ref{l:Lambda} to that of Lemma 
\ref{l:smoothing}.

\begin{proof}[Proof of Lemma \ref{l:smoothing}]
Let $d_{\mathcal{X}}$ denote the metric on $G/\Gamma$ that was introduced in 
\cite[Definition~2.2]{GT-polyorbits} and define for every $\eps'>0$ the following
$\eps'$-neighbourhood
$$\mathcal{B}_{\eps'} 
= \{x \in G/\Gamma: d_{\mathcal{X}}(x,\id_G\Gamma) < \eps' \}.$$
Let $\eps \in (0,1)$.
Since $F$ is Lipschitz, we have 
$|F(x)-F(y)| \leq \eps (1+\|F\|_{\mathrm{Lip}})$
whenever both $x$ and $y$ belong 
to the neighbourhood $\mathcal{B}_{\eps}$ of $\id_G \Gamma$.
To ensure that \eqref{eq:smoothing-L_inf} holds, 
it thus suffices to ensure that $\psi_m$ is non-negative, 
supported in $\mathcal{B}_{\eps}$ and that 
$\int_{G/\Gamma} \psi_m = 1$.
Indeed, these assumptions imply that
\begin{align*}
\bigg|F(x)-\int_{G/\Gamma} F(y)\psi_m(x-y)dy\bigg| 
&= \bigg|\int_{G/\Gamma} (F(y)-F(x))\psi_m(x-y)dy\bigg| \\
&\leq \eps (1+\|F\|_{\mathrm{Lip}}) \int_{G/\Gamma} \psi_m(x-y)dy 
= \eps (1+\|F\|_{\mathrm{Lip}}).  
\end{align*}
The function $\psi_m$ will be constructed as the $m$-fold convolution of a smooth 
bump-function.
For this purpose, observe that 
$$m\mathcal{B}_{\eps/m} \subseteq 
\mathcal{B}_{\eps}.
$$
If $g = \exp(s_1 X_1) \dots \exp(s_{\dim G} X_{\dim G})$, then the 
(unique) coordinates
$$\psi(g):=(s_1, \dots, s_{\dim G})$$ are called Mal'cev coordinates,
while the unique coordinates 
$$\psi_{\exp}(g):=(t_1, \dots, t_{\dim G})$$ 
for which $g=\exp(t_1 X_1 + \dots + t_{\dim G} X_{\dim G})$ are called 
exponential coordinates.
Proceeding as in the proof of \cite[Lemma A.14]{GT-polyorbits}, one can 
identify $G/\Gamma$ with the fundamental domain 
$\{g \in G: \psi(g) \in [-\frac{1}{2},\frac{1}{2})\} \subset G$.
Furthermore, \cite[Lemma A.2]{GT-polyorbits} shows that the change of 
coordinates between exponential and Mal'cev coordinates, i.e. $\psi \circ 
\psi_{\exp}^{-1}$ or $\psi_{\exp} \circ \psi^{-1}$, is in either direction a 
polynomial mapping with $M^{O(1)}$-rational coefficients.
Thus, $\mathcal{B}_{\eps}$ lies within the fundamental domain provided 
$\eps < c_0$ 
for some sufficiently small constant $c_0$. 
This embedding of $\mathcal{B}_{\eps}$ in $G$ allows us to define $\log$ on 
$\mathcal{B}_{\eps}$.
Let us equip $\mathfrak{g}$ with the maximum norm associated to 
$\mathcal{X}$, that is $\|X\|:=\max_i |t_i|$ for $X=\sum_i t_i X_i$.
Then the definition of $d_{\mathcal{X}}$ and \cite[Lemma A.2]{GT-polyorbits} 
imply that
$$
\{X \in \mathfrak{g}: \|X\| < \delta \}
\subseteq \log \mathcal{B}_{\eps/m}
$$
for some $\delta$ of the form $\delta = \frac{\eps}{m} M^{-O(1)}$.
Following the above preparation, we now choose a non-negative smooth 
function $\chi_1:\RR^{\dim G} \to \RR_{\geq 0}$ with support in 
$\{\t \in \RR^{\dim G}: \|\t\|_{\infty} < 1 \}$ that satisfies 
$\int_{\RR^{\dim G}} \chi_1(\t) \mathrm{d}\t = 1$.
Then, by setting $\chi(\t) = \delta \cdot \chi_1(\delta\t)$, we obtain a function 
$\chi: \RR^{\dim G} \to \RR_{\geq 0}$ that is supported on 
$\{\t \in \RR^{\dim G}: \|\t\|_{\infty} < \delta \}$,
satisfies $\int_{\RR^{\dim G}} \chi(\t) \mathrm{d}\t = 1$ and has furthermore the 
property that
\begin{equation}\label{eq:partial-chi}
\Big\|
 \frac{\partial}{\partial t_{i}} \chi(t_1, \dots, t_{\dim G}) 
\Big\|_{\infty} 
\ll {(m/\eps)}^{2} M^{O(1)}
\end{equation}
for $1 \leq i \leq \dim G$.
We may identify $\chi$ with a function defined on the 
vector space $\mathfrak{g}$ equipped with the basis 
$\{X_1, \dots X_{\dim G}\}$, by setting
$\chi(t_1 X_1 + \dots + t_{\dim G} X_{\dim G}) 
= \chi(t_1, \dots, t_{\dim G})$.

To obtain a smooth bump-function on $G/\Gamma$, we consider the composition 
$\chi \circ \log:G/\Gamma \to \RR$, which is supported in $\mathcal{B}_{\eps/m}$.
Since the differential $\mathrm{d}\log_{\id_{G}}: \mathfrak{g} \to \mathfrak{g}$ 
is the identity, there are positive constants $C_0,C_1$ and $c_1$, such that 
$$
C_0 \leq \int_{G/\Gamma} \chi \circ \log \leq C_1,
$$
provided $\eps< c_1$.
Hence there is a constant $C$ such that 
$\int_{G/\Gamma} \psi = 1$ for
$\psi = C\chi \circ \log$.

With this function $\psi$ at hand, let $\psi_m = \psi^{*m}$ be the $m$-th 
convolution power of $\psi$.
It is clear that for every $0<k \leq m$, the function $\psi^{*k}$ is supported in 
$\mathcal{B}_{\eps}$ and that $\int_{G/\Gamma} \psi^{*k} =1$.
Setting $\psi^{*0} = \delta_0$, where $\delta_0$ denotes the Kronecker 
$\delta$-function with weight $1$ at $0$, we furthermore have
$$
D_{X_{i_1}} \dots D_{X_{i_k}} (F*\psi_m)=
F*D_{X_{i_1}} \psi * \dots * 
D_{X_{i_k}} \psi * \psi^{*(m-k)}
$$
and, hence, 
$$
\|D_{X_{i_1}} \dots D_{X_{i_k}} (F*\psi_m)\|_{\infty} 
\leq 
\|F\|_{\infty} \cdot \|D_{X_{i_1}}(C\chi \circ \log)\|_{\infty} \cdots  
\|D_{X_{i_k}}(C\chi \circ \log)\|_{\infty}
$$
for any $k \leq m$.
Our final task is to bound 
$\|D_{X_{j}}(C\chi \circ \log)\|_{\infty}$ for every $j \leq \dim G$.
Writing $[\cdot]_i:\mathfrak{g} \to \RR$ for the $i$-th co-ordinate map with 
respect to the basis $\mathcal{X}$, we have 
\begin{equation} \label{eq:DXchi}
D_{X_{j}}(\chi \circ \log) (g)
= 
\sum_{i=1}^{\dim G} \frac{\partial\chi}{\partial X_{i}}  (\log g) \cdot
\lim_{t \to 0} \Big[\log ( \exp (tX_j) g)\Big]_i. 
\end{equation}
Since the differential $\mathrm{d}\log_{\id_{G}}: \mathfrak{g} \to \mathfrak{g}$ 
is the identity, there are constants $C_1>0$ and $c_1 > 0$, such that for every 
$g \in \mathcal{B}_{c_1}$ and for $1 \leq i \leq m$, the derivative
$$\Big|\lim_{t \to 0} \Big[\log ( \exp (tX_j) g)\Big]_i\Big|$$
is bounded by $C_1$.
Choosing $c < \min (c_0, c_1, c_2)$, the bound \eqref{eq:smoothing-Sob} now 
follows from \eqref{eq:DXchi} and the bounds given in \eqref{eq:partial-chi}.
\end{proof}

\end{document}